\documentclass{preprint}

\usepackage[full]{textcomp}
\usepackage[osf]{newtxtext}
\usepackage[cal=boondoxo]{mathalfa}
\usepackage{colortbl} 
\usepackage{amsmath}
\usepackage{mhequ}
\usepackage{longtable}
\usepackage{booktabs}
\usepackage{centernot}
\usepackage{amsthm}
\usepackage{amssymb}
\usepackage{mathrsfs}
\usepackage{enumerate}
\usepackage{mathabx}  
\usepackage{marginnote}
\usepackage{scalerel}
\usepackage{tikz}
\usepackage{tikz-cd}
\usepackage{graphicx}
\usepackage{caption}
\usepackage{subcaption}

\usepackage{pgfplots}
\pgfplotsset{compat=1.14}
\usepgfplotslibrary{fillbetween}

\numberwithin{equation}{section}

\newtheorem{thm}{Theorem}[section]

\newtheorem{lem}[thm]{Lemma}
\newtheorem{prop}[thm]{Proposition}
\newtheorem{cor}[thm]{Corollary}

\newtheorem{rem}[thm]{Remark}

\newcommand{\EE}{\mathbb{E}}

\newcommand{\HH}{\mathbb{H}}

\newcommand{\NN}{\mathbb{N}}
\newcommand{\PP}{\mathbb{P}}
\newcommand{\QQ}{\mathbb{Q}}
\newcommand{\RR}{\mathbb{R}}

\newcommand{\TT}{\mathbb{T}}
\newcommand{\ZZ}{\mathbb{Z}}
\newcommand{\BB}{\mathbb{B}}

\newcommand{\UU}{\mathbb{U}}
\newcommand{\TTN}{\mathbb{T}_N}

\newcommand{\BBN}{{\BB_N}}

\newcommand{\cA}{\mathcal A}
\newcommand{\cB}{\mathcal B}
\newcommand{\cC}{\mathcal C}
\newcommand{\cD}{\mathcal D}
\newcommand{\cE}{\mathcal E}
\newcommand{\cF}{\mathcal F}

\newcommand{\cH}{\mathcal H}

\newcommand{\cJ}{\mathcal J}

\newcommand{\cM}{\mathcal M}
\newcommand{\cN}{\mathcal N}

\newcommand{\cP}{\mathcal P}

\newcommand{\cR}{\mathcal R}

\newcommand{\cU}{\mathcal U}
\newcommand{\cV}{\mathcal V}

\newcommand{\cX}{\mathcal X}
\newcommand{\cY}{\mathcal Y}

\newcommand{\sZ}{\mathscr{Z}}

\def\fB{\mathfrak{B}}

\def\fm{\mathfrak{m}}

\def\rB{\mathrm{B}}
\def\rExt{\mathrm{Ext}}
\def\rInt{\mathrm{Int}}
\def\rnn{\mathrm{nn}}
\def\rext{\mathrm{ext}}

\def\bg{\mathbf{g}}

\newcommand{\nubn}{{\nu_{\beta,N}}}
\newcommand\vphi{\vec{\phi}}

\newcommand{\intx}{\int_{\TTN}}
\newcommand{\intt}{\int_0^K}

\newcommand{\hfbeta}{{\sqrt \beta}}
\newcommand{\dbar}{d\hspace*{-0.08em}\bar{}\hspace*{0.1em}}

\makeatletter
\newsavebox{\@brx}
\newcommand{\llangle}[1][]{\savebox{\@brx}{\(\m@th{#1\langle}\)}%
  \mathopen{\copy\@brx\kern-0.5\wd\@brx\usebox{\@brx}}}
\newcommand{\rrangle}[1][]{\savebox{\@brx}{\(\m@th{#1\rangle}\)}%
  \mathclose{\copy\@brx\kern-0.5\wd\@brx\usebox{\@brx}}}
\makeatother

\colorlet{symbols}{black}
\colorlet{testcolor}{green!60!black}

\def\1{\mathbf{{1}}}

\usetikzlibrary{shapes.misc}
\usetikzlibrary{shapes.symbols}
\usetikzlibrary{snakes}
\usetikzlibrary{decorations}
\usetikzlibrary{decorations.markings}
\definecolor{dblue}{rgb}{0.1, 0.1, 0.9}

\tikzset{
	root/.style={circle,fill=testcolor,inner sep=0pt, minimum size=2mm},		
	dot/.style={circle,fill=black,draw=black, solid,inner sep=0pt,minimum size=0.75mm},
	bdot/.style={circle,fill=blue,draw=dblue, solid,inner sep=0pt,minimum size=0.75mm},
		}
\colorlet{symbols}{blue!90!black}

\makeatletter
\def\DeclareSymbol#1#2#3{\expandafter\gdef\csname MH@symb@#1\endcsname{\tikz[baseline=#2,scale=0.15]{#3}}%
\expandafter\gdef\csname MH@symb@#1s\endcsname{\scalebox{0.6}{\tikz[baseline=#2,scale=0.15]{#3}}}}
\def\<#1>{\csname MH@symb@#1\endcsname}
\makeatother

\DeclareSymbol{sunset}{-3}{\draw (0,0) circle [x radius = 1.5, y radius = 1]; \draw (-1.5,0) node[dot] {} -- (1.5,0) node[dot] {};}
\DeclareSymbol{tadpole}{-3}{\draw (0,0) circle [x radius = 0.75, y radius = 1]; \draw (0,-1) node[dot] {};}

\DeclareSymbol{1}{0}{\draw[white] (-.4,0) -- (.4,0); \draw (0,0)  -- (0,1.2) node[dot] {};}
\DeclareSymbol{2}{0}{\draw (-0.5,1.2) node[dot] {} -- (0,0) -- (0.5,1.2) node[dot] {};}
\DeclareSymbol{3}{0}{\draw (0,0) -- (0,1.2) node[dot] {}; \draw (-.7,1) node[dot] {} -- (0,0) -- (.7,1) node[dot] {};}
\DeclareSymbol{4}{0}{\draw (-0.36,1.2) node[dot] {} -- (0,0) -- (0.36,1.2) node[dot] {}; \draw (-1,1) node[dot] {} -- (0,0) -- (1,1) node[dot] {};}
\DeclareSymbol{30}{-3}{\draw (0,0) -- (0,-1); \draw (0,0) -- (0,1.2) node[dot] {}; \draw (-.7,1) node[dot] {} -- (0,0) -- (.7,1) node[dot] {};}
\DeclareSymbol{31}{-3}{\draw (0,0) -- (0,-1) -- (1,0) node[dot] {}; \draw (0,0) -- (0,1.2) node[dot] {}; \draw (-.7,1) node[dot] {} -- (0,0) -- (.7,1) node[dot] {};}
\DeclareSymbol{32}{-3}{\draw (0,0) -- (0,-1) -- (1,0) node[dot] {}; \draw (0,0) -- (0,-1) -- (-1,0) node[dot] {}; \draw (0,0) -- (0,1.2) node[dot] {}; \draw (-.7,1) node[dot] {} -- (0,0) -- (.7,1) node[dot] {};}
\DeclareSymbol{20}{-3}{\draw (0,0) -- (0,-1);\draw (-.7,1) node[dot] {} -- (0,0) -- (.7,1) node[dot] {};}
\DeclareSymbol{22}{-3}{\draw (0,0.3) -- (0,-1) -- (1,0) node[dot] {}; \draw (0,0.3) -- (0,-1) -- (-1,0) node[dot] {};\draw (-.7,1) node[dot] {} -- (0,0.3) -- (.7,1) node[dot] {};}
\DeclareSymbol{202}{-3}{\draw (-.6625,0) -- (0,-1) -- (.6625,0)  {}; \draw (-1.1,1) node[dot] {} -- (-.6625,0) -- (-0.365,1) node[dot] {}; \draw (0.365,1) node[dot] {} -- (.6625,0) -- (1.1,1) node[dot] {};}

\DeclareSymbol{1b}{0}{\draw[white] (-.4,0) -- (.4,0); \draw[color=dblue] (0,0)  -- (0,1.2) node[bdot] {};}
\DeclareSymbol{2b}{0}{\draw[color=dblue] (-0.5,1.2) node[bdot] {} -- (0,0) -- (0.5,1.2) node[bdot] {};}
\DeclareSymbol{3b}{0}{\draw[color=dblue] (0,0) -- (0,1.2) node[bdot] {}; \draw[color=dblue] (-.7,1) node[bdot] {} -- (0,0) -- (.7,1) node[bdot] {};}
\DeclareSymbol{30b}{-3}{\draw[color=dblue] (0,0) -- (0,-1); \draw[color=dblue] (0,0) -- (0,1.2) node[bdot] {}; \draw[color=dblue] (-.7,1) node[bdot] {} -- (0,0) -- (.7,1) node[bdot] {};}
\DeclareSymbol{31b}{-3}{\draw[color=dblue] (0,0) -- (0,-1) -- (1,0) node[bdot] {}; \draw[color=dblue] (0,0) -- (0,1.2) node[bdot] {}; \draw[color=dblue] (-.7,1) node[bdot] {} -- (0,0) -- (.7,1) node[bdot] {};}
\DeclareSymbol{32b}{-3}{\draw[color=dblue] (0,0) -- (0,-1) -- (1,0) node[bdot] {}; \draw[color=dblue] (0,0) -- (0,-1) -- (-1,0) node[bdot] {}; \draw[color=dblue] (0,0) -- (0,1.2) node[bdot] {}; \draw[color=dblue] (-.7,1) node[bdot] {} -- (0,0) -- (.7,1) node[bdot] {};}
\DeclareSymbol{20b}{-3}{\draw[color=dblue] (0,0) -- (0,-1);\draw[color=dblue] (-.7,1) node[bdot] {} -- (0,0) -- (.7,1) node[bdot] {};}
\DeclareSymbol{22b}{-3}{\draw[color=dblue] (0,0.3) -- (0,-1) -- (1,0) node[bdot] {}; \draw[color=dblue] (0,0.3) -- (0,-1) -- (-1,0) node[bdot] {};\draw[color=dblue] (-.7,1) node[bdot] {} -- (0,0.3) -- (.7,1) node[bdot] {};}

\newcommand{\Xib}{{\color{dblue}\Xi\color{black}}}

\newcommand{\Omegab}{{\color{dblue}\Omega\color{black}}}
\newcommand{\PPb}{{\color{dblue}\PP\color{black}}}
\newcommand{\EEb}{{\color{dblue}\EE\color{black}}}

\DeclareSymbol{box}{0.65}{\draw[semithick] (-0.7,0) -- (0.7,0) --  (0.7,1.4) -- (-0.7,1.4) -- (-0.7,0);}
\DeclareSymbol{sbox}{0.5}{\draw[semithick] (-0.5,0) -- (0.5,0) -- (0.5,1) -- (-0.5,1) -- (-0.5,0);}
\DeclareSymbol{blackbox}{0.65}{\draw[semithick, fill=black] (-0.7,0) -- (0.7,0) --  (0.7,1.4) -- (-0.7,1.4) -- (-0.7,0);}
\DeclareSymbol{boxinbox}{0.65}{\draw[semithick, fill=black] (-0.7,0) -- (0.7,0) --  (0.7,1.4) -- (-0.7,1.4) -- (-0.7,0); \draw[semithick, fill=white] (-0.7,0) -- (0,0) -- (0,0.7) -- (-0.7,0.7) -- (-0.7,0) ; }
\DeclareSymbol{sboxinbox}{0.5}{\draw[semithick, fill=black] (-0.5,0) -- (0.5,0) -- (0.5,1) -- (-0.5,1) -- (-0.5,0);\draw[semithick, fill=white] (-0.5,0) -- (0,0) -- (0,0.5) -- (-0.5,0.5) -- (-0.5,0) ; }
\renewcommand{\Box}{{\<box>}}
\newcommand{\sBox}{{\<sbox>}}
\newcommand{\blackBox}{{\<blackbox>}}
\newcommand{\BoxBox}{{\<boxinbox>}}
\newcommand{\sBoxBox}{{\<sboxinbox>}}

\newcommand{\pe}{\mathbin{\scaleobj{0.7}{\tikz \draw (0,0) node[shape=circle,draw,inner sep=0pt,minimum size=8.5pt] {\footnotesize $=$};}}}

\newcommand{\pl}{\mathbin{\scaleobj{0.7}{\tikz \draw (0,0) node[shape=circle,draw,inner sep=0pt,minimum size=8.5pt] {\footnotesize $<$};}}}
\newcommand{\pg}{\mathbin{\scaleobj{0.7}{\tikz \draw (0,0) node[shape=circle,draw,inner sep=0pt,minimum size=8.5pt] {\footnotesize $>$};}}}

\definecolor{darkred}{rgb}{0.9,0.1,0.1}
\definecolor{darkergreen}{rgb}{0.0, 0.5, 0.0}
\def\comment#1{\marginpar{\raggedright\scriptsize{\textcolor{darkred}{#1}}}}

\begin{document}

\title{Phase transitions for $\phi^4_3$}
\author{\normalsize{Ajay Chandra$^1$, Trishen S. Gunaratnam$^2$, and Hendrik Weber$^2$}}
\institute{Imperial College London \\ \email{a.chandra@imperial.ac.uk} \and University of Bath \\
\email{t.gunaratnam@bath.ac.uk, h.weber@bath.ac.uk}}
\maketitle

\begin{abstract}
We establish a surface order large deviation estimate for the magnetisation of low temperature $\phi^4_3$. As a byproduct, we obtain a decay of spectral gap for its Glauber dynamics given by the $\phi^4_3$ singular stochastic PDE. Our main technical contributions are contour bounds for $\phi^4_3$, which extends 2D results by Glimm, Jaffe, and Spencer \cite{GJS75}. We adapt an argument by Bodineau, Velenik, and Ioffe \cite{BIV00} to use these contour bounds to study phase segregation. The main challenge to obtain the contour bounds is to handle the ultraviolet divergences of $\phi^4_3$ whilst preserving the structure of the low temperature potential. To do this, we build on the variational approach to ultraviolet stability for $\phi^4_3$ developed recently by Barashkov and Gubinelli \cite{BG19}. 
\end{abstract}

\tableofcontents

\section{Introduction}

We study the behaviour of the average 
 magnetisation
\begin{equs}
\fm_N(\phi)
=
\frac 1{N^3} \intx \phi(x) dx	
\end{equs}
for fields $\phi$ distributed according to the measure $\nubn$ with formal density
\begin{equs}\label{def: formal nubn}
d\nubn(\phi)
\propto
\exp \Big( - \intx \cV_\beta(\phi(x)) + \frac 12 |\nabla \phi(x)|^2 dx \Big)\prod_{x\in\TTN}d\phi(x)	
\end{equs}
in the infinite volume limit $N \rightarrow \infty$. Above, $\TTN = (\RR/N\ZZ)^3$ is the 3D torus of sidelength $N \in \NN$, $\prod_{x\in\TTN}d\phi(x)$ is the (non-existent) Lebesgue measure on fields $\phi:\TTN \rightarrow \RR$, $\beta > 0$ is the inverse temperature, and $\cV_\beta : \RR \rightarrow \RR$ is the symmetric double-well potential given by $\cV_\beta(a) = \frac 1\beta (a^2-\beta)^2$ for $a \in \RR$.

$\nubn$ is a finite volume approximation of a $\phi^4_3$ Euclidean quantum field theory \cite{G68, GJ73, FO76}. Its construction, first in finite volumes and later in infinite volume, was a major achievement of the constructive field theory programme in the '60s-'70s: Glimm and Jaffe made the first breakthrough in \cite{GJ73} and many results followed \cite{F74, MS77, BCGNOPS80, BFS83, BDH95, MW17, GH18, BG19}. The model in 2D was constructed earlier by Nelson \cite{N66}. In higher dimensions there are triviality results: in dimensions $\geq 5$ these are due to Aizenman and Fr\"ohlich \cite{A82, F82}, whereas the 4D case was only recently done by Aizenman and Duminil-Copin \cite{ADC20}. By now it is also well-known that the $\phi^4_3$ model has significance in statistical mechanics since it arises as a continuum limit of Ising-type models near criticality \cite{BS73, CMP95, HI18}. 

It is natural to define $\nubn$ using a density with respect to the centred Gaussian measure $\mu_N$ with covariance $(-\Delta)^{-1}$, where $\Delta$ is the Laplacian on $\TTN$ (see Remark \ref{rem: eta new} for how we deal with the issue of constant fields/the zeroeth Fourier mode). However, in 2D and higher $\mu_N$ is not supported on a space of functions and samples need to be interpreted as Schwartz distributions. This is a serious problem because there is no canonical interpretation of products of distributions, meaning that the nonlinearity $\intx \cV_\beta(\phi(x)) dx$ is not well-defined on the support of $\mu_N$. If one introduces an ultraviolet (small-scale) cutoff $K>0$ on the field to regularise it, then one sees that the nonlinearities $\cV_\beta(\phi_K)$ fail to converge as the cutoff is removed - there are divergences. The strength of these divergences grow as the dimension grows: they are only logarithmic in the cutoff in 2D, whereas they are polynomial in the cutoff in 3D. In addition, $\nubn$ and $\mu_N$ are mutually singular \cite{BG20} in 3D, which produces technical difficulties that are not present in 2D.

Renormalisation is required in order to kill these divergences. This is done by looking at the cutoff measures and subtracting the corresponding counter-term $\intx \delta m^2(K) \phi^2_K$ where $\phi_K$ is the field cutoff at spatial scales less than $\frac 1K$ and the renormalisation constant $\delta m^2(K) = \frac{C_1}{\beta} K - \frac{C_2}{\beta^2} \log K$ for specific constants $C_1, C_2 > 0$ (see Section \ref{sec: model}). If these constants are appropriately chosen (i.e. by perturbation theory), then a non-Gaussian limiting measure is obtained as $K \rightarrow \infty$. This construction yields a one-parameter family of measures $\nubn=\nubn(\delta m^2)$ corresponding to bounded shifts of $\delta m^2(K)$. 

\begin{rem}\label{rem: eta new}
For technical reasons, we work with a massive Gaussian free field as our reference measure. We do this by introducing a mass $\eta > 0$ into the covariance. This resolves the issue of the constant fields/zeroeth Fourier mode degeneracy. In order to stay consistent with \eqref{def: formal nubn}, we subtract $\intx \frac \eta 2 \phi^2 dx$ from $\cV_\beta(\phi)$.

Once we have chosen $\eta$, it is convenient to fix $\delta m^2$ by writing the renormalisation constants in terms of expectations with respect to $\mu_N(\eta)$. The particular choice of $\eta$ is inessential since one can show that changing $\eta$ corresponds to a bounded shift of $\delta m^2$ that is $O\Big(\frac 1\beta\Big)$ as $\beta \rightarrow \infty$. 
\end{rem}

The large-scale behaviour of $\nubn$ depends heavily on $\beta$ as $N \rightarrow \infty$. To see why, note that $a \mapsto \cV_\beta(a)$ has minima at $a = \pm \hfbeta$ with a potential barrier at $a=0$ of height $\beta$, so the minima become widely separated by a steep barrier as $\beta \rightarrow \infty$. Consequently, $\nubn$ resembles an Ising model on $\TTN$ with spins at $\pm \hfbeta$ (i.e. at inverse temperature $\beta > 0$) for large $\beta$. Glimm, Jaffe, and Spencer \cite{GJS75} exploited this similarity and proved phase transition for $\nu_\beta$, the infinite volume analogue of $\nubn$, in 2D using a sophisticated modification of the classical Peierls' argument for the low temperature Ising model \cite{P36, G64, D65}. See also \cite{GJS76-4, GJS76-3}. Their proof relies on contour bounds for $\nubn$ in 2D that hold in the limit $N \rightarrow \infty$. Their techniques fail in the significantly harder case of 3D. However, phase transition for $\nu_\beta$ in 3D was established by Fr\"ohlich, Simon, and Spencer \cite{FSS78} using a different argument based heavily on reflection positivity. Whilst this argument is more general (it applies, for example, to some models with continuous symmetry), it is less quantitative than the Peierls' theory of \cite{GJS75}. Specifically, it is not clear how to use it to control large deviations of the (finite volume) average magnetisation $\fm_N$. 

Although phase coexistence for $\nu_\beta$ has been established, little is known of this regime in comparison to the low temperature Ising model. In the latter model, the study of \textit{phase segregation} at low temperatures in large but finite volumes was initiated by Minlos and Sinai \cite{MS67,MS68}, culminating in the famous Wulff constructions: due to Dobrushin, Koteck\'y, and Shlosman in 2D \cite{DKS89,DKS92}, with simplifications due to Pfister \cite{P91} and results up to the critical point by Ioffe and Schonmann \cite{IS98}; and Bodineau \cite{B99} in 3D, see also results up to the critical point by Cerf and Pisztora \cite{CP00} and the bibliographical review in \cite[Section 1.3.4]{BIV00}. We are interested in a weaker form of phase segregation: \textit{surface} order large deviation estimates for the average magnetisation $\fm_N$. For the Ising model, this was first established in 2D by Schonmann \cite{S87} and later extended up to the critical point by Chayes, Chayes, and Schonmann \cite{CCS87}; in 3D this was first established by Pisztora \cite{P96}. These results should be contrasted with the volume order large deviations established for $\fm_N$ in the high temperature regime where there is no phase coexistence \cite{C86,E85, FO88, O88}.

Our main result is a surface order upper bound on large deviations for the average magnetisation under $\nubn$.
\begin{thm}\label{thm: ld}
	Let $\eta > 0$ and $\nubn = \nubn(\eta)$ as in Remark \ref{rem: eta new}. For any $\zeta\in (0,1)$, there exists $\beta_0 = \beta_0(\zeta,\eta) > 0$, $C=C(\zeta, \eta)>0$, and $N_0 = N_0(\zeta) \geq 4$ such that the following estimate holds: for any $\beta > \beta_0$ and any $N > N_0$ dyadic,
	\begin{equs} \label{eq: ld estimate}
	\frac{1}{N^2} \log \nu_{\beta,N} \Big( \fm_N \in (-\zeta \hfbeta, \zeta \hfbeta) \Big) \leq - C \hfbeta.
	\end{equs} 
\end{thm}

\begin{proof}
See Section \ref{subsec: proof of thm ld}.	
\end{proof}

The condition that $N$ is a sufficiently large dyadic in Theorem \ref{thm: ld} comes from Proposition \ref{prop:discretetight} (we also need that $N$ is divisible by $4$ to apply the chessboard estimates of Proposition \ref{prop:chessboard_estimates}). Our analysis can be simplified to prove Theorem \ref{thm: ld} in 2D with $N^2$ replaced by $N$ in \eqref{eq: ld estimate}.

Our main technical contributions are contour bounds for $\nubn$. As a result, the Peierls' argument of \cite{GJS75} is extended to 3D, thereby giving a second proof of phase transition for $\phi^4_3$. The main difficulty is to handle the ultraviolet divergences of $\nubn$ whilst preserving the structure of the low temperature potential. We do this by building on the variational approach to showing ultraviolet stability for $\phi^4_3$ recently developed by Barashkov and Gubinelli \cite{BG19}. Our insight is to separate scales within the corresponding stochastic control problem through a coarse-graining into an effective Hamiltonian and remainder. The effective Hamiltonian captures the macroscopic description of the system and is treated using techniques adapted from \cite{GJS76-3}. The remainder contains the ultraviolet divergences and these are killed using the renormalisation techniques of \cite{BG19}.

Our next contribution is to adapt arguments used by Bodineau, Velenik, and Ioffe \cite{BIV00}, in the context of equilibrium crystal shapes of discrete spin models, to study phase segregation for $\phi^4_3$. In particular, we adapt them to handle a block-averaged model with unbounded spins. Technically, this requires control over \textit{large fields}.

\subsection{Application to the dynamical $\phi^4_3$ model}

The Glauber dynamics of $\nubn$ is given by the singular stochastic PDE
\begin{equs}\label{def: dynamical phi4}
\begin{split}
	(\partial_t - \Delta + \eta)\Phi 
	&= 
	-\frac{4}{\beta}\Phi^3 + (4+\eta +\infty)\Phi + \sqrt 2\xi
	\\
	\Phi(0,\cdot)
	&=
	\phi_0
\end{split}
\end{equs}
where $\Phi \in S'(\RR_+\times\TTN)$ is a space-time Schwartz distribution, $\phi_0 \in \cC^{-\frac 12 -\kappa}(\TTN)$, the infinite constant indicates renormalisation (see Remark \ref{rem: spde infinity}), and $\xi$ is space-time white noise on $\TTN$. The well-posedness of this equation, known as the dynamical $\phi^4_3$ model, has been a major breakthrough in stochastic analysis in recent years \cite{H14, H16, GIP15, CC18, K16, MW17, GH19, MW18}. 

In finite volumes the solution is a Markov process and its associated semigroup $(\cP_t^{\beta,N})_{t \geq 0}$ is reversible and exponentially ergodic with respect to its unique invariant measure $\nubn$ \cite{HMb18,HS19, ZZb18}. As a consequence, there exists a spectral gap $\lambda_{\beta,N}>0$ given by the optimal constant in the inequality:
	\begin{equs}
	\Big\langle \Big(\cP_t^{\beta,N} F \Big)^2 \rangle_{\beta,N} - \Big( \Big\langle \cP_t^{\beta,N} F \Big\rangle_{\beta,N} \Big)^2
	\leq
	e^{-\lambda_{\beta,N}t} \Big(\langle F^2 \rangle_{\beta,N} - \langle F \rangle_{\beta,N}^2 \Big)	
	\end{equs}
	for suitable $F \in L^2(\nubn)$. $\lambda_{\beta,N}^{-1}$ is called the relaxation time and measures the rate of convergence of variances to equilibrium. An implication of Theorem \ref{thm: ld} is the exponential explosion of relaxation times in the infinite volume limit provided $\beta$ is sufficiently large. 
\begin{cor}\label{cor: sg}
	Let $\eta > 0$ and $\nubn = \nubn(\eta)$ as in Remark \ref{rem: eta new}. Then, there exists $\beta_0=\beta_0(\eta)>0$, $C=C(\beta_0,\eta)$, and $N_0 \geq 4$ such that, for any $\beta > \beta_0$ and $N > N_0$ dyadic,
	\begin{equs}\label{eq: cor}
	\frac{1}{N^2} \log \lambda_{\beta,N} \leq - C\hfbeta.
	\end{equs}
\end{cor}

\begin{proof}
See Section \ref{sec: decay of gap}.	
\end{proof}

Corollary \ref{cor: sg} is the first step towards establishing phase transition for the relaxation times of the Glauber dynamics of $\phi^4$ in 2D and 3D. This phenomenon has been well-studied for the Glauber dynamics of the 2D Ising model, where a relatively complete picture has been established (in higher dimensions it is less complete). The relaxation times for the Ising dynamics on the 2D torus of sidelength $N$ undergo the following trichotomy as $N \rightarrow \infty$: in the high temperature regime, they are uniformly bounded in $N$ \cite{AH87,MO94}; in the low temperature regime, they are exponential in $N$ \cite{S87, CCS87,T89,MO94,CGMS96}; at criticality, they are polynomial in $N$ \cite{H91, LS12}. It would be interesting to see whether the relaxation times for the dynamical $\phi^4$ model undergo such a trichotomy. 

\subsection{Paper organisation}

In Section \ref{sec: model} we introduce the renormalised, ultraviolet cutoff measures $\nu_{\beta,N,K}$ that converge weakly to $\nubn$ as the cutoff is removed. In Section \ref{sec: surface order} we carry out the statistical mechanics part of the proof of Theorem \ref{thm: ld}. In particular, conditional on the moment bounds in Proposition \ref{prop: cosh}, we develop contour bounds for $\nubn$. These contour bounds allow us to adapt techniques in \cite{BIV00}, which were developed in the context of discrete spin systems, to deal with $\nubn$.

 In Section \ref{sec: bd} we lay the foundation to proving Proposition \ref{prop: cosh} by introducing the Bou\'e-Dupuis formalism for analysing the free energy of $\nubn$ as in \cite{BG19}. We then use a low temperature expansion and coarse-graining argument within the Bou\'e-Dupuis formalism in Section \ref{sec: free energy} to establish Proposition \ref{prop: q bound main} which contains the key analytic input to proving Proposition \ref{prop: cosh}. 
 
 In Section \ref{sec: chessboard estimates}, we use the chessboard estimates of Proposition \ref{prop:chessboard_estimates} to upgrade the bounds of Proposition \ref{prop: q bound main} to those of Proposition \ref{prop: cosh}. Chessboard estimates follow from the well-known fact that $\nubn$ is reflection positive. We give an independent proof of this fact by using stability results for the dynamics \eqref{def: dynamical phi4} to show that lattice and Fourier regularisations of $\nubn$ converge to the same limit. Then, in Section \ref{sec: decay of gap}, we prove Corollary \ref{cor: sg} showing that the spectral gaps for the dynamics decay in the infinite volume limit provided $\beta$ is sufficiently large.
 
 We collect basic notations and analytic tools that we use throughout the paper in Appendix \ref{appendix: toolbox}. 

\subsection*{Acknowledgements}
We thank Roman Koteck\'y for inspiring discussions throughout all stages of this project. We thank Nikolay Barashkov for useful discussions regarding the variational approach to ultraviolet stability for $\phi^4_3$. We thank Martin Hairer for a particularly useful discussion. AC and TSG thank the Hausdorff Research Institute for Mathematics for the hospitality and support during the Fall 2019 junior trimester programme \textit{Randomness, PDEs and Nonlinear Fluctuations}. AC, TSG, and HW thank the Isaac Newton Institute for Mathematical Sciences for hospitality and support during the Fall 2018 programme \textit{Scaling limits, rough paths, quantum field theory}, which was supported by EPSRC Grant No. EP/R014604/1. AC was supported by the Leverhulme Trust via an Early Career Fellowship, ECF-2017-226. TSG was supported by EPSRC as part of the Statistical Applied Mathematics CDT at the University of Bath (SAMBa), Grant No. EP/L015684/1. HW was supported by the Royal Society through the University Research Fellowship UF140187 and by the Leverhulme Trust through a Philip Leverhulme Prize.

\section{The model} \label{sec: model}
In the following, we use notation and standard tools introduced in Appendix \ref{appendix: sub: basic}.

Let $\eta > 0$. Denote by $\mu_N = \mu_N(\eta)$ the centred Gaussian measure with covariance $(-\Delta + \eta)^{-1}$ and expectation $\EE_N$. Above, $\Delta$ is the Laplacian on $\TTN$. As pointed out in Remark \ref{rem: eta new}, the choice of $\eta$ is inessential. We consider it fixed unless stated otherwise and we do not make $\eta$-dependence explicit in the notation.

Fix $\beta > 0$. Let $\cV_\beta\colon\RR\rightarrow\RR$ be given by
\begin{equs}
\cV_\beta(a) 
= 
\frac{1}{\beta}(a^2 - \beta)^2
=
\frac 1\beta a^4 - 2a^2 + \beta.
\end{equs}
$\cV_\beta$ is a symmetric double well potential with minima at $a = \pm \hfbeta$ and a potential barrier at $a=0$ of height $\beta$. 

Fix $\rho \in C^\infty_c(\RR^3;[0,1])$ rotationally symmetric; decreasing; and satisfying $\rho(x)=1$ for $|x| \in [0,c_\rho)$, where $c_\rho >0$. See Lemma \ref{lem:trident} for why the last condition is important. Note that many of our estimates rely on the choice of $\rho$, but we omit explicit reference to this.

For every $K>0$, let $\rho_K$ be the Fourier multiplier on $\TTN$ with symbol $\rho_K(\cdot) = \rho(\frac{\cdot}{K})$. For $\phi \sim \mu_N$, we denote $\phi_K = \rho_K \phi$. Note that $\phi_K$ is smooth. Let
\begin{equs} \label{def: tadpole}
\<tadpole>_K
=
\EE_N[\phi_K^2(0)]
=
\frac 1{N^3} \sum_{n \in (N^{-1}\ZZ)^3}	\frac{\rho_K^2(n)}{\langle n \rangle^2}
\end{equs}
where $\langle \cdot \rangle = \sqrt{\eta + 4\pi^2|\cdot|}$. Note that $\<tadpole>_K = O(K)$ as $K \rightarrow \infty$. The first four Wick powers of $\phi_K$ are given by the generalised Hermite polynomials:
\begin{equs}
	:\phi_K(x): 
	&= 
	\phi_K(x) 
	\\
	:\phi_K^2(x):
	&=
	\phi_K^2(x) - \<tadpole>_K
	\\
	:\phi_K^3(x):
	&=
	\phi_K^3(x) - 3 \<tadpole>_K \phi_K(x)
	\\
	:\phi_K^4(x):
	&=
	\phi_K^4(x) - 6 \<tadpole>_K \phi_K^2(x) + 3 \<tadpole>_K^2.
\end{equs}
We define the Wick renormalised potential by linearity:
\begin{equs}
:\cV_\beta(\phi_K): 
= 
\frac{1}{\beta} : \phi_K^4: - 2 :\phi_K^2: + \beta.
\end{equs}

Let $\nu_{\beta,N,K}$ be the probability measure with density
\begin{equs}\label{eq: cutoff density}
d\nu_{\beta,N,K}(\phi)
=
\frac{e^{-\cH_{\beta,N,K}(\phi_K)}}{\sZ_{\beta,N,K}}	d\mu_N(\phi).
\end{equs}
Above, $\cH_{\beta,N,K}$ is the renormalised Hamiltonian 
\begin{equs} \label{eq: renorm hamiltonian}
\cH_{\beta,N,K}(\phi_K)
&=
\intx :\cV_\beta(\phi_K):  - \frac{\gamma_K}{\beta^2} :\phi_K^2: - \delta_K - \frac \eta 2 :\phi_K^2: dx
\end{equs}
where $\gamma_K$ and $\delta_K$ are additional renormalisation constants given by \eqref{def: mass renorm} and \eqref{def: energy renorm}, respectively, and $\sZ_{\beta,N,K} = \EE_N e^{-\cH_{\beta,N,K}(\phi_K)}$ is the partition function.

\begin{prop}\label{prop: phi4 existence}
For every $\beta > 0$ and $N \in \NN$, the measures $\nu_{\beta,N,K}$ converge weakly to a non-Gaussian measure $\nubn$ on $S'(\TTN)$ as $K \rightarrow \infty$. In addition, $\sZ_{\beta,N,K} \rightarrow \sZ_{\beta,N}$ as $K \rightarrow \infty$ and satisfies the following estimate: there exists $C=C(\beta,\eta) > 0$ such that
\begin{equs}
-CN^3 
\leq
-\log\sZ_{\beta,N}
\leq
CN^3.
\end{equs}
\end{prop}

\begin{proof}
Proposition \ref{prop: phi4 existence} is a variant of the classical ultraviolet stability for $\phi^4_3$ first established in \cite{GJ73}. Our precise formulation, i.e. the choice of $\gamma_\bullet$ and $\delta_\bullet$, is taken from \cite[Theorem 1]{BG19}.  
\end{proof}

We write $\langle \cdot \rangle_{\beta,N}$ and $\langle \cdot \rangle_{\beta,N,K}$ for expectations with respect to $\nubn$ and $\nu_{\beta,N,K}$, respectively.

\begin{rem}\label{rem: renorm 0}
The constants $\<tadpole>_K, \gamma_K, \delta_K$ are, respectively, Wick renormalisation, (second order) mass renormalisation, and energy renormalisation constants. They all depend on $\eta$ and $N$. $\delta_K$ additionally depends on $\beta$ and is needed for the convergence of $\sZ_{\beta,N,K}$ as $K \rightarrow \infty$, but drops out of the definition of the cutoff measures \eqref{eq: cutoff density}. 
\end{rem}

\begin{rem}
In 2D a scaling argument \cite{GJS76} allows one to work with the measure with density proportional to
\begin{equs}
	\exp \Big( - \intx :\cV_\beta(\phi_K): dx \Big) d\tilde\mu_N(\phi)
\end{equs}
where $\tilde\mu_N$ is the Gaussian measure with covariance $(-\Delta + \hfbeta^{-1})^{-1}$, i.e. a $\beta$-dependent mass. This measure is significantly easier to work with due to the degenerate mass when $\beta$ is large. In particular, it is easier to obtain contour bounds which, although suboptimal from the point of view of $\beta$-dependence, are sufficient for the Peierls' argument in \cite{GJS75} and for the analogue of our argument in Section \ref{sec: surface order} carried out in 2D. In 3D one cannot work with such a measure.
\end{rem}

\section{Surface order large deviation estimate} \label{sec: surface order}

In this section we carry out the statistical mechanics part of the proof of Theorem \ref{thm: ld}. Recall that for large $\beta$, the the minima of potential $\cV_\beta$ at $\pm \hfbeta$ are widely separated by a steep potential barrier of height $\beta$, so formally $\nubn$ resembles an Ising model at inverse temperature $\beta$. We use this intuition to prove contour bounds for $\nubn$ (see Proposition \ref{prop:peierls}) conditional on certain moment bounds (see Proposition \ref{prop: cosh}). The contour bounds are then used to adapt arguments from \cite{BIV00} to prove Theorem \ref{thm: ld}.

\subsection{Block averaging} \label{subsec: block average}
Let $e_1, e_2, e_3$ be the standard basis for $\RR^3$. We identify $\TTN$ with the set
\begin{equs}
\big\{ a_1 e_1 + a_2 e_2 + a_3 e_3 : a_1, a_2, a_3 \in [0,N) \big\}.
\end{equs}
Define
\begin{equs}
\BBN
=
\Big\{ \prod_{i=1}^3 [a_i,a_i+1) \subset \TTN : a_1,a_2,a_3 \in \{0,\dots,N-1\} \Big\}.	
\end{equs}
We call elements of $\BBN$ blocks. For any $B \subset \BBN$, we overload notation and write $B=\bigcup_{\sBox \in B} \Box \subset \TTN$. Hence, $|B| = \int_B 1 dx$ is the number of blocks in $B$. In addition, we identify any $\vec f \in \RR^\BBN$ with the piecewise continuous function on $\TTN$ given by $\vec f(x) = \vec f(\Box)$ for $x \in \Box$.

Let $\phi \sim \nubn$. For any $\Box \in \BBN$, let $\phi(\Box) = \int_{\sBox} \phi dx$. Here, the integral is interpreted as the duality pairing between $\phi$ (a distribution) and the indicator function $\1_\sBox$ (a test function); we use this convention throughout. We let $\vphi = (\phi(\Box))_{\sBox \in \BBN} \in \RR^\BBN$ denote the block averaged field obtained from $\phi$. 

\begin{rem}\label{rem:block_av}
Testing $\phi$ against $\1_\sBox$, which is not smooth, yields a well-defined random variable on the support of $\nubn$. Indeed, $\phi$ belongs almost surely to $L^\infty$-based Besov spaces of regularity $s$ for every $s < -\frac 12$ (see Appendix \ref{appendix: sub: besov} for a review of Besov spaces and see Section \ref{sec: bd} for the almost sure regularity of $\phi$). On the other hand, indicator functions of blocks belong to $L^1$-based Besov spaces of regularity $s$ for every $s < 1$ or, more generally, $L^p$-based Besov spaces of regularity $s$ for every $s < \frac 1p$ (see, for example, Lemma 1.1 in \cite{FR12}). This is sufficient to test $\phi$ against indicator functions of blocks (using e.g. Proposition \ref{prop: tool duality}). We also give an alternative proof using a type of It\^o isometry in Proposition \ref{prop: testing phi4}.
\end{rem}

\subsection{Phase labels}

We define a map $\vphi \in \RR^\BBN \mapsto \sigma \in \{-\hfbeta, 0, \hfbeta \}^\BBN$ called a phase label. A basic function of $\sigma$ is to identify whether the averages $\phi(\Box)$ take values around the well at $+\hfbeta$, the well at $-\hfbeta$, or neither. We quantify this to a given precision $\delta \in (0,1)$, which is taken to be fixed in what follows.

\begin{itemize}
\item We say that $\Box \in \BBN$ is plus (resp. minus) valued if
\begin{equs}
	|\phi(\Box) \mp \hfbeta | 
	< 
	\hfbeta \delta.
\end{equs}
The set of plus (resp. minus) valued blocks is denoted $\cP$ (resp. $\cM$).
\item The set of neutral blocks is defined as $\cN = \BBN \setminus (\cP \cup \cM)$.
\end{itemize}

Each block in $\BBN$ contains a midpoint. Given two distinct blocks in $\BBN$, we say that they are nearest-neighbours if their midpoints are of distance $1$. They are $*$-neighbours if either they are nearest-neighbours or if their midpoints are of distance $\sqrt 3$. For any $\Box \in \BBN$, the $*$-connected ball centred at $\Box$ is the set $\rB^*(\Box) \subset \BBN$ consisting of $\Box$ and its $*$-neighbours. It contains exactly $27$ blocks.

\begin{itemize}
\item We say that $\Box \in \BBN$ is \textit{plus good} if every $\Box' \in \rB^*(\Box)$ is plus valued. The set of plus good blocks is denoted $\cP_G$.
\item We say that $\Box \in \BBN$ is \textit{minus good} if every $\Box' \in \rB^*(\Box)$ is minus valued. The set of minus good blocks is denoted $\cM_G$.
\item The set of \textit{bad} blocks is defined as $\cB = \BBN \setminus (\cP_G \cup \cM_G)$.  
\end{itemize}

Define the phase label $\sigma$ associated to $\vphi$ of precision $\delta > 0$ by
 \begin{equs}
 \sigma(\Box) 
 = 
 \begin{cases}
 +\hfbeta, \quad \Box \in \cP_G, \\
 -\hfbeta, \quad \Box \in \cM_G, \\
 0, \hspace{9.75mm} \Box \in \cB.
 \end{cases}
 \end{equs}

The following proposition can be thought of as an extension of the contour bounds developed for $\phi^4$ in 2D \cite[Theorem 1.2]{GJS75} to 3D.
 
\begin{prop}\label{prop:peierls}
Let $\sigma$ be a phase label of precision $\delta \in (0,1)$. Then, there exists $\beta_0=\beta_0(\delta, \eta) > 0$ and $C_P=C_P(\delta, \eta) > 0$ such that, for $\beta > \beta_0$, the following holds for any $N \in 4\NN$: for any set of blocks $B\subset\BBN$,
    \begin{equs} \label{eq:peierls}
    \nu_{\beta,N}(\sigma(\Box) = 0 \text{ for all } \Box \in B) 
    \leq 
    e^{-C_P\hfbeta |B|}.	
    \end{equs}
\end{prop}

\begin{proof}
See Section \ref{subsec: proof of peierls}. The main estimates required in the proof are given in Proposition \ref{prop: cosh}, which extends \cite[Theorem 1.3]{GJS75} to 3D and improves the $\beta$-dependence. Assuming this, we then prove Proposition \ref{prop:peierls} in the spirit of the proof of \cite[Theorem 1.2]{GJS75}. 
\end{proof}

\subsection{Penalising bad blocks} \label{subsec: badblocks}
Given a phase label, we partition the set of bad blocks $\cB$ into two types.
\begin{itemize}
\item Frustrated blocks
 are blocks $\Box \in \BBN$ such that $\rB^*(\Box)$ contains a neutral block. We denote the set of frustrated blocks $\cB_F$.	
\item Interface block are blocks $\Box \in \BBN$ such that $\rB^*(\Box)$ contains no neutral blocks, but there exists at least one pair of \textit{nearest-neighbours} $\{ \Box',\Box'' \} \subset \rB^*(\Box)$ such that $\Box' \in \cP$ but $\Box'' \in \cM$. We denote the set of interface blocks $\cB_I$. 
\end{itemize}

For any $\Box \in \BBN$ and any nearest-neighbours $\Box',\Box'' \in \BBN$, define:
\begin{equs} \label{def: Q_i}
\begin{split}
Q_1(\Box)
&=
\frac 1{\hfbeta} \int_{\sBox} (\beta - :\phi^2(x):)dx 
\\
Q_2(\Box)
&=
\frac 1{\hfbeta} \int_{\sBox}(:\phi^2(x):-\phi(\Box)^2)dx 
\\
Q_3(\Box',\Box'')
&=
\phi(\Box') - \phi(\Box'').
\end{split}
\end{equs}

\begin{rem} \label{rem: testing wick square}
Note that testing $:\phi^2:$ against $\1_{\sBox}$ yields a well-defined random variable on the support of $\nubn$. We give a proof of this fact in Proposition \ref{prop: testing wick square}.
\end{rem}

We write $\rB^*_{\rnn}(\Box)$ for the set of unordered pairs of nearest-neighbour blocks $\{\Box', \Box'' \}$ in $\BBN$ such that $\Box',\Box'' \in \rB^*(\Box)$. There are $54$ elements in this set.

\begin{lem} \label{lem:badsetbounds}
Let $N \in \NN$ and fix a phase label of precision $\delta \in (0,1)$. Then, for every $\Box \in \BBN$, 
\begin{equs} \label{eq:badset1bound}
\1_{\sBox \in \cB_F} 
\leq 2
e^{-C_\delta\hfbeta} \sum_{\sBox' \in \rB^*(\sBox)}\Big(\cosh Q_1(\Box') + \cosh Q_2(\Box') \Big)
\end{equs}
\begin{equs} \label{eq:badset2bound}
\1_{\sBox \in \cB_I}
\leq
 2e^{-C_\delta \hfbeta} \sum_{\{\sBox',\sBox''\} \in \rB^{*}_{\rnn}(\sBox)}\cosh Q_3(\Box',\Box'')
\end{equs}
where $C_\delta = \min\Big( \frac \delta 2, 2-2\delta \Big) > 0$.
\end{lem}

Frustrated blocks are penalised by the potential $\cV_\beta$ whereas interface blocks are penalised by the gradient term in the Gaussian measure. Lemma \ref{lem:badsetbounds} formalises this through use of the random variables $Q_1, Q_2$ and $Q_3$, which (up to trivial modifications) were introduced in \cite{GJS75}. $Q_1$ penalises frustrated blocks. $Q_2$ is an error term coming from the fact that the potential is written in terms of $\phi$ rather than $\vphi$. $Q_3$ penalises interface blocks.

\begin{proof}[Proof of Lemma \ref{lem:badsetbounds}]
For any $\Box \in \BBN$,
\begin{equs} \label{eq:badset1bound1}
\begin{split}
\1_{\sBox \in \cN} 
&= 
\1_{|\phi(\sBox)|<(1-\delta)\hfbeta}+\1_{|\phi(\sBox)|>(1+\delta)\hfbeta}
\\
&=
\1_{\frac 1\hfbeta (\beta - \phi(\sBox)^2)> (2\delta - \delta^2)\hfbeta} + \1_{\frac 1{\hfbeta} (\phi(\sBox)^2 - \beta) > (2\delta+\delta^2)\hfbeta} 
\\
&= 
\1_{\frac 1\hfbeta \int_{\sBox}\beta - :\phi^2(x):dx + \frac{1}{\hfbeta}\int_{\sBox}:\phi^2(x):-\phi(\sBox)^2 dx > (2\delta-\delta^2)\hfbeta} 
\\
& \quad\quad+
\1_{\frac 1\hfbeta \int_{\sBox} :\phi^2(x):- \beta dx + \frac 1\hfbeta \int_{\sBox}\phi(\sBox)^2 - :\phi^2(x): dx > (2\delta + \delta^2)\hfbeta} 
\\
& \leq
\1_{\frac 1\hfbeta \int_{\sBox} \beta - :\phi^2(x): dx > \frac{2\delta - \delta^2}{2}\hfbeta}+\1_{\frac 1\hfbeta \int_{\sBox}:\phi^2(x): - \phi(\sBox)^2 dx > \frac{2\delta - \delta^2}{2}\hfbeta}
\\
& \quad\quad +
\1_{\frac 1\hfbeta \int_{\sBox} :\phi^2(x):- \beta dx > \frac{2\delta + \delta^2}{2}\hfbeta}+\1_{\frac 1\hfbeta \int_{\sBox}\phi(\sBox)^2 - :\phi^2(x): dx > \frac{2\delta + \delta^2}{2}\hfbeta}
\\
&\leq
e^{-\frac{\delta}{2}\hfbeta} \Big( e^{Q_1(\sBox)} + e^{Q_2(\sBox)} + e^{-Q_1(\sBox)}  + e^{-Q_2(\sBox)} \Big)
\\
&=
2e^{-\frac{\delta}{2}\hfbeta} \Big(\cosh Q_1(\Box) + \cosh Q_2(\Box) \Big)
\end{split}
\end{equs}
where in the penultimate line we have used that $\delta^2 \leq \delta$.

By the definition of $\cB_F$,
\begin{equs} \label{eq:badset1entropy}
\1_{\sBox \in \cB_F} \leq \sum_{\sBox' \in \rB^*(\sBox)}	\1_{\sBox' \in \cN}.
\end{equs}
Using \eqref{eq:badset1bound1} applied to $\1_{\sBox' \in \cN}$ in \eqref{eq:badset1entropy} yields \eqref{eq:badset1bound}.

\eqref{eq:badset2bound} is established by the following estimates: by the definition of $\cB_I$,
\begin{equs}
\1_{\sBox \in \cB_I}
&\leq
\sum_{\{\sBox',\sBox''\} \in \rB^*_{\rnn}(\sBox)} (\1_{\sBox' \in \cP}\1_{\sBox'' \in \cM} + \1_{\sBox' \in \cM}\1_{\sBox'' \in \cP})
\\
&\leq
\sum_{\{\sBox',\sBox''\} \in \rB^*_{\rnn}(\sBox)} (\1_{\phi(\sBox') - \phi(\sBox'') > (2-2\delta)\hfbeta} + \1_{\phi(\sBox'') - \phi(\sBox') > (2-2\delta)\hfbeta})
\\
&\leq
\sum_{\{\sBox',\sBox''\} \in \rB^*_{\rnn}(\sBox)} e^{-(2-2\delta)\hfbeta}\Big(e^{Q_3(\sBox',\sBox'')} + e^{-Q_3(\sBox',\sBox'')}\Big)
\\
&=
\sum_{\{\sBox',\sBox''\} \in \rB^*_{\rnn}(\sBox)} 2e^{-(2-2\delta)\hfbeta}\cosh Q_3(\Box',\Box'').
\end{equs}
\end{proof}

In order to use Lemma \ref{lem:badsetbounds} to prove Proposition \ref{prop:peierls}, we want to control expectations of $\cosh Q_1, \cosh Q_2$ and $\cosh Q_3$ by the exponentially small (in $\hfbeta$) prefactor in \eqref{eq:badset1bound} and \eqref{eq:badset2bound}. Moreover, we want to control these expectations over a set of blocks as opposed to just single blocks.

Let $B_1, B_2 \subset \BBN$ and let $B_3$ be any set of unordered pairs of nearest-neighbours in $\BBN$. Define
\begin{equs} \label{eq:cosh}
\begin{split}
\cosh Q_1(B_1) 
&= 
\prod_{\sBox \in B_1} \cosh Q_1(\Box)
\\
\cosh Q_2(B_2)
&=
\prod_{\sBox \in B_2} \cosh Q_2(\Box)
\\
\cosh Q_3(B_3)
&=
\prod_{\{\sBox,\sBox'\} \in B_3} \cosh Q_3(\Box,\Box').
\end{split}
\end{equs}

\begin{rem}
Although the random variable $Q_3(\Box,\Box')$ does depend on the ordering of $\Box$ and $\Box'$, $\cosh Q_3(\Box,\Box')$ does not.
\end{rem}

\begin{prop}\label{prop: cosh}
For every $a_0 > 0$, there exist $\beta_0 = \beta_0(a_0,\eta)>0$ and $C_Q = C_Q(a_0,\beta_0,\eta)>0$ such that the following holds uniformly for all $\beta > \beta_0$, $a_1,a_2,a_3 \in \RR$ such that $|a_i| \leq a_0$, and $N \in 4\NN$: let $B_1, B_2 \subset \BBN$ and $B_3$ a set of unordered pairs of nearest-neighbour blocks in $\BBN$. Then, 
\begin{equs}\label{eq: cosh estimate}
\Big \langle \prod_{i=1}^3 \cosh \big(a_i Q_i(B_i)\big) \Big\rangle_{\beta,N}
\leq
e^{C_Q(|B_1|+|B_2|+|B_3|)}	
\end{equs}
where $|B_3|$ is given by the number of pairs in $B_3$. 
\end{prop}

\begin{proof}
Proposition \ref{prop: cosh} is established in Section \ref{subsec: proof cosh prop}, but its proof takes up most of this article. The overall strategy is as follows: the crucial first step is to obtain upper and lower bounds on the free energy $-\log\sZ_{\beta,N}$ that are uniform in $\beta$ and extensive in the \textit{volume}, $N^3$. We then build on this analysis to obtain upper bounds on expectations of the form $\langle \exp Q \rangle_{\beta,N}$ that are uniform in $\beta$ and extensive in $N^3$. Here, $Q$ is a placeholder for random variables that are derived from the $Q_i$'s, but that are supported on the whole of $\TTN$ rather than arbitrary unions of blocks. This is all done in Section \ref{sec: free energy}, where the key results are Propositions \ref{prop:zbounds} and \ref{prop: q bound main}, within the framework developed in Section \ref{sec: bd}.

The next step in the proof is to use the chessboard estimates of Proposition \ref{prop:chessboard_estimates} (which requires $N \in 4\NN$) to bound the lefthand side of \eqref{eq: cosh estimate} in terms of $|B_1|+|B_2|+|B_3|$ products of expectations of the form $\langle \exp Q \rangle_{\beta,N}^\frac{1}{N^3}$. Applying the results of Section \ref{sec: free energy} then completes the proof.
\end{proof}

Key features of the estimate \eqref{eq: cosh estimate} used in the proof of Proposition \ref{prop:peierls} are that it is \textit{uniform in $\beta$} and \textit{extensive in the support of the $Q_i$'s}.

\subsubsection{Proof of the Proposition \ref{prop:peierls} assuming Proposition \ref{prop: cosh}} \label{subsec: proof of peierls}	

We first show that we can reduce to the case where $B$ contains no $*$-neighbours, which simplifies the combinatorics later on. Identify $\BBN$ with a subset of $\ZZ^3$. For every $e_l \in \{-1,0,1\}^3$, let $\ZZ^3_l = e_l + (3 \ZZ)^3$. There are 27 such sub-lattices which we order according to $l \in \{1,\dots,27\}$. Note that $\ZZ^3 = \bigcup_{l=1}^{27} \ZZ^3_l$. Let $\BB_N^l = \BBN \cap \ZZ^3_l$. Each $*$-connected ball in $\BBN$ contains at most one block from each of these $\BB_N^l$.

Assume that \eqref{eq:peierls} has been established for sets with no $*$-neighbours with constant $C_P'$. Then, by H\"older's inequality,
\begin{equs} \label{eq:sublattices}
\begin{split}
\nubn(\sigma(\Box) = 0 \text{ for all }\Box \in B) 
&=
\Big\langle \prod_{\sBox \in B} \1_{\sBox \in \cB} \Big\rangle_{\beta,N} 
\\
&\leq
\prod_{l=1}^{27} \Big\langle \prod_{\sBox \in B \cap \BB_N^l} \1_{\sBox \in \cB} \Big\rangle_{\beta,N}^\frac{1}{27}
\\
&\leq
e^{-\frac{C_P'}{27}|B|}
\end{split}
\end{equs}
thereby establishing \eqref{eq:peierls} with $C_P = \frac{C_P'}{27}$.

Now assume that $B$ contains no $*$-neighbours. Fix any $A \subset B$. Let $\rB^*(A) = \bigcup_{\sBox \in A} \rB^*(\Box)$ and let $\rB^*_{\rnn}(A) = \bigcup_{\sBox \in A} \rB^*_{\rnn}(\Box)$. By our assumption, $A$ contains no $*$-neighbours. Hence, for any $\Box' \in \rB^*(A)$ there exists a unique $\Box \in A$ such that $\Box' \in \rB^*(\Box)$; we define the root of $\Box'$ to be $\Box$. Similarly, for any $\{ \Box', \Box'' \} \in \rB^*_{\rnn}(A)$ there exists a unique $\Box \in A$ such that $\{\Box',\Box''\} \in \rB^*_{\rnn}(\Box)$; we define the root of $\{ \Box', \Box'' \}$ to be $\Box$. Note that the definition of root is $A$-dependent in both cases.

By Lemma \ref{lem:badsetbounds}, there exists $C_\delta$ such that
\begin{equs} \label{eq: peierls indicator bounds}
\begin{split}
\prod_{\sBox \in B}\1_{\sBox \in \cB}
&=
\sum_{A \subset B} \Big( \prod_{\sBox \in A} \1_{\sBox \in \cB_F} \Big) \Big( \prod_{\sBox \in B\setminus A} \1_{\sBox \in \cB_I} \Big)
\\
&\leq
2^{|B|}e^{-C_\delta \hfbeta |B|} \sum_{A \subset B}	\Big( \prod_{\sBox \in A} \sum_{\sBox' \in \rB^*(\sBox)} \big( \cosh Q_1(\Box') + \cosh Q_2(\Box') \big) \Big) 
\\
&\quad\quad\quad
\times \Big( \prod_{\sBox \in B\setminus A} \sum_{\{\sBox',\sBox''\} \in \rB^*_{\rnn}(\sBox)} \cosh Q_3(\Box',\Box'') \Big)
\\
&=
2^{|B|}e^{-C_\delta \hfbeta |B|} \sum_{A \subset B} \sum_{A_1,A_2, A_3} \cosh Q_1(A_1) \cosh Q_2(A_2) \cosh Q_3(A_3)
\end{split}
\end{equs}
where the last sum is over all $A_1, A_2 \subset \rB^*(A)$ and $A_3 \subset \rB^*_{\rnn}(B \setminus A)$ such that: no two blocks in $A_1 \cup A_2$ share a root, and no two pairs of blocks in $A_3$ share a root; and, $|A_1| + |A_2| = |A|$ and $|A_3| = |B \setminus A|$. We note that there are $(2 \cdot 27)^{|A|}=54^{|A|}$ possible $A_1$ and $A_2$, and $54^{|B \setminus A|}$ possible $A_3$. 

By Proposition \ref{prop: cosh}, there exists $C_Q$ such that, after taking expectations in \eqref{eq: peierls indicator bounds} and using that $|A| + |B \setminus A| = |B|$, we obtain
\begin{equs}
\nubn(\sigma(\Box) = 0 \text{ for all } \Box \in B)
\leq
2^{|B|}e^{-C_\delta \hfbeta |B|} 2^{|B|} 54^{|B|} e^{C_Q|B|}.	
\end{equs}
Thus, choosing
\begin{equs}
\hfbeta > \frac{4\log 2 + 2\log 54 + 2C_Q}{C_\delta}	
\end{equs}
yields \eqref{eq:peierls} with $C_P= \frac {C_\delta} 2$. This completes the proof.

\subsection{Exchanging the block averaged field for the phase label}

We now show that Propositions \ref{prop:peierls} and \ref{prop: cosh} allow one to reduce the problem of analysing the block averaged field to that of analysing the phase label. The main difficulty here is dealing with \textit{large fields}, i.e. those $\vphi$ for which $\int_{\cB} |\vphi|$ is large. 

\begin{prop}\label{prop:discrete}
	Let $\delta, \delta' \in (0,1)$ satisfy $\delta' \leq \frac \delta 2$. Then, there exists $\beta_0 = \beta_0(\delta,\eta)>0$, $C=C(\delta, \beta_0, \eta)>0$ and $N_0=N_0(\delta)>0$ such that, for all $\beta > \beta_0$ and $N \in 4\NN$ with $N>N_0$,
	\begin{equs}\label{eq:lem:discrete}
	\frac{1}{N^3} \log \nubn \Big( \intx \Big| \sigma - \vphi\Big| dx > \delta \hfbeta N^3 \Big) 
	\leq 
	-C\hfbeta
	\end{equs}
	where $\sigma$ is the phase label of precision $\delta' \leq \frac \delta 2$. 
\end{prop}

\begin{proof}
Observe that
\begin{equs}\label{eq:splitting1}
\begin{split}
\nubn \Bigg( &\intx | \sigma - \vphi|dx > \delta \hfbeta N^3 \Bigg)
\\
&\leq 
\nubn \Bigg( \intx | \sigma - \vphi|dx> \delta \hfbeta N^3 , |\cB| < \frac {\delta}8 N^3 \Bigg) 
\\
&\quad\quad\quad 
+ \nubn \Bigg( |\cB| \geq \frac{\delta}8 N^3 \Bigg).
\end{split}
\end{equs}

By Proposition \ref{prop:peierls}, there exists $\beta_0>0$ and $C_P>0$ such that, for $\hfbeta> \max \Big(\sqrt {\beta_0}, \frac{16 \log 2}{C_P\delta} \Big) $,
\begin{equs} \label{eq:badset}
\begin{split}
\nubn\Bigg(|\cB|\geq \frac{\delta}8  N^3 \Bigg)
&\leq 
\sum_{m = \lceil \frac{\delta}8 N^3 \rceil}^{N^3} \nubn(|\cB| = m) 
	\\
&\leq
\sum_{m=\lceil \frac{\delta}8  N^3 \rceil}^{N^3}{N^3 \choose m}e^{-C_P\hfbeta m} 
	\\
&\leq
2^{N^3} e^{-\frac{C_P\delta}{8}\hfbeta N^3}
\\
&\leq 
e^{-\frac{C_P\delta}{16}\hfbeta N^3}.
\end{split}
\end{equs}

Now consider the first term on the right hand side of \eqref{eq:splitting1}. We decompose one step further:
\begin{equs}
\nubn \Bigg( \intx | \sigma - \vphi|dx> \delta \hfbeta N^3 , |\cB| < \frac{\delta}8 N^3 \Bigg) 
\leq
\nubn(T_1) + \nubn(T_2)
\end{equs}
where
\begin{equs}
T_1 
&=
\Bigg\{\intx | \sigma - \vphi|dx> \delta \hfbeta N^3 , \int_{\cB} | \vphi|dx \leq \frac\delta2 \hfbeta N^3 \Bigg\} \\
T_2 &= \Bigg\{ |\cB| < \frac{\delta}8 N^3, \int_{\cB} | \vphi|dx > \frac \delta 2 \hfbeta N^3 \Bigg\}.
\end{equs}

We show that $T_1 = \emptyset$ and that 
\begin{equs} \label{eq:t2est}
\nubn (T_2) \leq e^{-C\hfbeta N^3}
\end{equs}
for some constant $C=C(\delta)>0$ and for $\beta$ sufficiently large. Combining these estimates with \eqref{eq:badset} completes the proof.

First, we treat $T_1$. On good blocks $|\phi_i - \sigma|$ is bounded by the $\hfbeta$ multiplied by the precision of the phase label ($\delta' \leq \frac \delta 2$ in this instance) and $\sigma = 0$ on bad blocks. Therefore, on the set $\Big\{ \int_{\cB} |\vphi| dx \leq \frac \delta 2 \hfbeta N^3\Big\}$, we have:
\begin{equs}
\intx |\sigma -  \vphi |dx 
&=
 \int_{\cP_G \cup \cM_G}|\sigma - \vphi |dx + \int_{\cB}|\sigma - \vphi|dx
 	\\
&\leq
 \frac \delta 2 \hfbeta (|\cP_G| + |\cM_G|)+ \int_{\cB}|\vphi|dx
 	 \\
&\leq
 \delta  \hfbeta N^3
\end{equs}
which shows that the first condition in $T_1$ is inconsistent with the second, so $T_1 = \emptyset$.

We turn our attention to $T_2$. Fix $B \subset \BBN$. By Chebyschev's inequality, Young's inequality, and Proposition \ref{prop: cosh}, there exists $\beta_0>0$ and $C_Q>0$ such that, for $\beta > \beta_0$,
\begin{equs}
\nubn\Bigg( \int_B |\vphi| > \frac \delta 2 \hfbeta N^3 \Bigg)
&\leq 
e^{-\frac \delta 2 \hfbeta N^3} \langle e^{\sum_{\sBox \in B}|\phi(\sBox)|} \rangle_{\beta,N} 
	\\
&\leq
e^{-\frac \delta 2\hfbeta N^3} e^{\frac \hfbeta 2 |B|} \langle e^{\frac{1}{2\hfbeta} \sum_{\sBox \in B} \phi(\sBox)^2}\rangle_{\beta,N} 
	\\
&\leq
e^{-\frac\delta2 \hfbeta N^3} e^{\hfbeta|B|} \langle e^{\frac{1}{2\hfbeta}\sum_{\sBox \in B}(\phi(\sBox)^2 - \beta)} \rangle_{\beta,N} 
	\\
&\leq
e^{-\frac\delta2 \hfbeta N^3} e^{\hfbeta|B|} \langle \prod_{\Box \in B} e^{-\frac 12 Q_1(\sBox)}e^{-\frac 12 Q_2(\sBox)} \rangle_{\beta,N}
\\
&\leq
e^{-\frac\delta2 \hfbeta N^3} e^{\hfbeta|B|} 2^{|B|} \Big\langle \cosh \Big( \frac 12 Q_1(B) \Big) \cosh \Big( \frac 12 Q_2(B) \Big) \Big\rangle_{\beta,N}
\\
&\leq 
e^{-\frac \delta2  \hfbeta N^3}e^{\hfbeta|B|}2^{|B|}e^{C_Q|B|}.
\end{equs}

Therefore,
\begin{equs}\label{eq:t2estimate}
\begin{split}
\nubn(T_2) 
&\leq
\sum_{m = 1}^{\lfloor \frac{\delta}8 N^3 \rfloor} \sum_{B:|B|=m} \nubn\Bigg( \int_B |\vphi| dx > \frac\delta2 \beta N^3 \Bigg)
	\\
&\leq
 e^{-\frac\delta2 \hfbeta N^3}\sum_{m=1}^{\lfloor \frac{\delta}8 N^3 \rfloor} {N^3 \choose m} e^{\hfbeta m}e^{(C_Q+\log 2)m} \\
&\leq
e^{-\frac \delta2 \hfbeta N^3}2^{N^3} e^{\frac{\delta}8\hfbeta N^3}e^{\frac{(C_Q+\log 2)\delta}8 N^3}
\\
&=
e^{ \big( - \frac {3\delta} 8 \hfbeta + \log 2 + \frac{(C_Q+\log2) \delta}8  \big) N^3}.
\end{split} 
\end{equs}
Taking 
\begin{equs}
\hfbeta 
>
\frac{16\log 2}{3\delta} + \frac 23 (C_Q+\log2)	
\end{equs}
yields \eqref{eq:t2est} with $C=\frac{3\delta}{16}$.
\end{proof}

\subsection{Proof of the main result} \label{subsec: proof of thm ld}

Adapting an argument from \cite{B02}, we reduce the proof of Theorem \ref{thm: ld} to bounding the probability that $\vphi$ is far from $\pm \hfbeta$-valued functions on $\BBN$ whose boundary (between regions of opposite spins) is of certain fixed area. Proposition \ref{prop:discrete} then allows us to go from analysing $\vphi$ to the phase label, for which we use existing results from \cite{BIV00}.

For any $B \subset \BBN$, let $\partial B$ denotes its boundary, which is given by the union of faces of blocks in $B$. Let $|\partial B| = \int_{\partial B} 1 ds(x)$, where $ds(x)$ is the 2D Hausdorff measure (normalised so that faces have unit area). Thus, $|\partial B|$ is the number of faces in $\partial B$. 

For any $a > 0$, let $C_a$ be the set of functions $\vec f \in \{ \pm 1\}^\BBN$ such that $|\partial \{ \vec f = +1\}|\leq aN^2$. For any $\delta > 0$, let $\fB(C_a,\delta)$ be the set of integrable functions $g$ on $\TTN$ such that there exists $\vec f \in C_a$ that satisfies $\intx|g-\vec f| dx \leq \delta N^3$.

\begin{prop}\label{prop:discretetight}
Let $\delta, \delta' \in (0,1)$ satisfy $\delta ' \leq \delta$. Then, there exists $\beta_0 = \beta_0(\delta,\eta)>0$ and $C=C(\delta,\beta_0,\eta)>0$ such that, for all $\beta > \beta_0$, the following estimate holds: for all $a>0$, there exists $N_0 = N_0(a,\delta) \geq 4$ such that, for all $N > N_0$ dyadic,
\begin{equs}
\frac{1}{N^2} \log \nubn \Big( \frac 1\hfbeta \sigma \notin \fB(C_a, \delta) \Big) 
\leq 
-C\hfbeta a
\end{equs}
where $\sigma$ is the phase label of precision $\delta'$.
\end{prop}

\begin{proof}
See \cite[Theorem 2.2.1]{BIV00} where Proposition \ref{prop:discretetight} is proven for a more general class of phase labels that satisfy a Peierls' type estimate such as the one in Proposition \ref{prop:peierls}. We give a self-contained proof for our setting in Section \ref{sec:proofdiscretetight}.
\end{proof}

The following lemma is our main geometric tool. It is a weak form of the isoperimetric inequality on $\TTN$, although it can be reformulated in arbitrary dimension. Its proof is a standard application of Sobolev's inequality and we include it for the reader's convenience.

\begin{lem}\label{lem:isoperimetry}
There exists $C_I>0$ such that the following estimate holds for every $N \in \NN$:
\begin{equs}
	\min(|\{ \vec f = 1 \}|, |\{ \vec f = - 1 \}| )
	\leq
	C_I
	|\partial \{ \vec f = 1 \}|^\frac 32
\end{equs}
for every $\vec f \in \{\pm 1 \}^\BBN$.
\end{lem}
\begin{proof}

Let $\theta \in C^\infty_c(\RR^3)$ be rotationally symmetric with $\int_{\RR^3} \theta dx = 1$. By Sobolev's inequality, there exists $C$ such that, for every $\varepsilon$,
\begin{equs} \label{eq:sobolev}
	\intx |f_\varepsilon - c_\varepsilon|^\frac 32 dx
	\leq
	C \Big( \intx |\nabla f_\varepsilon| dx \Big)^\frac 32
\end{equs}
where $\vec f_\varepsilon = \vec f \ast \varepsilon^{-3} \theta ( \varepsilon^{-1} \cdot )$ and $c_\varepsilon = \frac{1}{N^3} \intx \vec f_\varepsilon dx$. Note that $C$ is independent of $N$ by scaling. 

Letting $\varepsilon \rightarrow 0$ in the left hand side of \eqref{eq:sobolev}, we obtain
\begin{equs} \label{eq:sobolevLHS0}
	\intx |\vec f_\varepsilon - c_\varepsilon|^\frac 32 dx
	\rightarrow
	\intx |\vec f-c|^\frac 32 dx
\end{equs}
where $c = \frac{ |\{ \vec f = 1 \}| - |\{ \vec f = -1 \}|}{N^3}$. Note that $c\in [-1,1]$. 

Without loss of generality, assume $c\geq 0$. This implies that $|\{ \vec f = 1 \}| \geq \{ \vec f = -1 \}|$. Then, evaluating the integral on the righthand side of \eqref{eq:sobolevLHS0}, we find that
\begin{equs} \label{eq:sobolevLHS1}
\begin{split}
\intx |\vec f-c|^\frac 32 dx
&=
(1-c)^\frac 32 |\{\vec f = 1\}| + (1+c)^\frac 32 |\{ \vec f = -1 \}|
\\
&=
(1-c)^\frac 32 cN^3 + \Big( (1-c)^\frac 32 + (1+c)^\frac 32 \Big)|\{ \vec f = -1\}|
\\
&\geq
2|\{\vec f = -1 \}|
\end{split}
\end{equs}
where we have used that the function 
\begin{equs}
c 
\mapsto
(1-c)^\frac 32 + (1+c)^\frac 32
\end{equs}
has minimum at $c=0$ on the interval $[0,1]$. 

For the term on the right hand side of \eqref{eq:sobolev}, using duality we obtain
\begin{equs}\label{eq:sobolevRHS0}
\intx |\nabla \vec f_\varepsilon| dx
=
\sup_{\bg \in C^\infty(\TTN,\RR^3) : |\bg|_\infty \leq 1} \Big| \intx \nabla \vec f_\varepsilon \cdot  \bg dx \Big|
\end{equs}
where $|\cdot|_\infty$ denotes the supremum norm on $C^\infty(\TTN,\RR^3)$.

For any such $\bg$, using integration by parts and commuting the convolution with differentiation,
\begin{equs} \label{eq:sobolevRHS1}
\Big| \intx \nabla \vec f_\varepsilon \bg dx \Big|
=
\Big| \intx \vec f_\varepsilon \nabla \cdot \bg dx \Big|
=
\Big| \intx \vec f \nabla \cdot \bg_\varepsilon dx \Big|
\end{equs}
where the $\bg_\varepsilon$ is interpreted as convolving each component of $\bg$ with $\varepsilon^{-3} \theta(\varepsilon^{-1}\cdot)$ separately. 

Hence, by the divergence theorem, Young's inequality for convolutions, and using the supremum norm bound on $\bg$,
\begin{equs} \label{eq:sobolevRHS2}
\eqref{eq:sobolevRHS1}
=
2\Big| \int_{\partial \{ \vec f = 1 \}} \bg_\varepsilon \cdot \hat n ds(x) \Big|
\leq
2 |\partial \{ \vec f = 1\}|
\end{equs}
where $\hat n$ denotes the unit normal to $\partial \{ \vec f = 1 \}$ pointing into $\{ \vec f = -1\}$. 

Inserting \eqref{eq:sobolevRHS2} in \eqref{eq:sobolevRHS0} implies that, for any $\varepsilon$, 
\begin{equs} \label{eq:sobolevRHS3}
\intx |\nabla \vec f_\varepsilon| dx
\leq
2 |\partial \{ \vec f = 1 \}|.	
\end{equs}

Thus, by inserting \eqref{eq:sobolevRHS3}, \eqref{eq:sobolevLHS0} and \eqref{eq:sobolevLHS1} into \eqref{eq:sobolev}, we obtain
\begin{equs}
	|\{ \vec f = -1 \}|
	\leq 
	\sqrt 2 C |\partial \{ \vec f = 1 \}|^\frac 32. 
\end{equs}
\end{proof}

\begin{proof}[Proof of Theorem \ref{thm: ld}]

Let $\zeta \in (0,1)$. Choose $a>0$ and $\delta \in (0,1)$ such that
\begin{equs}\label{eq:assumption_a_delta}
1 - 2 C_I a^\frac 32 - \delta 
= 
\zeta	
\end{equs}
where $C_I$ is the same constant as in Lemma \ref{lem:isoperimetry}. We first show that
\begin{equs} \label{eq:magnetisation_set_inclusion}
\{ \fm_N(\phi) \in (-\zeta\hfbeta, \zeta\hfbeta) \}
\subset
\Big\{ \frac{1}{\hfbeta} \vphi \notin \fB(C_a,\delta) \Big\}.	
\end{equs}

Assume $\frac{1}{\hfbeta} \vphi \in \fB(C_a,\delta).$ Then, there exists $\vec f \in C_a$ such that
\begin{equs}
\intx \Big| \frac 1\hfbeta \vphi	- \vec f \Big|dx 
\leq 
\delta N^3.
\end{equs}
This implies
\begin{equs}
\Biggr| \Big| \intx \frac {1}{\hfbeta} \vphi dx\Big| - \Big| \intx \vec f dx \Big|\Biggr| 
\leq 
\delta N^3
\end{equs}
from which we deduce, together with Lemma \ref{lem:isoperimetry},
\begin{equs}
\Big|\frac{1}{\hfbeta}\fm_N(\phi) \Big| 
&\geq 
1 - \frac{2\min\Big( |\{\vec f=+1\}|, |\{\vec f=-1\}|\Big)}{N^3} -\delta.
\\
&\geq 
1 - \frac{2C_I |\partial \{ \vec f =+1\}|^\frac 32}{N^3}- \delta.	
\end{equs}
Since $\vec f \in C_a$, we obtain
\begin{equs}
|\fm_N(\phi)| \geq \hfbeta(1-2C_Ia^\frac 32 - \delta) = \zeta\hfbeta	
\end{equs}
by \eqref{eq:assumption_a_delta}.

Hence, 
\begin{equs}
\Big\{ \frac{1}{\hfbeta}\vphi \in \fB(C_a,\delta) \Big\} \subset \{ |\fm_N(\phi)| \geq \zeta\hfbeta \}.	
\end{equs}
Taking complements establishes \eqref{eq:magnetisation_set_inclusion}.

Now let $\sigma$ be the phase label of precision $\frac \delta 2$. Note that
	\begin{equs}
	\Big\{ \frac 1\hfbeta \vphi \notin \fB(C_a,\delta) \Big\}
	 \subset 
	\Big\{ \frac 1\hfbeta \sigma \notin \fB(C_a, 2\delta) \Big\} \bigcup \Big\{ \intx | \vphi - \sigma |dx>\delta \hfbeta N^3\Big\}.
	\end{equs}
Applying Proposition \ref{prop:discrete}, Proposition \ref{prop:discretetight}, and using \eqref{eq:magnetisation_set_inclusion} finishes the proof.
\end{proof}

\subsection{Proof of Proposition \ref{prop:discretetight}} \label{sec:proofdiscretetight}

For any $B \subset \BBN$, let $\partial^* B$ be the set of blocks in $B$ with $*$-neighbours in $\TTN \setminus B$. Note that this is not the same as $\partial B$, which was defined earlier. Let $\cD$ be the set of $*$-connected components of $\partial^* (\TTN \setminus \cM_G)$. We call this the set of \textit{defects}. Necessarily, any $\Gamma \in \cD$ satisfies $\Gamma \subset \cB$. 

Fix $\gamma \in (0, 1)$. Let $\cD^\gamma \subset \cD$ be the set of $\Gamma \in \cD$ such that $|\Gamma| \leq 6 N^\gamma$. The elements of $\cD^\gamma$ are called $\gamma$-small defects and the elements of $\cD \setminus \cD^\gamma$ are called $\gamma$-large defects. 

Take any $\Gamma \in \cD^\gamma$. Recall that we identify $\Gamma$ with the subset of $\TTN$ given by the union of blocks in $\Gamma$. Write $\mathrm{Cl}(\Gamma)$ for its closure in $\TTN$. The condition $\gamma < 1$ ensures that, provided $N$ is taken sufficiently large depending on $\gamma$, any $\Gamma \in \cD^\gamma$ is contained in a (translate of a) sphere of radius $\frac N 4$ in $\TTN$. Let $\rExt(\Gamma)$ be the unique connected component of $\TTN \setminus \mathrm{Cl}(\Gamma)$ that intersects with the complement of this sphere. Let $\rInt(\Gamma) = \TTN \setminus \rExt(\Gamma)$. We identify $\rExt(\Gamma)$ and $\rInt(\Gamma)$ with their representations as subsets of $\BBN$. Note that $\Gamma \subset \rInt(\Gamma)$ and generically the inclusion strict, e.g. when $\Gamma$ encloses a region.

Let $\cD^{\gamma,\max}$ be the set of $\Gamma \in \cD^{\gamma}$ such that $\Gamma \bigcap \rInt(\tilde\Gamma) = \emptyset$ for any $\tilde \Gamma \in \cD^\gamma \setminus \{ \Gamma \}$. In other words, $\cD^{\gamma,\max}$ is the set of $\gamma$-small defects that are not contained in the interior of any other $\gamma$-small defects, and we call these maximal $\gamma$-small defects.

We define two events, one corresponds to the total surface area of $\gamma$-large defects being small and the other corresponding to the total volume contained within maximal $\gamma$-small defects being small. Let 
\begin{equs}
S_1 
&=
\Bigg\{ \sum_{\Gamma \in \cD \setminus \cD^\gamma} |\Gamma| \leq \frac{a}6N^2 \Bigg\} 
	\\
S_2
&=
\Bigg\{ \sum_{\Gamma \in \cD^{\gamma,\max}} |\rInt (\Gamma)| \leq \frac\delta4 N^3 \Bigg\}.
\end{equs}

We now show that for $\phi \in S_1 \cap S_2 \cap \{|\cB| < \frac\delta2 N^3\}$, we have $\frac 1\hfbeta \sigma \in \fB(C_a,\delta)$. 

We obtain a $\pm \hfbeta$-valued spin configuration from $\sigma$ by erasing all $\gamma$-small defects in two steps: First, we reset the values on bad blocks to $\hfbeta$. Define $\sigma_1 \in \{ \pm \hfbeta \}^\BBN$ by $\sigma_1(\Box) = \hfbeta$ if $\Box \in \cB$, otherwise $\sigma_1(\Box) = \sigma(\Box)$. Second, define $\sigma_2\in\{\pm \hfbeta \}^\BBN$ as follows: Given $\Box \in \rInt(\Gamma)$ for some $\Gamma \in \cD^{\gamma,\max}$, let $\sigma_2(\Box) = \sigma_1(\tilde\Box)$, where $\tilde\Box$ is any block in $\rExt(\Gamma)$ that is $*$-neighbours with a block in $\Gamma$. Note that the second step is well-defined since the first step ensures that every block in $\rExt(\Gamma)$ that is $*$-neighbours with $\Gamma$ has the same value. See Figure \ref{fig: small contour} for an example of this procedure. 

\begin{figure}[!htb]
\includegraphics{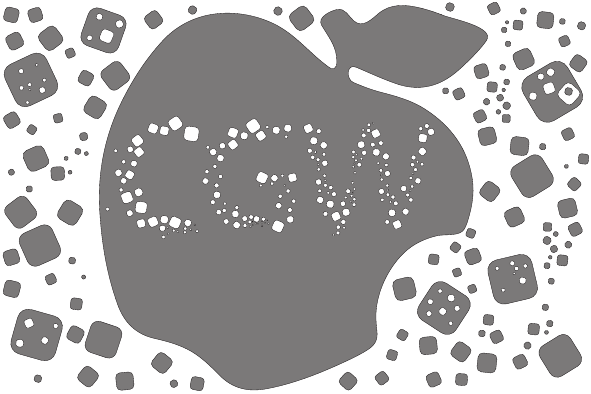}
\includegraphics{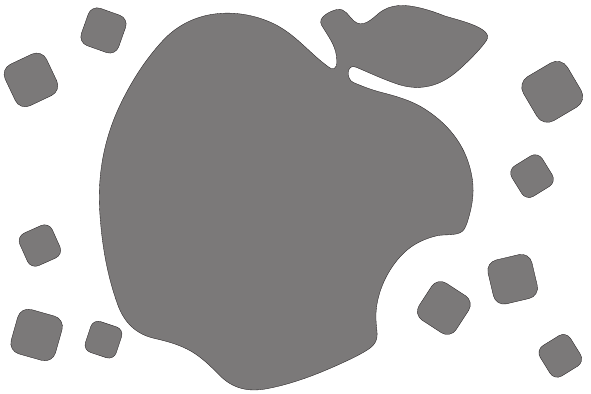}
\caption{An example of the $\sigma$ to $\sigma_2$ procedure (left to right). Image courtesy of J. N. Gunaratnam}
\label{fig: small contour}
\end{figure}

From the definition of $S_1$ and using that the factor $6$ in the definition of $\gamma$-small defects accounts for the discrepancy between $|\partial\cdot|$ and $|\partial^*\cdot|$,
\begin{equs}
|\partial \{ \sigma_2 = +\hfbeta \}|
\leq 
aN^2	
\end{equs}
yielding $\frac 1\hfbeta \sigma_2 \in C_a$. Then, from the definition of $S_2$ and using the smallness assumption on the number of bad blocks,
\begin{equs}
\intx \frac 1\hfbeta \Big|\sigma - \sigma_2 \Big|dx 
\leq 
2\sum_{\Gamma \in \cD^{\gamma,\max} } |\rInt(\Gamma)| + |\cB| 
< 
2 \frac\delta 4 N^3 + \frac\delta2 N^3
<
\delta N^3
\end{equs}
which establishes that $\frac 1\hfbeta \sigma \in \fB(C_a,\delta)$.

We deduce that the event $\Big\{ \frac 1\hfbeta \sigma \not\in \fB(C_a,\delta) \Big\}$ necessarily implies one of three things: either there are many bad blocks; or, the total surface area of $\gamma$-large defects is large; or, the density of $\gamma$-small defects is high. That is,
\begin{equs} \label{eq:phaselabel_to_contours}
\begin{split}
\nubn\Big( \frac 1\hfbeta \sigma \notin &\fB(C_a,\delta) \Big)
\\	
&\leq 
\nubn\Big(|\cB|\geq \frac\delta2 N^3\Big) + \nubn(S_1^c) + \nubn(S_2^c). 
\end{split}
\end{equs}

Proposition \ref{prop:peierls} gives control on the first event. The other two are controlled by the following lemmas.

\begin{lem}\label{lem:largecont}
Let $\gamma,\delta \in (0,1)$. Then, there exists $\beta_0=\beta_0(\gamma,\delta,\eta)>0$ and $C=C(\gamma,\delta,\beta_0,\eta)>0$ such that, for all $\beta > \beta_0$, the following holds: for any $a > 0$, there exists $N_0 = N_0(\gamma,a)>0$ such that, for any $N \in 4\NN$ with $N > N_0$,
\begin{equs}
\frac{1}{N^2}\log \nubn \Bigg( \sum_{\Gamma \in \cD \setminus \cD^\gamma}|\Gamma| > a N^2 \Bigg) 
\leq 
-C\hfbeta \Big( a + \frac{N^\gamma}{N^2} \Big)
\end{equs}
where the underlying phase label is of precision $\delta$.
\end{lem}

\begin{proof}
We give a proof based on arguments from \cite[Theorem 6.1]{DKS92} in Section \ref{sec:largecont}.	
\end{proof}

\begin{lem}\label{lem:smallcont}
Let $\gamma,\delta,\delta' \in (0,1)$. Then, there exists $\beta_0=\beta_0(\gamma,\delta,\delta',\eta)>0$, $C=C(\gamma,\delta, \delta',\beta_0,\eta)>0$ and $N_0=N_0(\gamma, \delta) \geq 4$ such that, for all $\beta > \beta_0$ and $N > N_0$ dyadic,
\begin{equs}
\frac{1}{N^2} \log \nu_{\beta,N} \Bigg( \sum_{\Gamma \in \cD^{\gamma,\max}} |\mathrm{Int}(\Gamma)| > \delta N^3 \Bigg) 
\leq 
-C\hfbeta \frac{N}{N^{3\gamma}}	
\end{equs}
where the underlying phase label is of precision $\delta'$.
\end{lem}

\begin{proof}
See \cite[Section 5.1.3]{BIV00} for a proof in a more general setting. We give an alternative proof in Section \ref{sec:smallcont} that avoids the use of techniques from percolation theory.
\end{proof}

As in \eqref{eq:badset}, by Proposition \ref{prop:peierls} there exists $C_P > 0$ such that
\begin{equs} \label{eq:badset_phaselabelsection}
\nubn(|\cB|\geq \delta N^3) 
\leq 
e^{-\frac{C_P\delta}{4}\hfbeta N^3}
\end{equs}
provided $\hfbeta > \frac{4\log 2}{\delta C_P}$.

Therefore, from \eqref{eq:phaselabel_to_contours}, \eqref{eq:badset_phaselabelsection}, Lemma \ref{lem:largecont} and Lemma \ref{lem:smallcont}, there exists $C>0$ such that
\begin{equs}
\frac{1}{N^2}\log\nubn \Big( \sigma \notin \fB(C_a,\delta) \Big) 
\leq 
-C\hfbeta \min \Big( N, a + \frac{N^\gamma}{N^2}, \frac{N}{N^{3\gamma}}\Big).
\end{equs}
Taking $\gamma < \frac 13$ and $N$ sufficiently large completes the proof. All that remains is to show Lemmas \ref{lem:largecont} and \ref{lem:smallcont}. 

\subsubsection{Proof of Lemma \ref{lem:largecont}} \label{sec:largecont}

By a union bound
\begin{equs} \label{eq:largecont:union}
\begin{split}
\nubn \Biggr( \sum_{\Gamma \in \cD \setminus \cD^\gamma}|\Gamma| > a N^2 \Biggr)
&=
\sum_{\substack{\{ \Gamma_i \}: |\Gamma_i| > N^\gamma \\ \sum_i |\Gamma_i| > aN^2 } }\nubn \Big(\cD\setminus\cD^\gamma = \{ \Gamma_i \} \Big)
	\\
&\leq
\sum_{\substack{\{ \Gamma_i \}: |\Gamma_i| > N^\gamma \\ \sum_i |\Gamma_i| > aN^2 } } \nubn \Big( \Gamma_i \subset \cB \text{ for all } \Gamma_i \in \{ \Gamma_i \} \Big),
\end{split}
\end{equs}
where$\{ \Gamma_i \}$ refers to a non-empty set  of distinct $*$-connected subsets of $\BBN$.

By Proposition \ref{prop:peierls} there exists $C_P$ such that, for any $\{ \Gamma_i \}$,
\begin{equs}
\nubn \Big( \Gamma_i \subset \cB \text{ for all } \Gamma_i \in \{\Gamma_i\} \Big)
&=
\langle \prod_{\Gamma_i \in \{ \Gamma_i \}} \prod_{\sBox \in \Gamma_i} \1_{\sBox \in \cB} \rangle_{\beta,N}
\\
&\leq
e^{-C_P\hfbeta \sum|\Gamma_i|}.
\end{equs}

Inserting this into \eqref{eq:largecont:union} and using the trivial estimate $\sum|\Gamma_i| \geq \frac12 aN^2  + \frac12 \sum|\Gamma_i|$,
\begin{equs} \label{eq:largecont_energyentropy} 
\begin{split}
\nubn \Biggr( \sum_{\Gamma \in \cD\setminus\cD^\gamma}|\Gamma| > a N^2 \Biggr)
&\leq 
\sum_{\substack{\{ \Gamma_i \}: |\Gamma_i| > N^\gamma \\ \sum_i |\Gamma_i| > aN^2 } }	e^{-C_P\hfbeta \sum|\Gamma_i|}
	\\
&\leq
e^{-\frac{C_P}2 \hfbeta aN^2} \sum_{\{\Gamma_i\}: |\Gamma_i|>N^\gamma} e^{-\frac{C_P}2 \hfbeta \sum|\Gamma_i|}
	\\
&=
e^{-\frac{C_P}2 \hfbeta aN^2} \sum_{\{ \Gamma_i \}: |\Gamma_i| > N^\gamma} \prod_{\Gamma_i \in \{ \Gamma_i \}} e^{-\frac {C_P} 2 \hfbeta |\Gamma_i|}.
\end{split}
\end{equs}

Summing first over the number of elements in $\{ \Gamma_i \}$ and then the number of $*$-connected regions containing a fixed number of blocks,
\begin{equs}\label{eq:largecont_entropy}
\begin{split}
\sum_{\substack{\{ \Gamma_i \} \\ |\Gamma_i| >N^\gamma }} \prod_{\Gamma_i \in \{ \Gamma_i \}} e^{-\frac{C_P}2 \hfbeta |\Gamma_i|}	
&=
\sum_{m=1}^\infty \sum_{\{\Gamma_i\}_{i=1}^m : |\Gamma_i| > N^\gamma} \prod_{i=1}^m e^{-\frac{C_P}2 \hfbeta |\Gamma_i|}
\\
&\leq
\sum_{m=1}^\infty \Big( \sum_{\Gamma \text{ $*$-connected }: |\Gamma| \geq N^\gamma} e^{-\frac{C_P}2 \hfbeta |\Gamma|} \Big)^m
\\
&\leq
\sum_{m=1}^\infty \Big( \sum_{n \geq N^\gamma} N^3 27 \cdot 26^{n-1} e^{-\frac{C_P}2\hfbeta n} \Big)^m
\\
&\leq
\sum_{m=1}^\infty e^{3m \log N - \frac{C_P}4 \hfbeta m N^\gamma} \Big( \sum_{n \geq 1} e^{-\frac{C_P}4 \hfbeta n} \Big)^m
\\
&\leq
\sum_{m=1}^\infty e^{\Big(3\log N - \frac{C_P}4 \hfbeta  N^\gamma\Big)m}
\\
&\leq
e^{-\frac{C_P}{8}\hfbeta N^\gamma}\sum_{m=1}^\infty e^{3m\log N - \frac{C_P}8 \hfbeta m N^\gamma}
\end{split}
\end{equs}
provided $\hfbeta > \max \Big( \frac{4\log 27}{C_P}, \frac{4\log 2}{C_P} \Big) = \frac{4 \log 27}{C_P} $ (note that the condition arises so that $e^{-\frac{C_P}4 \hfbeta} < \frac 12$, so that the geometric series with this rate is bounded by $1$). 

For any $\gamma > 0$, the final series in \eqref{eq:largecont_entropy} is summable provided $N^\gamma > \log N$ and $\hfbeta > \frac{24}{C_P}$, thereby finishing the proof.

\subsubsection{Proof of Lemma \ref{lem:smallcont}} \label{sec:smallcont}

Choose $ 2 N^\gamma \leq K \leq 4 N^\gamma$ such that $K$ divides $N$. Such a choice is possible since we take $N$ to be a sufficiently large dyadic. Let
\begin{equs}
\BB_N^K
=
\Big\{ \blackBox = \prod_{i=1}^3[n_i,n_i+K) \subset \TTN : n_1,n_2,n_3 \in \{0,K,\dots,N-K\} \Big\}.	
\end{equs}
Elements of $\BB_N^K$ are called $K$-blocks.

We say that two distinct $K$-blocks are $*_K$-neighbours if their corresponding midpoints are of distance at most $K\sqrt 3$. We define the $*_K$-connected ball around $\blackBox \in \BB_N^K$ to be the set containing itself and its $*_K$-neighbours. As in the proof of Proposition \ref{prop:peierls}, we can decompose $\BB_N^K = \bigcup_{l=1}^{27} \BB_N^{K,l}$ such that any $*_K$-connected ball in $\BB_N^K$ contains exactly one $K$-block from each element of the decomposition. 

For each $\blackBox = [n_1, n_1+K)\times[n_2,n_2+K)\times[n_3,n_3+K)$, distinguish the unit block $\BoxBox = [n_1,n_1+1)\times[n_2,n_2+1)\times[n_3,n_3+1)$. For every $h \in \{0,\dots,K-1\}^3$, let $\tau_h$ be the translation map on $\BBN$ induced from the translation map on $\TTN$. We identify $\blackBox = \bigcup_{h \in \{0,\dots,K-1\}^3} \tau_h\BoxBox$. Denote the set of distinguished unit blocks in $\BB_N^K$ (respectively, $\BB_N^{K,l}$) as $\UU\BB_N^K$ (respectively, $\UU\BB_N^{K,l}$).

By our choice of $K$, $\rInt(\Gamma)$ is entirely contained in a translation of a $K$-block for any $\Gamma \in \cD^\gamma$. As a result, $\rInt(\Gamma)$ intersects at most one $K$-block in $\BB_N^{K,l}$ for any fixed $l$.

Using the correspondence between $K$-blocks and unit blocks described above, we have
\begin{equs}
\sum_{\Gamma \in \cD^{\gamma,\max}} |\rInt(\Gamma)| 
&= 
\sum_{\sBox \in \BBN} \sum_{\Gamma \in \cD^{\gamma,\max}} \1_{\sBox \in \rInt(\Gamma)}
\\
&=
\sum_{\sBoxBox \in \UU\BB_N^K} \sum_{h \in \{0,\dots,K-1\}^3}  \sum_{\Gamma \in \cD^{\gamma,\max}} \1_{\tau_h \sBoxBox \in \rInt(\Gamma)}
\\
&=
\sum_{l=1}^{27}  \sum_{\sBoxBox \in \UU\BB_N^{K,l}} \sum_{h \in \{0,\dots,K-1\}^3}\sum_{\Gamma \in \cD^{\gamma,\max}} \1_{\tau_h \sBoxBox \in \rInt(\Gamma)}.
\end{equs}

Hence,
\begin{equs} \label{eq:smallcontourpain0}
\nubn \Bigg( &\sum_{\Gamma \in \cD^{\gamma,\max}} |\rInt(\Gamma)|>\delta N^3 \Bigg)
\\
&\leq
27 K^3 \max_{h,l} \nubn \Bigg(\sum_{\sBoxBox \in \UU\BB_N^{K,l}} \sum_{\Gamma \in \cD^{\gamma,\max}} \1_{\tau_h \sBoxBox \in \rInt(\Gamma)} > \frac{\delta}{27} \Big(\frac{N}{K}\Big)^3 \Bigg).	
\end{equs}
where the maximum is over $h \in \{0,\dots,K-1\}^3$ and $1\leq l \leq 27$.

Let $E_k$ be the event that precisely $k$ indicator functions appearing on the right hand side of \eqref{eq:smallcontourpain0} are nonzero. In other words, $E_k$ is the event that there are $k$ distinct defects of size at most $N^\gamma$ such that the $k$ distinct $\tau_h \BoxBox$, where $\BoxBox \in \UU\BB_N^{K,l}$, are contained in their interiors.

Given a block there are $27 \cdot 26^{n-1}$ possible defects of size $n$ that contain this block. Thus, by Proposition \ref{prop:peierls}, there exists $C_P$ such that
\begin{equs} \label{eq:smallcontourpain2}
\nubn(E_k)
&\leq
{ \frac{N^3}{27K^3} \choose k }  \sum_{1 \leq n_1, \dots, n_k \leq N^\gamma} \prod_{j=1}^k n_j\cdot 26\cdot27^{n_j-1}e^{-C_P\hfbeta n_j}
\\
&\leq
{ \frac{N^3}{27K^3} \choose k } e^{-\frac {C_P} 2 \hfbeta k} \Big(\sum_{n=1}^{N^\gamma} n \cdot 26 \cdot 27^{n-1} e^{-\frac{C_P}{2} \hfbeta n} \Big)^k
\\
&\leq
{ \frac{N^3}{27K^3} \choose k } e^{- \frac {C_P} 2 \hfbeta k}
\end{equs}
provided e.g. $\hfbeta > \max \Big( \frac{4\log 27}{C_P}, \frac{2\log 2}{C_P} \Big) = \frac{4\log 27}{C_P}$. This estimate is uniform over the choice of $h$ and $l$. 

By a union bound on \eqref{eq:smallcontourpain0}, using \eqref{eq:smallcontourpain2}, and that $2N^\gamma \leq K \leq 4N^\gamma$,
\begin{equs}
\nubn \Big( \sum_{\Gamma \in \cD^{\gamma,\max}} |\rInt(\Gamma)|>\delta N^3) 
&\leq
27 K^3 \sum_{k = \lfloor \frac{\delta N^3}{27 K^3} \rfloor + 1}^{\frac{N^3}{27K^3}}{ \frac{N^3}{27K^3} \choose k } e^{-\frac{C_P}2 \hfbeta k}
\\
&\leq
27K^3 \cdot 2^{\frac{N^3}{27K^3}}e^{-\frac{\delta C_P}{2\cdot 27}\hfbeta \frac{N^3}{K^3}} 
\\
&\leq
27 \cdot 64 e^{3 \gamma \log N + \frac{\log 2}{27 \cdot 8}\frac{N^3}{N^{3\gamma}} - \frac{\delta C_P}{27 \cdot 16}\hfbeta \frac{N^3}{N^{3\gamma}}}
\\
&\leq
27 \cdot 64 e^{-\frac{\delta C_P}{27 \cdot 32}\hfbeta \frac{N^3}{N^{3\gamma}}}
\end{equs}
provided $\gamma \log N < N^{3-3\gamma}$ and $\hfbeta > \frac{81 \cdot 32 + 4\log2}{\delta C_P}$. Taking logarithms and dividing by $N^2$ completes the proof.

\section{Bou\'e-Dupuis formalism for $\phi^4_3$}\label{sec: bd}

In this section we introduce the underlying framework that we build on to analyse expectations of certain random variables under $\nubn$, as required in the proof of Proposition \ref{prop: cosh}. This framework was originally developed in \cite{BG19} to show ultraviolet stability for $\phi^4_3$ and identify its Laplace transform. 

\sloppy In particular, we want to obtain estimates on expectations of the form $\langle e^{Q_K} \rangle_{\beta,N,K}$, where $Q_K$ are random variables that converge (in an appropriate sense) to some random variable $Q$ of interest. We always work with a fixed ultraviolet cutoff $K$ and establish estimates on $\langle e^{Q_K} \rangle_{\beta,N,K}$ that are uniform in $K$: this requires handling of ultraviolet divergences. The first observation is that we can represent such expectations as a ratio of Gaussian expectations: 
\begin{equs} \label{eq: tilde q expectation}
\langle e^{Q_K} \rangle_{\beta,N,K}
=
\frac{\EE_N e^{-\cH_{\beta,N,K}(\phi_K) + Q_K(\phi_K)}}{\sZ_{\beta,N,K}}	
\end{equs}
where we recall $\EE_N$ denotes expectation with respect to $\mu_N$ and $\sZ_{\beta,N,K} = \EE_N e^{-\cH_{\beta,N,K}(\phi_K)}$ is the partition function. 

We then introduce an auxiliary time variable that continuously varies the ultraviolet cutoff between $0$ and $K$, and use it to represent these Gaussian expectations in terms of expectations of functionals of finite dimensional Brownian motions. This allows us to use the Bou\'e-Dupuis variational formula given in Proposition \ref{prop: bd} to write these expectations in terms of a stochastic control problem. Hence, the problem of obtaining bounds is translated into choosing appropriate controls. An insight made in \cite{BG19} is that one can use methods developed in the context of singular stochastic PDEs, specifically the paracontrolled calculus approach of \cite{GIP15}, within the control problem to kill ultraviolet divergences. 

\begin{rem} \label{rem: appendices}
In the following, we make use of tools in Appendices \ref{appendix: sub: besov} and \ref{appendix: sub: paracontrolled} concerning Besov spaces and paracontrolled calculus. In addition, for the rest of Sections \ref{sec: bd} and \ref{sec: free energy}, we consider $N \in \NN$ fixed and drop it from notation when clear.
\end{rem}

\subsection{Construction of the stochastic objects} \label{subsec: construction of stochastic objects}
Fix $\kappa_0 > 0$ sufficiently small. We equip $\Omega = C(\RR_+; \cC^{-\frac 32 -\kappa_0})$ with its Borel $\sigma$-algebra. Denote by $\PP$ the probability measure on $\Omega$ under which the coordinate process $X_{\bullet}=(X_k)_{k \geq 0}$ is an $L^2$ cylindrical Brownian motion. We write $\EE$ to denote expectation with respect to $\PP$. We consider the filtered probability space $(\Omega, \cA,  (\cA_k)_{k \geq 0},\PP)$, where $\cA$ is the $\PP$-completion of the Borel $\sigma$-algebra on $\Omega$, and $(\cA_k)_{k \geq 0}$ is the natural filtration induced by $X$ and augmented with $\PP$-null sets of $\cA$.

Given $n \in (N^{-1}\ZZ)^3$, define the process $B^n_\bullet$ by $B^n_k = \frac 1{N^\frac 32} \intx X_k e_{-n} dx$, where $e_n(x) = e^{2\pi i n \cdot x}$ and we recall that the integral denotes duality pairing between distributions and test functions. Then, $\{B^n_\bullet : n \in (N^{-1}\ZZ)^3\}$ is a set of complex Brownian motions defined on $(\Omega, \cA,  (\cA_k)_{k \geq 0},\PP)$, independent except for the constraint $\overline{B_k^n} = B_k^{-n}$. Moreover,
\begin{equs}
X_k 
= 
\frac{1}{N^3} \sum_{n \in (N^{-1}\ZZ)^3} B_k^n N^\frac 32 e_n
\end{equs}
where $\PP$-almost surely the sum converges in $\cC^{-\frac 32 - \kappa_0}$.

Let $\cJ_k$ be the Fourier multiplier with symbol
\begin{equs}
\cJ_k(\cdot)
=
\frac{\sqrt{\partial_k \rho_k^2( \cdot)}}{\langle \cdot \rangle}
\end{equs}
where $\rho_k$ is the ultraviolet cutoff defined in Section \ref{sec: model} and we recall $\langle \cdot \rangle = \sqrt{\eta + 4\pi^2|\cdot|^2}$. $\cJ_k$ arises from a continuous decomposition of the covariance of the pushforward measure $\mu_N$ under $\rho_k$:
\begin{equs}
\int_0^k \cJ_{k'}^2(\cdot) d{k'}
=
\frac{\rho_k^2(\cdot)}{\langle \cdot \rangle^2}	
=
\cF \Big\{ \cF^{-1}(\rho_k) \ast (-\Delta + \eta)^{-1} \ast \cF^{-1}(\rho_k) \Big\} (\cdot)
\end{equs}
where $\cF$ denotes the Fourier transform and $\cF^{-1}$ denotes its inverse (see Appendix \ref{appendix: sub: basic}). Note that the function $\partial_k \rho_k^2$ has decay of order $\langle k \rangle^{-\frac 12}$ and the corresponding multiplier is supported frequencies satisfying $|n| \in (c_\rho k, C_\rho k)$ for some $c_\rho < C_\rho$. Thus, we may think of $\cJ_k$ as having the same regularising properties as the multiplier $\frac{\cF \{(-\Delta + \eta)^{-\frac 12}\}}{\langle k \rangle^\frac 12} \1_{c_\rho k \leq |\cdot| \leq C_\rho k}$; precise statements are given in Proposition \ref{prop: multiplier estimate}.

Define the process $\<1>_\bullet$ by
\begin{equs} \label{def: lollipop}
\<1>_k
=
\int_0^k \cJ_{k'} dX_{k'}
=
\frac{1}{N^\frac 32} \sum_{n \in (N^{-1}\ZZ)^3} \Bigg( \int_0^k  \frac{\sqrt{\partial_{k'} \rho_{k'}^2 ( n)}}{\langle n \rangle} dB_{k'}^n \Bigg) e_n.
\end{equs}
$\<1>_\bullet$ is a centred Gaussian process with covariance:
\begin{equs}
	\EE \Big[ \intx \<1>_k f dx \intx \<1>_{k'} g dx \Big]
	=
	\frac{1}{N^3} \sum_{n \in (N^{-1}\ZZ)^3} \frac{\rho^2_{\min(k,k')}}{\langle n \rangle^2} \cF f(n) \cF g(n)
\end{equs}
for any $f,g \in L^2$. Thus, the law of $\<1>_k$ is the law of $\rho_k \phi$ where $\phi \sim \mu_N$. As with other processes in the following, we simply write $\<1> = \<1>_\bullet$. 

\subsubsection{Renormalised multilinear functions of the free field} \label{subsec: diagrams}

The second, third, and fourth Wick powers of $\<1>$ are the space-stationary stochastic processes $\<2>, \<3>, \<4>$ defined by: 
\begin{equs}
\<2>_k
&=
\<1>_k^2 - \<tadpole>_k
\\
\<3>_k
&=
\<1>_k^3 - 3 \<tadpole>_k
\\
\<4>_k
&=
\<1>_k^4 - 6 \<tadpole>_k \<1>_k^2 + 3 \<tadpole>_k^2	
\end{equs}
where we recall from Section \ref{sec: model} that $\<tadpole>_k = \EE_N [\phi_k^2(0)] = \EE [\<1>_k^2(0)]$. Note that $\<2>_k, \<3>_k$, and $\<4>_k$ are equal in law to $:\phi_k^2:, :\phi_k^3:$, and $:\phi_k^4:$, respectively. 

The Wick powers of $\<1>$ can be expressed as iterated integrals using It\^o's formula (see \cite[Section 1.1.2]{N06}). We only need the iterated integral representation $\<3>$:
\begin{equs} \label{eq: iterated integral third wick}
\<3>_k
&=
\frac{3!}{N^\frac{9}{2}} \sum_{n_1,n_2,n_3} \int_0^k \int_0^{k_1} \int_0^{k_2} \prod_{i=1}^3 \frac{\sqrt{\partial_{k_i}\rho_{k_i}^2(n_i)}}{\langle n_i \rangle} dB^{n_3}_{k_3} dB^{n_2}_{k_2} dB^{n_1}_{k_1}
\end{equs}
where we have used the convention that sums over frequencies $n_i$ range over $(N^{-1}\ZZ)^3$. 

We define additional space-stationary stochastic processes $\<30>, \<31>, \<32>, \<202>$ by
\begin{equs}
\<30>_k
&=
\int_0^k \cJ_{k'}^2 \<3>_{k'} d{k'} 
\\
\<31>_k
&=
\<1>_k \pe \<30>_k
\\
\<32>_k
&=
\<2>_k \pe \<30>_k - \frac{12}{N^6}\<1>_k\sum_{n_1 + n_2 + n_3} \int_0^k \frac{\rho_{k'}^2(n_1) \rho_{k'}^2(n_2)\partial_{k'} \rho_{k'}^2 (n_3)}{\langle n_1 \rangle^2 \langle n_2 \rangle^2 \langle n_3 \rangle^2} d{k'}
\\
\<202>_k
&=
\cJ_k \<2>_k \pe \cJ_k \<2>_k - \frac{4}{N^6} \sum_{n_1 + n_2 + n_3 = 0} \frac{\rho_k^2(n_1) \rho_k^2(n_2)\partial_k \rho_k^2 (n_3)}{\langle n_1 \rangle^2 \langle n_2 \rangle^2 \langle n_3 \rangle^2}.
\end{equs}

We make two observations: first, a straightforward calculation shows that $\<3>_k$ diverges in variance as $k \rightarrow \infty$. However, due to the presence of $\cJ_k$, $\<30>_k$ can be made sense of as $k \rightarrow \infty$. See Lemma \ref{lem:trident}. 

Second, $\<31>_k$, $\<32>_k$, and $\<202>_k$ are renormalised resonant products of $\<1>_k \<30>_k$, $\<2>_k\<30>_k$, and $(\cJ_k \<2>_k)^2$, respectively. The latter products are classically divergent in the limit $k \rightarrow \infty$. We refer to Remark \ref{rem: resonant} for an explanation of why the resonant product is used. 

\begin{rem} \label{rem: resonant}
Let $f \in \cC^{s_1}$ and $g \in \cC^{s_2}$ for $s_1 < 0 < s_2$. Bony's decomposition states that, if the product exists, $fg = f \pl g + f \pe g + f \pg g$ and is of regularity $s_1$ (see Appendix \ref{appendix: sub: paracontrolled}). Since paraproducts are always well-defined (see Proposition \ref{prop:paraproduct}), the resonant product contains all of the difficulty in defining the product. However, the resonant product gives regularity information of order $s_1 + s_2$ (see Proposition \ref{prop: resonant}), which is strictly stronger than the regularity information of the product: i.e. the bound on $\| f \pe g \|_{\cC^{s_1 + s_2}}$ is strictly stronger than the bound on $\| fg \|_{\cC^{s_1}}$. This is the key property that makes paracontrolled calculus useful in this context \cite{GIP15}.
\end{rem}

The required renormalisations of $\<32>_K$ and $\<202>_K$ are related to the usual "sunset" diagram appearing in the perturbation theory for $\phi^4_3$,
\begin{equs} \label{def: sunset}
\<sunset>_k
=
\frac{1}{N^6}\sum_{n_1 + n_2 + n_3 = 0} \frac{\rho_k^2(n_1)\rho_k^2(n_2)\rho_k^2(n_3)}{\langle n_1 \rangle^2 \langle n_2 \rangle^2 \langle n_3 \rangle^2}.	
\end{equs}	
See \cite[Theorem 1]{F74}. We emphasise that $\<sunset>_k$ depends on $\eta, N$ and $k$. 

By the fundamental theorem of calculus, the Leibniz rule, and symmetry,
\begin{equs}
\<sunset>_k
&=
\frac{1}{N^6}\sum_{n_1 + n_2 + n_3 = 0} \int_0^k \frac{\partial_{k'} \Big( \rho_{k'}^2(n_1)\rho_{k'}^2(n_2)\rho_{k'}^2(n_3) \Big)}{\langle n_1 \rangle^2 \langle n_2 \rangle^2 \langle n_3 \rangle^2} dk'
\\
&=
\frac{3}{N^6} \sum_{n_1 + n_2 + n_3=0} \frac{\int_0^k \rho_{k'}^2(n_1)\rho_{k'}^2(n_2)\partial_{k'} \rho_k^2(n_3)}{\langle n_1 \rangle^2 \langle n_2 \rangle^2 \langle n_3 \rangle^2}.
\end{equs}

Thus, the renormalisations of $\<32>_K$ and $\<202>_k$ are given by $4\<sunset>_k \<1>_k$ and $\frac 43 \partial_k \<sunset>_k$, respectively.

\begin{rem} \label{rem: growth of renorm}
It is straightforward to verify that there exists $C=C(\eta)>0$ such that 
\begin{equs}
	\<sunset>_k
	\leq
	\frac{C(\eta)}{N^6}\log \langle k \rangle
	\quad \text{ and } \quad
	\partial_k \<sunset>_k
	\leq
	\frac{C(\eta)}{N^6} \frac{\log \langle k \rangle}{\langle k \rangle}.
\end{equs}
\end{rem}

Let $\Xi = (\<1>, \<2>, \<30>, \<31>, \<32>, \<202>)$. We refer to the coordinates of $\Xi$ as diagrams. The following proposition gives control over arbitrarily high moments of diagrams in Besov spaces.

\begin{prop}\label{prop: diagrams}
For any $p,p' \in [1,\infty)$, $q \in [1,\infty]$, and $\kappa > 0$ sufficiently small, there exists $C=C(p,p',q,\kappa,\eta)>0$ such that
\begin{equs} \label{eq: diagram norm estimates}
\begin{split}
\sup_{k > 0}\EE \Big[ &\| \<1>_k \|^p_{B^{-\frac 12 -\kappa}_{p',q}} + \| \<2>_k \|^p_{B^{-1-\kappa}_{p',q}} + \| \<30>_k \|^p_{B^{\frac 12-\kappa}_{p',q}}  
\\
&+ 
\| \<31>_k \|^p_{B^{-\kappa}_{p',q}} + \| \<32>_k \|^p_{B^{-\frac 12-\kappa}_{p',q}} + \Big(\int_0^k \| \<202>_{k'} \|_{B^{-\kappa}_{p',q}}\Big)^p dk' \Big]
\leq
C.
\end{split}
\end{equs}
\end{prop}

\begin{proof}
See \cite[Lemma 24]{BG19}.	
\end{proof}

\begin{rem} \label{rem: stochastic norms}
The constant on the righthand side of \eqref{eq: diagram norm estimates} is independent of $N$ because our Besov spaces are defined with respect to normalised Lebesgue measure $\dbar x = \frac{dx}{N^3}$ (see Appendix \ref{appendix: sub: besov}). For $p=\infty$, bounds that are uniform in $N$ do not hold. Indeed, for $L^\infty$-based norms, there is in general no chance of controlling space-stationary processes uniformly in the volume. Thus, we cannot work in Besov-H\"older spaces.  
\end{rem}

We prove the bound in \eqref{eq: diagram norm estimates} for $\<30>_k$ since it illustrates the role of $\cJ_k$, is used later in the proof of Proposition \ref{prop: testing phi4}, and gives the reader a flavour of how to prove the bounds on the other diagrams.

\begin{lem}\label{lem:trident}
	There exists $C=C(\eta)>0$ such that, for any $n \in (N^{-1}\ZZ)^3$,
	\begin{equs} \label{eq: trident fourier bound}
	\sup_{k > 0}\EE \Big|\cF \<30>_k(n)\Big|^2	
	\leq
	\frac{CN^3}{\langle n \rangle^4}.
	\end{equs}

	As a consequence, for every $p \in [1,\infty)$ and $s < \frac 12$, there exists $C=C(p,s,\eta)>0$ such that
	\begin{equs} 
		\sup_{k>0}\EE \Big[ \| \<30>_K \|^p_{B^s_{p,p}} \Big]
		\leq
		C.
	\end{equs}
\end{lem}

\begin{proof}
Inserting \eqref{eq: iterated integral third wick} in the definition of $\<30>_k$ and switching the order of integration,
\begin{equs}
\cF \<30>_k(n)
&=	
\frac 6{N^\frac 32} \sum_{n_1 + n_2 + n_3 = n} \int_0^k \frac{\partial_{k'}\rho_{k'}^2( n )}{\langle n \rangle^2} \int_0^{k'} \int_0^{k_1} \int_0^{k_2} 
\\
&\quad\quad\quad\quad\quad\quad\quad\quad\quad
 \times 
 \Bigg( \prod_{i=1}^3 \frac{\sqrt{\partial_{k_i} \rho_{k_i}^2( n_i )}}{\langle n_i \rangle} \Bigg) dB_{k_3}^{n_3} dB_{k_2}^{n_2} dB_{k_1}^{n_1} dk'
\\
&=
\frac 6{N^\frac 32}\sum_{n_1 + n_2 + n_3 = n} \int_0^k \int_0^{k_1} \int_0^{k_2} \Bigg( \int_{k_1}^k \frac{\partial_{k'} \rho_{k'}^2(n)}{\langle n \rangle^2} dk' \Bigg)
\\
&\quad\quad\quad\quad\quad\quad\quad\quad\quad
\times\Bigg(\prod_{i=1}^3 \frac{\sqrt{\partial_{k_i}\rho_{k_i}^2(n_i)}}{\langle n_i \rangle} \Bigg) dB_{k_3}^{n_3} dB_{k_2}^{n_2} dB_{k_1}^{n_1}.
\end{equs}

Therefore, by It\^o's formula,
\begin{equs}\label{eq:trident1}
\begin{split}
\EE \Big| &\cF \<30>_K (n) \Big|^2
\\
&\leq
\frac {36}{N^3}\sum_{n_1 + n_2 + n_3 = n} \int_0^k \int_0^{k_1} \int_0^{k_2} \Bigg( \int_{k_1}^k \frac{\partial_{k'} \rho_{k'}^2( n )}{\langle n \rangle^2} dk' \Bigg)^2 
\\
&\quad\quad\quad\quad\quad\quad\quad\quad\quad
\times \Bigg(\prod_{i=1}^3 \frac{\partial_{k_i}\rho_{k_i}^2(n_i)}{\langle n_i \rangle^2}\Bigg) dk_3 dk_2 dk_1
\\
&\leq
\frac {36}{N^3}\sum_{n_1 + n_2 + n_3 = n} \frac{1}{\langle n_2 \rangle^2 \langle n_3 \rangle^2} \int_0^k \Bigg( \int_{k_1}^k \frac{\partial_{k'} \rho_{k'}^2(n)}{\langle n \rangle^2} dk' \Bigg)^2 \frac{\partial_{k_1}\rho_{k_1}^2(n_1)}{\langle n_1 \rangle^2} dk_1
\end{split}
\end{equs}
where we have performed the $k_2$ and $k_3$ integrations, and used that $|\rho_k| \leq 1$.

Recall that $\partial_{k'} \rho_{k'}^2$ is supported on frequencies $| n | \in (c_\rho k', C_\rho k')$. Hence, for any $\kappa > 0$,
\begin{equs} \label{eq:trident2}
\begin{split}
\eqref{eq:trident1}
&\lesssim
\frac 1{N^3}\sum_{n_1 + n_2 + n_3 = n} \frac{1}{\langle n_2 \rangle^2 \langle n_3 \rangle^2}
 \intt \Bigg( \int_{k_1}^k \frac{\partial_{k'} \rho_{k'}^2(n)}{\langle n \rangle^{2-\frac\kappa2}} dk' \Bigg)^2 \frac{\partial_{k_1}\rho_{k_1}^2(n_1)}{\langle n_1 \rangle^2 \langle k_1 \rangle^{\kappa}} dk_1
\\
&\leq
\frac 1{N^3}\sum_{n_1 + n_2 + n_3 = n} \frac{1}{\langle n \rangle^{4-\kappa}\langle n_2 \rangle^2 \langle n_3 \rangle^2} \int_0^k \frac{\partial_{k_1}\rho_{k_1}^2(n_1)}{\langle n_1 \rangle^{2+\kappa}} dk_1 
\\
&\lesssim
\frac 1{N^3}\sum_{n_1 + n_2 + n_3 = n} \frac{1}{\langle n \rangle^{4-\kappa}\langle n_1 \rangle^{2+\kappa} \langle n_2 \rangle^2 \langle n_3 \rangle^2} 
\lesssim 
\frac{N^3}{\langle n \rangle^4},
\end{split}
\end{equs}
where $\lesssim$ means $\leq$ up to a constant depending only on $\eta$, $c_\rho$ and $C_\rho$; the last inequality uses standard bounds on discrete convolutions contained in Lemma \ref{lem: discrete convolution bound}; and we have used that the double convolution produces a volume factor of $N^6$. Note that, as said in Section \ref{sec: model}, we omit the dependence on $c_\rho$ and $C_\rho$ in the final bound.

By Fubini's theorem, Nelson's hypercontractivity estimate \cite{N73} (or the related Burkholder-Davis-Gundy inequality \cite[Theorem 4.1]{RY13}), and space-stationarity
\begin{equs} \label{eq: trident 3}
\begin{split}
\EE \| \<30>_k \|_{B^s_{p,p}}^p
&=
\sum_{j \geq -1} 2^{jps} \EE \|\Delta_j \<30>_k \|_{L^p}^p
\\
&=
\sum_{j \geq -1} 2^{jps} \intx \EE |\Delta_j \<30>_k(x)|^p \dbar x
\\
&\lesssim
\sum_{j \geq -1} 2^{jps} \intx \Big( \EE |\Delta_j \<30>_k(x)|^2 \Big)^\frac p2 \dbar x
\\
&=
\sum_{j \geq -1} 2^{jps} \Big( \EE |\Delta_j \<30>_k(0)|^2 \Big)^\frac p2
\end{split}
\end{equs}
where $\Delta_j$ is the $j$-th Littlewood-Paley block defined in Appendix \ref{appendix: toolbox} and we recall $\dbar x = \frac{dx}{N^3}$.

We overload notation and also write $\Delta_j$ to mean its corresponding Fourier multiplier. Then, by space-stationarity, for any $j \geq -1$,
\begin{equs} \label{eq: trident 4}
\begin{split}
\EE |\Delta_j \<30>_k(0)|^2 
&=
\intx \EE |\Delta_j \<30>_k(x)|^2 \dbar x
\\
&=
\frac 1{N^6} \sum_{n} |\Delta_j(n)|^2 \EE | \cF \<30>_k (n)|^2
\\
&\lesssim
\frac 1{N^3} \sum_{n} \frac{\Delta_j(n)^2}{\langle n \rangle^4}
\lesssim
\frac{2^{3j}}{2^{4j}}
=
\frac{1}{2^j}.
\end{split}
\end{equs}

Inserting \eqref{eq: trident 4} into \eqref{eq: trident 3} we obtain
\begin{equs}
\EE \| \<30>_K \|_{B^s_{p,p}}^p
\lesssim
\sum_{j \geq -1} 2^{jps} 2^{-\frac p2 j}	
\end{equs}
which converges provided $s < \frac 12$, thus finishing the proof.
\end{proof}

\subsection{The Bou\'e-Dupuis formula}

Fix an ultraviolet cutoff $K$. Recall that we are interested in Gaussian expectations of the form
\begin{equs}
\EE_N e^{-\cH(\phi_K)}	
\end{equs}
where $\cH(\phi_K) = \cH_{\beta,N,K}(\phi_K) + Q_K(\phi_K)$. 

We may represent such expectations on $(\Omega, \cA, (\cA_k)_{k \geq 0}, \PP)$:
\begin{equs} \label{eq: expectation representation}
\EE_N e^{-\cH(\phi_K)}	
=
\EE e^{-\cH(\<1>_K)}.
\end{equs}
The key point is that the righthand side of \eqref{eq: expectation representation} is written in terms of a measurable functional of Brownian motions. This allows us to exploit continuous time martingale techniques, crucially Girsanov's theorem \cite[Theorems 1.4 and 1.7]{RY13}, to reformulate \eqref{eq: expectation representation} as a stochastic control problem.

Let $\HH$ be the set of processes $v_\bullet$ that are $\PP$-almost surely in $L^2(\RR_+;L^2(\TTN))$ and progressively measurable with respect to $(\cA_k)_{k \geq 0}$. We call this the space of drifts. For any $v \in \HH$, let $V_\bullet$ be the process defined by
\begin{equs}
V_k = \int_0^k \cJ_{k'} v_{k'} d{k'}.
\end{equs} 
For our purposes, it is sufficient to consider the subspace of drifts $\HH_K \subset \HH$ consisting of $v \in \HH$ such that $v_k = 0$ for $k > K$.

We also work with the subset of bounded drifts $\HH_{b,K} \subset \HH_K$, defined as follows: for every $M \in \NN$, let $\HH_{b,M,K}$ be the set of $v \in \HH_K$ such that
\begin{equs}\label{eq: almost sure l2 boundedness}
\int_0^K \intx	v_k^2  dx dk 
\leq
M
\end{equs}
$\PP$-almost surely. Set $\HH_{b,K} = \bigcup_{M \in \NN} \HH_{b,M,K}$.

The following proposition is the main tool of this section.
\begin{prop}\label{prop: bd}
Let $N \in \NN$ and $\cH:C^\infty(\TTN) \rightarrow \RR$ be measurable and bounded. Then, for any $K > 0$,
\begin{equs} \label{eq: boue dupuis}
	- \log \EE \Big[ e^{-\cH(\<1>_K)} \Big] 
	= 
	\inf_{v} \EE \Big[ \cH(\<1>_K + V_K) + \frac 12 \int_0^K \intx v_k^2 dx dk \Big]
\end{equs}
where the infimum can be taken over $v$ in $\HH_K$ or $\HH_{b,K}$. 
\end{prop}

\begin{proof}
\eqref{eq: boue dupuis} was first established by Bou\'e and Dupuis \cite{BD98}, but we use the version in \cite[Theorem 8.3]{BDu19}, adapted to our setting.
\end{proof}

We cannot directly apply Proposition \ref{prop: bd} for the case $\cH = \cH_{\beta,N,K} + Q_K$ because it is not bounded. To circumvent this technicality, we introduce a \textit{total energy cutoff} $E \in \NN$. Since $K$ is taken \textit{fixed}, $\cH_{\beta,N,K} + Q_K$ is bounded from below. Hence, by dominated convergence
\begin{equs} \label{eq: total energy cutoff approx}
\lim_{E \rightarrow \infty}\EE_N e^{-\big(\cH_{\beta,N,K}(\phi_K) + Q_K(\phi_K)\big) \wedge E}
=
\EE_N e^{- \cH_{\beta,N,K}(\phi_K) + Q_K(\phi_K)}.	
\end{equs}

We apply Proposition \ref{prop: bd} to $\cH = \big(\cH_{\beta,N,K} + Q_K\big) \wedge E$. For the lower bound on the corresponding variational problem, we establish estimates that are uniform over $v \in \HH_{b,K}$. For the upper bound, we establish estimates for a specific choice of $v \in \HH_K$ which is constructed via a fixed point argument. All estimates that we establish are independent of $E$. Hence, using \eqref{eq: total energy cutoff approx} and the representation \eqref{eq: expectation representation}, they carry over to $\EE_N e^{-\cH_{\beta,N,K}(\phi_K) + Q_K(\phi_K)}$. We suppress mention of $E$ unless absolutely necessary.

\begin{rem} \label{rem: local martingale issue}
The assumption that $\cH$ is bounded allows the infimum in \eqref{eq: boue dupuis} to be interchanged between $\HH_K$ and $\HH_{b,K}$. The use of $\HH_{b,K}$ allows one to overcome subtle stochastic analysis issues that arise later on: specifically, justifying certain stochastic integrals appearing in Lemmas \ref{lem: perturb 1} and \ref{lem: perturb 2} are martingales and not just local martingales. See Lemma \ref{lem: martingales}. The additional boundedness condition is important in the lower bound on the variational problem as the only other a priori information that we have on $v$ there is that $\EE \intt \intx v_k^2 dx dk < \infty$, which alone is insufficient. On the other hand, the candidate optimiser for the upper bound is constructed in $\HH_K$, but it has sufficient moments to guarantee the aforementioned stochastic integrals in Lemma \ref{lem: martingales} are martingales. See Lemma \ref{lem: bg drift}.
\end{rem}

\begin{rem}
	A version of the Bou\'e-Dupuis formula for $\cH$ measurable and satisfying certain integrability conditions is given in \cite[Theorem 7]{U14}. These integrability conditions are broad enough to cover the cases that we are interested in, and it is required in \cite{BG19} to identify the Laplace transform of $\phi^4_3$. However, it is not clear to us that the infimum in the corresponding variational formula can be taken over $\HH_{b,K}$. Therefore, it seems that the stochastic analysis issues discussed in Remark \ref{rem: local martingale issue} cannot be resolved directly using this version without requiring some post-processing (e.g. via a dominated convergence argument with a total energy cutoff as above).
\end{rem}

\subsubsection{Relationship with the Gibbs variational principle}

Given a drift $v \in \HH_K$, we define the measure $\QQ$ whose Radon-Nikodym derivative with respect to $\PP$ is given by the following stochastic exponential:
\begin{equs}\label{eq: stochastic exponential}
d\QQ
=
e^{\intt v_k dX_k - \frac 12 \intx \intt v_k^2 dk dx }	d\PP.
\end{equs}
Let $\HH_{c,K}$ be the set of $v \in \HH_K$ such that its associated measure defined in \eqref{eq: stochastic exponential} is a probability measure, i.e. the expectation of the stochastic integral is $1$. Then, by Girsanov's theorem \cite[Theorems 1.4 and 1.7 in Chapter VIII]{RY13} it follows that the process $X_\bullet$ is a semi-martingale under $\QQ$ with decomposition:
\begin{equs}
	X_K
	=
	X_K^v + \int_0^K v_k dx
\end{equs}
where $X^v_\bullet$ is an $L^2$ cylindrical Brownian motion with respect to $\QQ$. This induces the decomposition
\begin{equs}\label{eq: girsanov decomp}
\<1>_K
=
\<1>_K^v + V_K	
\end{equs}
where $\<1>_K^v = \int_0^K \cJ_k dX^v_k$. 

\begin{lem}\label{lem: gibbs var principle}
Let $N \in \NN$ and $\cH:C^\infty(\TTN) \rightarrow \RR$ be measurable and bounded from below. Then, for any $K > 0$,
\begin{equs}\label{eq: gibbs var principle 0}
- \log \EE e^{-\cH(\<1>_K)}
=
\min_{v \in \HH_{c,K}} \EE_{\QQ} \Big[ \cH(\<1>_K^v + V_K) + \frac 12 \int_0^\infty \intx v_k^2 dx dk  \Big]
\end{equs}
where $\EE_\QQ$ denotes expectation with respect to $\QQ$.
\end{lem}

\begin{proof}
\sloppy \eqref{eq: gibbs var principle 0} is a well-known representation of the classical Gibbs variational principle \cite[Proposition 4.5.1]{DE11}. Indeed, one can verify that $R(\QQ \| \PP) = \EE_{\QQ} \Big[ \int_0^\infty \intx v_k^2 dx dk\Big]$, where $R(\QQ \| \PP) = \EE_{\QQ} \log \frac{d\QQ}{d\PP}$ is the relative entropy of $\QQ$ with respect to $\PP$. A full proof in our setting is given in \cite[Proposition 4.4]{GOTW18}.
\end{proof}

 Proposition \ref{prop: bd} has several upshots over Lemma \ref{lem: gibbs var principle}. The most important for us is that drifts can be taken over a Banach space, thus allowing candidate optimisers to be constructed using fixed point arguments via contraction mapping. In addition, the underlying probability space is fixed (i.e. with respect to the canonical measure $\PP$), although this is a purely aesthetic advantage in our case. The cost of these upshots is that the minimum in \eqref{eq: gibbs var principle 0} is replaced by an infimum in \eqref{eq: boue dupuis}, and more rigid conditions on $\cH$ are required. We refer to \cite[Section 8.1.1]{BDu19} or \cite[Remark 1]{BG19} for further discussion. 
 
With the connection with the Gibbs variational principle in mind, we call $\cH(V_K)$ the drift (potential) energy and we call $\intt \intx v_k^2 dx dk$ the drift entropy.

\subsubsection{Regularity of the drift} \label{subsec: flat}

In our analysis we use intermediate scales between $0$ and $K$. As we explain in Section \ref{subsec: strategy}, this means that we require control over the process $V_\bullet$ in terms of the drift energy and drift entropy terms in \eqref{eq: boue dupuis}. 

The drift entropy allows a control of $V_\bullet$ in $L^2$-based topologies.
\begin{lem}\label{lem: drift bound}
For every  $v \in L^2(\RR_+;L^2(\TTN))$ and $K > 0$,
\begin{equs}\label{eq: l2 drift bound}
\sup_{0 \leq k \leq K} \|V_k \|^2_{H^1}
\leq
\intt \intx v_k^2 \dbar x dk.	
\end{equs}
\end{lem}

\begin{proof}
\eqref{eq: l2 drift bound} is a straightforward consequence definition of $\cJ_k$, see \cite[Lemma 2]{BG19}.
\end{proof}

To control the homogeneity in our estimates, we also require bounds on $\|V_\bullet\|_{L^4}^4$. This is a problem: for our specific choices of $\cH$, the drift energy allows a control in $L^4$-based topologies at the \textit{endpoint} $V_K$. It is in general impossible to control the history of the path by the endpoint (for example, consider an oscillating process $V_\bullet$ with $V_K = 0$). We follow \cite{BG19} to sidestep this issue.

Let $\tilde\rho \in C^\infty_c(\RR_+;\RR_+)$ be non-increasing such that
\begin{equs}
\tilde\rho(x) =
\begin{cases}
	1 & \quad |x| \in \Big[ 0, \frac{c_\rho} 2 \Big] \\
    0 & \quad |x| \in \Big[ c_\rho, \infty \Big)	
\end{cases}
\end{equs}
and let $\tilde\rho_k(\cdot) = \tilde\rho(\frac{\cdot}{k})$ for every $k>0$.

Define the process $V^\flat_\bullet$ by
\begin{equs}
V_k^\flat
&=
\frac{1}{N^3} \sum_{n} \tilde\rho_k(n) \Bigg(\int_0^k \cJ_{k'}(n) \cF{v_{k'}}(n) dk'\Bigg)e^n.
\end{equs}
Note that $\cF(V_k^\flat)(n) = \cF(V_k)(n)$ if $|n| \leq \frac {c_\rho} 2$. Thus, $V^\flat_\bullet$ and $V_\bullet$ have the same low frequency/large-scale behaviour (hence the notation).

The two processes differ on higher frequencies/small-scales. Indeed, as a Fourier multiplier, $\tilde\rho_k \cJ_k = 0$ for $k' > k$. Hence, for any $k \leq K$,
\begin{equs}
V_k^\flat
=
\frac{1}{N^3} \sum_{n} \tilde\rho_k(n) \Bigg(\int_0^K \cJ_{k'}(n) \cF{v_{k'}}(n) dk'\Bigg)e^n
=
\tilde\rho_k V_K.
\end{equs}
This is sufficient for our purposes because $\tilde\rho_k$ is an $L^p$ multiplier for $p \in (1,\infty)$, and hence the associated operator is $L^p$ bounded for $p \in (1,\infty)$.

\begin{lem} \label{lem: flat drift bound}
	For any $p \in (1,\infty)$, there exists $C=C(p,\eta)>0$ such that, for every $v \in L^2(\RR_+;L^2(\TTN))$,
	\begin{equs} \label{eq: a priori flat drift}
	\sup_{0 \leq k \leq K}\|V_k^\flat\|_{L^p}
	\leq 
	C\|V_K\|_{L^p}.
	\end{equs}
	
	Moreover, for any $s,s' \in \RR$, $p \in (1,\infty)$, $q \in [1,\infty]$, there exists $C=C(s,s',p,q,\eta)$ such that, for every $v \in L^2(\RR_+;L^2(\TTN))$,
	\begin{equs} \label{eq: a priori derivative drift}
		\sup_{0 \leq k \leq K}\| \partial_k V_k^\flat \|_{B^{s'}_{p,q}} 
		\leq 
		C \frac{\|V_K\|_{B^s_{p,q}}}{\langle k \rangle^{1+s-s'}}.
	\end{equs}
\end{lem}

\begin{proof}
\eqref{eq: a priori flat drift} and \eqref{eq: a priori derivative drift} are a consequence of the preceding discussion together with the observation that $\partial_k V_k^\flat$ is supported on an annulus in Fourier space and, subsequently, applying Bernstein's inequality \eqref{tool: bernstein annulus}. See \cite[Lemma 20]{BG19}.
\end{proof}

\section{Estimates on $Q$-random variables} \label{sec: free energy}

The main results of this section are upper bounds on expectations of certain random variables, derived from $Q_1,Q_2,$ and $Q_3$ defined in \eqref{def: Q_i}, that are uniform in $\beta$ and extensive in $N^3$.  

\begin{prop} \label{prop: q bound main}
For every $a_0 > 0$, there exist $\beta_0 = \beta_0(a_0,\eta) \geq 1$ and $C_Q = C_Q(a_0, \beta_0, \eta)>0$ such that the following estimates hold: for all $\beta > \beta_0$ and $a \in \RR$ satisfying $|a| \leq a_0$,
\begin{equs}
- \frac{1}{N^3} \log \Big\langle \prod_{\sBox \in \BBN} \exp (aQ_1(\Box)) \Big\rangle_{\beta,N}
&\geq
-C_Q
\\
- \frac{1}{N^3} \log \Big\langle \prod_{\sBox \in \BBN} \exp (aQ_2(\Box)) \Big\rangle_{\beta,N}
&\geq
-C_Q.
\end{equs}

In addition,
\begin{equs}
- \frac{1}{N^3} \log\Big\langle \prod_{ \{\sBox, \sBox'\} \in B} \exp (|aQ_3(\Box,\Box')|) \Big\rangle_{\beta,N}
\geq
-C_Q	
\end{equs}
where $B$ is any set of unordered pairs of nearest-neighbour blocks that partitions $\BBN$. 
\end{prop}

\begin{proof}
See Section \ref{subsec: proof of q bound}. 
\end{proof}

Proposition \ref{prop: q bound main} is used in Section \ref{subsec: proof cosh prop}, together with the chessboard estimates of Proposition \ref{prop:chessboard_estimates}, to prove Proposition \ref{prop: cosh}. Indeed, chessboard estimates allow us to obtain estimates on expectations of random variables, derived from the $Q_i$, that are extensive in their support \textit{from} estimates that are extensive in $N^3$. Note that the latter are significantly easier to obtain than the former since these random variables may be supported on arbitrary unions of blocks.

\begin{rem} \label{rem: eta dependence final}
For the remainder of this section, we assume $\eta < \frac{1}{392 C_P}$ where $C_P$ is the Poincar\'e constant on unit blocks (see Proposition \ref{prop:poincare_blocks}). This is for convenience in the analysis of Sections \ref{subsec: proof z lower} and \ref{subsec: proof of q bound} (see also Lemma \ref{lem:pointwise}). Whilst this may appear to fix the specific choice of renormalisation constants $\delta m^2$, we can always shift into this regime by absorbing a finite part of $\delta m^2$ into $\cV_\beta$.
\end{rem}

Most of the difficulties in the proof of Proposition \ref{prop: q bound main} are contained in obtaining the following upper and lower bounds on the free energy $-\log\sZ_{\beta,N}$ that are uniform in $\beta$.

\begin{prop}\label{prop:zbounds}
There exists $C=C(\eta)>0$ such that, for all $\beta \geq 1$, 
\begin{equs} \label{eq:zboundlower}
\liminf_{K \rightarrow \infty} -\frac{1}{N^3}\log\sZ_{\beta,N,K} \geq - C	
\end{equs}
and
\begin{equs} \label{eq:zboundupper}
\limsup_{K \rightarrow \infty} -\frac{1}{N^3} \log \sZ_{\beta,N,K} \leq C.
\end{equs}
\end{prop}

\begin{proof}
See Sections \ref{subsec: proof z lower} and \ref{subsec: proof upper bound} for a proof of \eqref{eq:zboundlower} and \eqref{eq:zboundupper}, respectively. These proofs rely on Sections \ref{sec: translation} - \ref{subsec: pointwise bound on effective}, and the overall strategy is sketched in Section \ref{subsec: strategy}.
\end{proof}

\begin{rem} \label{rem: beta dependence}
In \cite{BG19} estimates on $-\log\sZ_{\beta,N,K}$ are obtained that are uniform in $K > 0$ and extensive in $N^3$. However, one can show that these estimates are $O(\beta)$ as $\beta \rightarrow \infty$. This is insufficient for our purposes (compare with the uniform in $\beta$ estimates required to prove Proposition \ref{prop:peierls}). 
\end{rem}

\subsection{Strategy to prove Proposition \ref{prop:zbounds}}\label{subsec: strategy}

The lower bound on $-\log\sZ_{\beta,N,K}$, given by  \eqref{eq:zboundlower}, is the harder bound to establish in Proposition \ref{prop:zbounds}. Our approach builds on the analysis of \cite{BG19} by incorporating a low temperature expansion inspired by \cite{GJS76-4,GJS76-3}. This is explained in more detail in Section \ref{subsec: fixing beta}. 

On the other hand, we establish the upper bound on $-\log\sZ_{\beta,N,K}$, given by \eqref{eq:zboundupper}, by a more straightforward modification of the analysis in \cite{BG19}. See Section \ref{subsec: proof upper bound}.

We now motivate our approach to establishing \eqref{eq:zboundlower} by first isolating the the difficulty in obtaining $\beta$-independent bounds when using \cite{BG19} straight out of the box. The starting point is to apply Proposition \ref{prop: bd} with $\cH = \cH_{\beta,N,K}$, together with a total energy cutoff that we refrain from making explicit (see Remark \ref{rem: local martingale issue} and the discussion that precedes it), to represent $-\log\sZ_{\beta,N,K}$ as a stochastic control problem. 

For every $v \in \HH_{b,K}$, define
\begin{equs}
\Psi_K(v) 
= 
\cH_{\beta,N,K}(\<1>_K + V_K) + \frac 12 \intt \intx v_k^2 dx dk.
\end{equs}
Ultraviolet divergences occur in the expansion of $\cH_{\beta,N,K}(\<1>_K+V_K)$ since the integrals $\intx \<3>_K V_K dx$ and $\intx \<2>_K V_K^2 dx$ appear and cannot be bounded uniformly in $K$:
\begin{itemize}
\item For the first integral, there are difficulties in even interpreting $\<3>_K$ as a random distribution in the limit $K \rightarrow \infty$. Indeed, the variance of $\<3>_K$ tested against a smooth function diverges as the cutoff is removed.
\item On the other hand, one can show that $\<2>_K$ does converge as $K \rightarrow \infty$ to a random distribution of Besov-H\"older regularity $-1-\kappa$ for any $\kappa > 0$ (see Proposition \ref{prop: diagrams}). However, this regularity is insufficient to obtain bounds on the second integral uniform on $K$ . Indeed, $V_K$ can be bounded in at most $H^1$ uniformly in $K$ (see Lemma \ref{lem: drift bound}), and hence we cannot test $\<2>_K$ against $V_K$ (or $V_K^2$) in the limit $K \rightarrow \infty$.
\end{itemize}
This is where the need for renormalisation beyond Wick ordering appears.

To implement this, we follow \cite{BG19} and postulate that the small-scale behaviour of the drift $v$ is governed by explicit renormalised polynomials of $\<1>$ through the change of variables:
\begin{equs} \label{eq: bg ansatz}
v_k 
=
- \frac 4\beta \cJ_k \<3>_k - \frac {12}\beta \cJ_k(\<2>_k \pg V_k^\flat) + r_k	
\end{equs}
where the remainder term $r=r(v)$ is defined by \eqref{eq: bg ansatz}. Since $v \in \HH_K \supset \HH_{b,K}$, we have that $r \in \HH_K$ and, hence, has finite drift entropy; however, note that $r \not\in \HH_{b,K}$. The optimisation problem is then changed from optimising over $v \in \HH_{b,K}$ to optimising over $r(v) \in \HH_K$. 

The change of variables \eqref{eq: bg ansatz} means that the drift entropy of any $v$ now contains terms that are divergent as $K \rightarrow \infty$. One uses It\^o's formula to decompose the divergent integrals identified above into intermediate scales, and then uses these divergent terms in the drift entropy to mostly cancel them. Using the renormalisation counterterms beyond Wick ordering (i.e. the terms involving $\gamma_K$ and $\delta_K$), the remaining divergences can be written in terms of well-defined integrals involving the diagrams $\Xi = (\<1>, \<2>, \<30>, \<31>, \<32>, \<202>)$ defined in Section \ref{subsec: diagrams}.

One can then establish that, for every $\varepsilon>0$, there exists $C=C(\varepsilon,\beta,\eta)>0$ such that, for every $v \in \HH_{b,K}$,
\begin{equs}\label{eq: bg bd estimate}
\EE \Psi_K(v)
\geq
-CN^3 + (1-\varepsilon)\EE [G_K(v)]	
\end{equs}
where
\begin{equs}
G_K(v)
=
\frac 1\beta \intx V_K^4 dx + \frac 12 \intt \intx r_k^2 dk dx
\geq 
0.
\end{equs}
The quadratic term in $G_K(v)$ allows one to control the $H^{\frac 12 -\kappa}$ norm of $V_K$ for any $\kappa > 0$, uniformly in $K$ (see Proposition \ref{prop: u and v bound}). These derivatives on $V_K$ appear when analysing terms in $\Psi_K(v)$ involving Wick powers of $\<1>_K$ tested against (powers of) $V_K$. However, some of these integrals have quadratic or cubic dependence on the drift, thus the quadratic term in $G_K(v)$ is insufficient to control the homogeneity in these estimates; instead, this achieved by using the quartic term in $G_K(v)$. Note that the good sign of the quartic term in the $\cH_{\beta,N,K}$ ensures that $G_K(v)$ is indeed non-negative. 

Using the representation \eqref{eq: expectation representation} on $\EE_N e^{-\cH_{\beta,N,K}(\phi_K)}$ and applying Proposition \ref{prop: bd}, one obtains $-\log\sZ_{\beta,N,K} \geq - CN^3$ from \eqref{eq: bg bd estimate} and the positivity of $G_K(v)$. 

As pointed out in Remark \ref{rem: beta dependence}, this argument gives $C=O(\beta)$ for $\beta$ large and this is insufficient for our purposes. The suboptimality in $\beta$-dependence comes from the treatment of the integral
\begin{equs}\label{eq: bad v integral}
\intx \cV_\beta(V_K) - \frac\eta 2 V_K^2 dx
\end{equs}
in $\cH_{\beta,N,K}(\<1>_K + V_K)$. The choice of $G_K(v)$ in the preceding discussion is not appropriate in light of \eqref{eq: bad v integral} since the term $\intx V_K^4 dx$ destroys the structure of the non-convex potential $\intx \cV_\beta(V_K) dx$. On the other hand, replacing $\frac 1\beta \intx V_K^4 dx$ with the whole integral $\eqref{eq: bad v integral}$ in $G_K(v)$ does not work. This is because \eqref{eq: bad v integral} does not admit a $\beta$-independent lower bound.

\subsubsection{Fixing $\beta$ dependence via a low temperature expansion}\label{subsec: fixing beta}

We expand \eqref{eq: bad v integral} as two terms
\begin{equs}\label{eq: splitting bad v integral}
\eqref{eq: bad v integral} 
=
\intx \frac 12 \cV_\beta(V_K) dx	+ \intx \frac 12 \cV_\beta(V_K) - \frac \eta 2 V_K^2 dx.
\end{equs}

The first integral in \eqref{eq: splitting bad v integral} is non-negative so we use it as a stability/good term for the deterministic analysis, i.e. replacing $G_K(v)$ by
\begin{equs} \label{eq: a better coercive term}
\intx \frac 12 \cV_\beta(V_K) dx + \frac 12 \intt \intx r_k^2 dx dk.
\end{equs}
This requires a comparison of $L^p$ norms of $V_K$ for $p \leq 4$ on the one hand, and $\intx \cV_\beta(V_K) dx$ on the other. Due to the non-convexity of $\cV_\beta$, this produces factors of $\beta$; these have to be beaten by the good (i.e. negative) powers of $\beta$ appearing in $\cH_{\beta,N,K}(\<1>_K + V_K)$. We state the required bounds in the following lemma.

\begin{lem}\label{lem: potential}
For any $p\in[1,4]$, there exists $C=C(p) > 0$ such that, for all $a \in \RR$,
\begin{equs}\label{tool: potential bound}
|a|^p 
\leq 
C (\hfbeta)^\frac p2 \cV_\beta(a)^\frac p4 + C (\hfbeta)^p. 	
\end{equs}
Hence, for any $f \in C^\infty(\TTN)$,
\begin{equs}\label{tool: potential bound norm}
\begin{split}
\| f \|_{L^p}
&\leq
C (\hfbeta)^\frac 12 \Big( \intx \cV_\beta(f) \dbar x \Big)^\frac 14 + C\hfbeta
\end{split}
\end{equs}
where we recall $\dbar x = \frac{dx}{N^3}$.
\end{lem}

\begin{proof}
\eqref{tool: potential bound} follows from a straightforward computation. \eqref{tool: potential bound norm} follows from using \eqref{tool: potential bound} and Jensen's inequality.	
\end{proof}

The difficulty lies in bounding the second integral in \eqref{eq: splitting bad v integral} uniformly in $\beta$. In 2D an analogous problem was overcome in \cite{GJS76-4, GJS76-3} in the context of a low temperature expansion for $\Phi^4_2$. Those techniques rely crucially on the logarithmic ultraviolet divergences in 2D, and the mutual absolute continuity between $\Phi^4_2$ and its underlying Gaussian measure. Thus, they do not extend to 3D. However, we use the underlying strategy of that low temperature expansion in our approach. 

We write $\sZ_{\beta,N,K}$ as a sum of $2^{N^3}$ terms, where each term is a modified partition function that enforces the block averaged field to be either positive or negative on blocks. For each term in the expansion, we change variables and shift the field on blocks to $\pm \hfbeta$ so that the new mean of the field is small. We then apply Proposition \ref{prop: bd} to each of these $2^{N^3}$ terms. 

We separate the scales in the variational problem by coarse-graining the resulting Hamiltonian. Large scales are captured by an effective Hamiltonian, which is of a similar form to the second integral in \eqref{eq: splitting bad v integral}. We treat this using methods inspired by \cite[Theorem 3.1.1]{GJS76-3}: the expansion and translation allow us to obtain a $\beta$-independent bound on the effective Hamiltonian with an error term that depends only on the difference between the field and its block averages (the \textit{fluctuation field}). The fluctuation field can be treated using the massless part of the underlying Gaussian measure (compare with \cite[Proposition 2.3.2]{GJS76-3}). 

The remainder term contains all the small-scale/ultraviolet divergences and we renormalise them using the pathwise approach of \cite{BG19} explained above. Patching the scales together requires uniform in $\beta$ estimates on the error terms from the renormalisation procedure using an analogue of the stability term \eqref{eq: a better coercive term} that incorporates the translation, and Lemma \ref{lem: potential}.

\subsection{Expansion and translation by macroscopic phase profiles}\label{sec: translation}
Let $\chi_+,\chi_-:\RR \rightarrow \RR$ be defined as
\begin{equs}
    \chi_+(a) 
    = 
    \frac{1}{\sqrt \pi} \int_{-a}^\infty e^{-c^2} dc,
    \quad
    \chi_-(a) = \chi_+(-a).
\end{equs}
They satisfy
\begin{equs}
    \chi_+(a) + \chi_-(a) = 1
\end{equs}
and hence
\begin{equs}
\sum_{\sigma \in \{\pm 1\}^\BBN} \prod_{\sBox \in \BBN} \chi_{\sigma(\sBox)}\big(\phi(\Box)\big) 
=
1
\end{equs}
for any $\vphi = (\phi(\Box))_{\sBox \in \BBN} \in \RR^\BBN$.

For any $K>0$, we expand
\begin{equs} \label{partition expansion}
\begin{split}
    \sZ_{\beta,N,K} 
    &=
    \sum_{\sigma \in \{\pm 1\}^{\BBN}} \EE_N \Big[ e^{-\cH_{\beta,N,K}} \prod_{\sBox \in \BBN}\chi_{\sigma(\sBox)}\big(\phi_K(\Box)\big)  \Big]
    \\
    &= 
    \sum_{\sigma \in \{\pm 1\}^{\BBN}} \sZ_{\beta,N,K}^\sigma
  \end{split}
\end{equs}
where we recall $\phi_K(\Box) = \intx \phi_K \1_\sBox dx$. 

We fix $\sigma$ in what follows and sometimes suppress it from notation. Let $h = \hfbeta \sigma$. We then have
\begin{equs}
\sZ_{\beta,N,K}^\sigma
&=
\EE_N \exp \Bigg( - \intx :\cV_\beta(\phi_K): - \frac{\gamma_K}{\beta^2}:\phi_K^2: - \delta_K  
\\
&\quad\quad
 - \frac\eta2 :(\phi_K-h)^2: -\eta \phi_K h + \frac \eta 2 h^2 dx
\\
&\quad\quad
+\sum_{\sBox \in \BBN}\log\Big(\chi_{\sigma(\sBox)}\big(\phi_K(\Box)\big)\Big) \Bigg).
\end{equs}

We translate the Gaussian fields so that their new mean is approximately $h$. The translation we use is related to the \textit{classical magnetism}, or response to the external field $\eta h$, used in the 2D setting \cite{GJS76-4} and given by $\eta(-\Delta+\eta)^{-1} h$. 

\begin{lem}\label{lem: bounds on translation}
For every $K>0$, let $h_K = \rho_K h$. Define $\tilde g_K = \eta (-\Delta + \eta)^{-1} h_K$ and $g_K = \rho_K \tilde g_K$. Then, there exists $C=C(\eta)$ such that
\begin{equs}\label{tool: g bounds}
| g_K |_{\infty}, | \nabla g_K |_{\infty}
\leq
C\hfbeta
\end{equs}
where $| \cdot |_{\infty}$ denotes the supremum norm. Moreover,
\begin{equs} \label{tool: grad g g tilde}
\intx |\nabla g_K|^2 dx
\leq
\intx |\nabla \tilde g_K|^2 dx.	
\end{equs}

Finally, let
\begin{equs}
g_k^\flat
=
\sum_{n \in (N^{-1}\ZZ)^3} \frac{1}{N^3} \tilde \rho_k \int_0^k \cJ_{k'}(n) \cF g (n) dk
\end{equs}
where $\tilde\rho_k$ is as in Section \ref{subsec: flat}. Then, for any $s,s' \in \RR$, $p \in (1,\infty)$ and $q \in [1,\infty]$, there exists $C_1=C_1(\eta,s,p,q)$ and $C_2=C_2(\eta,s,s',p,q)$ such that 
\begin{equs} \label{tool: g flat}
	\| g_k^\flat \|_{B^s_{p,q}}
	\leq
	C_1\| g_K \|_{B^s_{p,q}}
\end{equs}
and
\begin{equs} \label{tool: partial g flat}
\| \partial_k g_k^\flat \|_{B^{s'}_{p,q}}
\leq
C_2\frac{1}{\langle k \rangle^{1+s-s'}} \| g_K \|_{B^s_{p,q}}.
\end{equs}
\end{lem}

\begin{proof}
The estimate \eqref{tool: g bounds} follows from the fact that $\eta (-\Delta + \eta)^{-1}$ and $\nabla \eta (-\Delta + \eta)^{-1}$ are $L^\infty$ bounded operators. This is because the ($\eta$-dependent) Bessel potential and its first derivatives are absolutely integrable on $\RR^3$. Hence, by applying Young's inequality for convolutions one obtains the $L^\infty$ boundedness. The uniformity of the estimate over $\sigma$ follows from $\| \sigma \|_{L^\infty}=1$ for every $\sigma \in \BBN$. The other estimates follow from standard results about smooth multipliers, the observation that $g_k^\flat = \tilde\rho_k g_K$ for any $K \geq k$, and Lemma \ref{lem: flat drift bound}.
\end{proof}

\begin{rem}
Note that $g_K$ is given by the covariance operator of $\mu_N$ applied to $\eta h$. Moreover, note that $g_K \neq \tilde g_K$ since $\rho_K^2 \neq \rho_K$, i.e. the Fourier cutoff is not sharp. 
\end{rem}

By the Cameron-Martin theorem the density of $\mu_N$ under the translation $\phi = \psi + \tilde g_K$ transforms as
\begin{equs}
    d\mu_N(\psi+\tilde g_K) 
    = 
    \exp \Big( - \intx \frac 12 \tilde g_K (-\Delta + \eta) \tilde g_K + \psi (-\Delta + \eta)\tilde g_K dx \Big) d\mu_N(\psi).
\end{equs}
Hence,
\begin{equs}
    \sZ_{\beta,N,K}^\sigma 
    = 
    \EE_N e^{-\cH_{\beta,N,K}^\sigma(\psi_K) - F_{\beta,N,K}^\sigma}(\psi)
\end{equs}
where
\begin{equs}
    \cH_{\beta,N,K}^\sigma(\psi_K)
    &=
    \intx : \cV_\beta(\psi_K + g_K): - \frac{\gamma_K}{\beta^2}:(\psi_K+g_K)^2: - \delta_K
    \\
    &\quad\quad\quad
    - \frac \eta 2 :(\psi_K + g_K - h)^2: dx - \sum_{\sBox \in \BBN} \log \Big( \chi_{\sigma(\sBox)}\big((\psi_K + g_K)(\Box)\big) \Big)
\end{equs}
and
\begin{equs}
	F_{\beta,N,K}^\sigma(\psi)
	=
	\intx -\eta (\psi_K + g_K) h + \frac \eta 2 h^2 + \frac 12 \tilde g_K (-\Delta + \eta) \tilde g_K + \psi (-\Delta + \eta)\tilde g_K dx.
\end{equs}

By integration by parts, the self-adjointness of $\rho_K$, and the definition of $\tilde g_K$
\begin{equs} \label{def: f}
\begin{split}
	F_{\beta,N,K}^\sigma(\psi)
	&=
	\intx - \eta (\psi + \tilde g_K) h_K + \frac \eta 2 h^2 + \frac 12 |\nabla \tilde g_K|^2 + \frac \eta 2 (\tilde g_K)^2 + \eta \psi h_K dx
	\\
	&=
	\intx \frac \eta 2 (\tilde g_K - h_K)^2 + \frac \eta 2 (1-\rho_K^2) h^2 + \frac 12 |\nabla \tilde g_K|^2 dx.
\end{split}
\end{equs}
Thus, $F^\sigma_{\beta,N,K}(\psi)$ is independent of $\psi$ and non-negative.

\begin{rem}
Let $g = \eta(-\Delta + \eta)^{-1} h$. Then,
\begin{equs} \label{eq: F term}
\lim_{K \rightarrow \infty} F^\sigma_{\beta,N,K}
=
\intx \frac \eta 2 (g-h)^2 + \frac 12 |\nabla g|^2 dx.	
\end{equs}
The second integrand on the righthand side of \eqref{eq: F term} penalises the discontinuities of $\sigma$. Indeed, $e^{-\intx \frac 12 |\nabla g|^2 dx}$ is approximately equal to $e^{-C\hfbeta |\partial \sigma|}$, where $\partial \sigma $ denotes the surfaces of discontinuity of $\sigma$, $|\partial \sigma|$ denotes the area of these surfaces, and $C>0$ is an inessential constant. Thus, for $\beta$ sufficiently large, $\sZ_{\beta,N}$ is approximately equal to
\begin{equs}
e^{- C \sqrt \beta |\partial \sigma|} \times O(1)
=
\prod_{\Gamma_i \in \sigma} e^{-C \hfbeta |\Gamma_i|} \times O(1)	
\end{equs}
where $\Gamma_i$ are the connected components of $\partial \sigma$ (called contours). It would be interesting to further develop this contour representation for $\nubn$ (compare with the 2D expansions of \cite{GJS76-4, GJS76-3}).
\end{rem}

\subsection{Coarse-graining of the Hamiltonian} \label{subsec: coarsegrain}
We apply Proposition \ref{prop: bd} to $-\log \EE_N e^{-\cH_{\beta,N,K}^\sigma(\<1>_K)}$. For every $v \in \HH_{b,K}$, define
\begin{equs}\label{def: psi}
    \Psi_K^\sigma(v)
    =
    \cH^\sigma_{\beta,N,K}(\<1>_K+V_K) + \frac 12 \intt \intx v_k^2\ dx dk.
\end{equs}

Let $Z_K = \vec{\<1>}_K + V_K + g_K$, where $\vec{\<1>}_K = ( \<1>_K(\Box) )_{\sBox \in \BBN}$. We split the Hamiltonian as
\begin{equs}\label{hamiltonian_split}
   \cH_{\beta,N,K}^\sigma(\<1>_K+V_K) 
   = 
   \cH^\mathrm{eff}_K(Z_K) + \cR_K + \frac 12 \intx \cV_\beta(V_K+g_K) dx
\end{equs}
where 
\begin{equs}
\cH_K^\mathrm{eff}(Z_K)
&=
\intx \frac 12 \cV_{\beta,N,K}(Z_K) - \frac \eta 2 (Z_K -h)^2dx - \sum_{\sBox \in \BBN}\log \Big(\chi_{\sigma(\sBox)}\big(Z_K(\Box)\big) \Big)
\end{equs}
is an effective Hamiltonian introduced to capture macroscopic scales of the system. The quantity $\cR_K$ is then determined by \eqref{hamiltonian_split} and is explicitly given by
\begin{equs}
\cR_K
&=
\intx :\cV_\beta(\<1>_K+V_K+g_K): - \frac{\gamma_K}{\beta^2}:(\<1>_K+V_K+g_K)^2: - \delta_{K} 
\\
&\quad\quad\quad
- \frac 12 \cV_\beta(\vec{\<1>}_K + V_K + g_K) - \frac 12 \cV_\beta(V_K+g_K)
\\
&\quad\quad\quad
- \frac \eta 2 :(\<1>_K+V_K+g_K - h)^2: + \frac \eta 2 (\vec{\<1>}_K + V_K + g_K - h)^2 dx.
\end{equs}
All analysis/cancellation of ultraviolet divergences occurs within the sum of $\cR_K$ and the drift entropy, see \eqref{eq: all the UV is here}. Finally, the last term in \eqref{hamiltonian_split} is a stability term which is key for our non-perturbative analysis, namely it allows us to obtain estimates that are uniform in the drift. 

The key point is that we coarse-grain the field by block averaging $\<1>_K$, the most singular term. This allows us to preserve the structure of the low temperature potential $\cV_\beta$ on macroscopic scales (captured in $\cH_K^\mathrm{eff}(Z_K)$), which is crucial to obtaining estimates independent of $\beta$ on the free energy. 

\subsection{Killing divergences} \label{subsec: killing divergences}

\subsubsection{Changing drift variables} \label{subsec: change of variables}
For any $v \in \HH_{b,K}$, define $r=r(v) \in \HH_{K}$ by
\begin{equs} \label{eq: full paracontrolled ansatz}
r_k
=
v_k + \frac{4}{\beta} \cJ_k \<3>_k + \frac{12}{\beta} \cJ_k(\<2>_k \pg (V_k^\flat+g_k^\flat)).
\end{equs}

In our analysis it is convenient to use an intermediate change of variables for the drift. Define $u=u(v) \in \HH_K$ by
\begin{equs} \label{eq: intermediate ansatz}
u_k = v_k + \frac 4\beta \cJ_k \<3>_k.	
\end{equs}

Inserting \eqref{eq: full paracontrolled ansatz} and \eqref{eq: intermediate ansatz} into the definition of the integrated drift, $V_k = \int_0^k \cJ_{k'} v_{k'} dk'$, we obtain
\begin{equs} \label{eq: intermediate integrated ansatz}
\begin{split}
V_k
&=
- \frac{4}{\beta} \<30>_k  - \frac{12}{\beta}\int_0^k \cJ_{k'}^2 (\<2>_{k'} \pg (V_{k'}^\flat+g_{k'}^\flat)) dk'+ R_k
\\
&=
-\frac 4\beta \<30>_k + U_k	
\end{split}
\end{equs}
where $R_k = \int_0^k \cJ_{k'} r_{k'} dk'$ and $U_k = \int_0^k \cJ_{k'} u_{k'} dk'$. 

The following proposition contains useful estimates estimates on $U_K$ and $V_K$.

\begin{prop}\label{prop: u and v bound}
For any $\varepsilon > 0$ and $\kappa > 0$ sufficiently small, there exists $C=C(\varepsilon,\kappa,\eta) > 0$ such that, for all $\beta > 1$,
\begin{equs}
\begin{split}\label{eq: u estimate}
\sup_{0 \leq k \leq K} \|U_k\|_{H^{1-\kappa}}^2
&\leq
\frac{C N^\Xi_K}{N^3} + \frac{\varepsilon}{\beta^3} \intx \cV_\beta(V_K+g_K) \dbar x 
\\
&\quad\quad\quad \quad\quad\quad
+C\intt \intx r_k^2 \dbar x dk
\end{split}
\\
\begin{split}\label{eq: v estimate}
\sup_{0 \leq k \leq K} \| V_k \|_{H^{\frac 12 - \kappa}}^2 
&\leq
\frac{CN^\Xi_K}{N^3} + \frac{\varepsilon}{\beta^3} \intx \cV_\beta (V_K+g_K) \dbar x 
\\
&\quad\quad\quad\quad\quad\quad
+ C\intt \intx r_k^2 \dbar x dk	
\end{split}
\end{equs}
where we recall $\dbar x = \frac{dx}{N^3}$; and $N^\Xi_K$ is a positive random variable on $\Omega$ that is $\PP$-almost surely given by a finite linear combination of powers of (finite integrability) Besov and Lebesgue norms of the diagrams $\Xi = \{ \<1>, \<2>, \<30>, \<31>, \<32>, \<202> \}$ on the interval $[0,K]$.
\end{prop}

\begin{proof}
See Section \ref{subsec: estimates on drift}.
\end{proof}

\begin{rem} \label{rem: N convention}
As a consequence of Proposition \ref{prop: diagrams}, the random variable $N^\Xi_K$ satisfies the following estimate: there exists $C = C(\eta)>0$ such that
\begin{equs} \label{eq: extensive xi estimate}
\EE N^\Xi_K
\leq
C N^3.	
\end{equs}
In the following we denote by $N^\Xi_K$ \textit{any} positive random variable on $\Omega$ that satisfies \eqref{eq: extensive xi estimate}. In practice it is always $\PP$-almost surely given by a finite linear combination of powers of (finite integrability) Besov norms of the diagrams in $\Xi$ on $[0,K]$. Note that $N^\Xi_K$ includes constants of the form $C=C(\eta)>0$. 
\end{rem}

\subsubsection{The main small-scale estimates}

In the following we write $\approx$ to mean equal up to a term with expectation $0$ under $\PP$.

\begin{prop}\label{prop: killing}
Let $\beta > 0$. For every $K>0$, define
\begin{equs}\label{def: mass renorm}
\gamma_K 
&=
-4^2 \cdot 3 \cdot \<sunset>_K
\end{equs}
where $\<sunset>_K$ is defined in \eqref{def: sunset}, and
\begin{equs}
\begin{split} \label{def: energy renorm}
\delta_{K}
&=	
\EE \Biggr[ \intx \int_0^K - \frac{8}{\beta^2}(\cJ_k\<3>_k)^2 dk  - \frac{256}{\beta^4}\<1>_K \Big( \<30>_K \Big)^3 
\\
&\quad\quad\quad\quad\quad\quad 
+ \frac{96}{\beta^3}\Big(\<30>_K\Big)^2\<2>_K dx \Biggr].
\end{split}
\end{equs}

Then, for every $v \in \HH_{b,K}$,
\begin{equs}\label{eq: all the UV is here}
\cR_K + \frac 12 \intt \intx v_k^2 dk dx
\approx
\sum_{i=1}^4 \cR^i_K  + \frac 12 \int_0^K \int_{\TTN}  r_k^2 dx dk
\end{equs}
where
\begin{equs}
\cR^1_K
&=
\intx - \frac 1{2\beta}(\vec{\<1>}_K)^4 - \frac 2\beta (\vec{\<1>}_K)^3 (V_K + g_K)- \frac 3\beta (\vec{\<1>}_K)^2 (V_K+g_K)^2
\\
&\quad\quad\quad
- \frac 2\beta \vec{\<1>}_K V_K^3 - \frac 6\beta \vec{\<1>}_K V_K^2 g_K - \frac 6\beta \vec{\<1>}_K V_K g_K^2 
\\
&\quad\quad\quad
+ \frac{\eta + 2}2 (\vec{\<1>}_K)^2 + (\eta+2) \vec{\<1>}_K V_K dx
\\ 
\cR^2_K
&=
\intx \frac{192}{\beta^3}\<1>_K\<30>_K^2U_K  - \frac{48}{\beta^2}\<1>_K\<30>_KU_K^2 - \frac{96}{\beta^2}\<1>_K\<30>_Kg_KU_K
\\
&\quad\quad\quad
+ \frac{4}{\beta}\<1>_K U_K^3 + \frac{12}{\beta}\<1>_K g_KU_K^2 + \frac{12}{\beta}\<1>_K g_K^2 U_K - (4+\eta)\<1>_K U_K dx
\\ 
\cR^3_K
&=
\intx \frac{12}{\beta}(\<2>_K \pe g_K) U_K + \frac{6}{\beta}(\<2>_K \pe U_K - \<2>_K \pg U_K)U_K
\\
&\quad\quad\quad
- \frac{48}{\beta^2} (\<2>_K \pl \<30>_K) U_K + \frac{12}\beta (\<2>_K \pl g_K) U_K + \frac 6\beta (\<2>_K \pl U_K)U_K dx
\\
&\quad
+ \intx \intt \frac{12}\beta \Big(\<2>_k \pg (\partial_k V_k^\flat + \partial_k g_k^\flat)\Big)U_k 
\\
&\quad\quad\quad
+ \frac{12}{\beta} \Big(\<2>_K \pg (V_K+g_K-V_K^\flat-g_K^\flat)\Big) U_K 
\\
&\quad\quad\quad
- \frac{72}{\beta^2} \Bigg( \Big(\cJ_k (\<2>_k \pg (V_k^\flat+g_k^\flat)) \Big)^2 - \Big(\cJ_k \<2>_k \pe \cJ_k \<2>_k \Big)(V_k^\flat+g_k^\flat)^2 \Bigg) dk dx
\\ 
\cR^4_K
&=
-\intx \frac{48}{\beta^2} \<32>_KU_K + \frac{2\gamma_K}{\beta^2}(V_K^\flat+g_K)(V_K + g_K - V_K^\flat - g_K^\flat) 
\\
&\quad\quad\quad
+ \frac{\gamma_K}{\beta^2}(V_K + g_K-V_K^\flat-g_K^\flat)^2 dx
\\
&\quad
+ \intx \intt \frac{2\gamma_k}{\beta^2}( \partial_k V_k^\flat + \partial_k g_k^\flat)(V_k^\flat+g_k^\flat) + \frac{72}{\beta^2}\<202>_k (V_k^\flat+g_k^\flat)^2 dk dx.
\end{equs}

Moreover, the following estimate holds: for any $\varepsilon >0$, there exists $C=C(\varepsilon, \eta)>0$ such that, for all $\beta > 1$,
\begin{equs} \label{eq: remainder estimates}
\begin{split}
\max_{i=1,\dots,4}\Big| \cR^i_K \Big|
\leq 
CN^\Xi_K 
+ \varepsilon \Big( \intx \cV_\beta(V_K+g_K) dx + \frac 12 \int_0^K \int_{\TTN} r_k^2 dx dk \Big)
\end{split}
\end{equs}
where $N^\Xi_K$ is as in Remark \ref{rem: N convention}.
\end{prop}

\begin{proof}
We establish \eqref{eq: all the UV is here} in Section \ref{subsec: perturb} by arguing as in \cite[Lemma 5]{BG19}. The remainder estimates \eqref{eq: remainder estimates} are then established in Section \ref{subsec: estimates}.
\end{proof}

\begin{rem}
The products $\<1>_K \<30>_K$ and $\<1>_K \<30>_K^2$ appearing above are classically ill-defined in the limit $K \rightarrow \infty$. However, (probabilistic) estimates on the resonant product $\<31>_K$ uniform in $K$ are obtained in Proposition \ref{prop: diagrams}. Hence, the first product can be analysed using a paraproduct decompositions \eqref{tool: bony decomp}. The second product is less straightforward and requires a double paraproduct decomposition (see \cite[Lemma 21 and Proposition 6]{BG19} and \cite[Proposition 2.22]{CC18}). 
\end{rem}

\subsection{Proof of \eqref{eq: all the UV is here}: Isolating and cancelling divergences}\label{subsec: perturb}

Using that $\<1>_K, \vec{\<1>}_K, \<2>_K, \<3>_K$ and $\<4>_K$ all have expectation zero,

\begin{equs}
	\cR_K
	&=
	\intx \frac 1\beta \<4>_K + \frac 4\beta \<3>_K (V_K+g_K) + \frac 6\beta \<2>_K (V_K+g_K)^2 + \frac 4\beta \<1>_K (V_K+g_K)^3
	\\
	&\quad\quad\quad
	- 2\<2>_K - 4\<1>_K (V_K+g_K) + \cV_\beta(V_K+g_K) 
	\\
	&\quad\quad\quad
	-\frac{\gamma_K}{\beta^2} \Big( \<2>_K + 2\<1>_K(V_K+g_K) + (V_K+g_K)^2 \Big) - \delta_K
	\\
	&\quad\quad\quad
	- \frac1{2\beta} (\vec{\<1>}_K)^4 - \frac{2}{\beta}(\vec{\<1>}_K)^3 (V_K+g_K) - \frac3{\beta} (\vec{\<1>}_K)^2(V_K+g_K)^2 
	\\
	&\quad\quad\quad
	- \frac2\beta \vec{\<1>}_K (V_K+g_K)^3 +(\vec{\<1>}_K)^2 + 2\vec{\<1>}_K(V_K+g_K) - \frac 12 \cV_\beta(V_K+g_K) 
	\\
	&\quad\quad\quad
	- \frac 12 \cV_\beta(V_K+g_K)
	\\
	&\quad\quad\quad
	-\frac\eta2 \<2>_K - \eta \<1>_K (V_K+g_K-h) - \frac\eta2(V_K+g_K-h)^2
	\\
	&\quad\quad\quad
	+\frac\eta2 (\vec{\<1>}_K)^2 + \eta \vec{\<1>}_K(V_K+g_K-h) + \frac \eta 2 (V_K+g_K-h)^2 dx
	\\
	&\approx
	\intx \frac 4\beta \<3>_K V_K + \frac 6\beta \<2>_K (V_K+g_K)^2 + \frac 4\beta \<1>_K(V_K+g_K)^3 -4\<1>_KV_K
	\\
	&\quad\quad\quad
	- \frac{2\gamma_K}{\beta^2}\<1>_KV_K -\frac{\gamma_K}{\beta^2} (V_K+g_K)^2 - \delta_K
	\\
	&\quad\quad\quad
	-\frac 1{2\beta} (\vec{\<1>}_K)^4 - \frac 2\beta (\vec{\<1>}_K)^3(V_K+g_K) - \frac 3\beta (\vec{\<1>}_K)^2 (V_K+g_K)^2 
	\\
	&\quad\quad\quad
	- \frac 2\beta \vec{\<1>}_K V_K^3 - \frac 6\beta \vec{\<1>}_K V_K^2 g_K - \frac 6\beta \vec{\<1>}_K V_K g_K^2 + (\vec{\<1>}_K)^2 + 2 \vec{\<1>}_K V_K
	\\
	&\quad\quad\quad
	-\eta \<1>_K V_K + \frac \eta 2 (\vec{\<1>}_K)^2 + \eta \vec{\<1>}_K V_K dx
\end{equs}

Hence, by reordering terms,
\begin{equs}
	\begin{split} \label{perturb0}
	\cR_K
	&\approx
	\cR_K^1 + \intx \frac 4\beta \<3>_K V_K + \frac 6\beta \<2>_K (V_K+g_K)^2 + \frac 4\beta \<1>_K(V_K+g_K)^3 
	\\
	&\quad\quad\quad
	- (4+\eta)\<1>_K V_K - \frac{2\gamma_K}{\beta^2}\<1>_K V_K - \frac{\gamma_K}{\beta^2}(V_K+g_K)^2 -\delta_K dx.
\end{split}
\end{equs}

Ignoring the renormalisation counterterms (i.e. those involving $\gamma_K$ and $\delta_K$), the divergences in \eqref{perturb0} are contained in the integrals $\intx \frac 4\beta \<3>_K V_K dx$ and $\intx \frac 6\beta (V_K+g_K)^2$. In order to kill these divergences, we use changes of variables in the drift entropy to mostly cancel them; the remaining divergences are killed by the renormalisation counterterms. We renormalise the leading order divergences, ie. those polynomial in $K$, in Section \ref{subsec: polynomial divergences}. The divergences that are logarithmic in $K$ are renormalised in Section \ref{subsec: log divergences}.

In order to use the drift entropy to cancel divergences, we decompose certain (spatial) integrals across ultraviolet scales $k \in [0,K]$ using It\^o's formula. Error terms are produced that are stochastic integrals with respect to martingales (specifically, with respect to $d\<3>_k$ and $d\<2>_k$). The following lemma allows us to argue that these stochastic integrals are $\approx 0$.
\begin{lem} \label{lem: martingales}
For any $v \in \HH_{b,K}$, the stochastic integrals
\begin{equs}\label{eq: stoch integral 1}
\intx \intt V_k d\<3>_k dx
\end{equs}
and
\begin{equs} \label{eq: stoch integral 2}
\sum_{i < j-1}\intx \intt U_k \Delta_i (V_k^\flat + g_k^\flat)d(\Delta_j \<2>_k) dx
\end{equs}
are martingales. We recall that, above, $\Delta_i$ denotes the $i$-th Littlewood-Paley block.
\end{lem}

\begin{proof}
In this proof, for any continuous local martingale $Z_\bullet$, we write $\llangle Z,Z\rrangle_\bullet$ for the corresponding quadratic variation process. Moreover, for any $Z$-adapted process $Y_\bullet$, we write $ \intt Y_k \cdot dZ_k$ to denote the stochastic integral $\intx \intt Y_k dZ_k dx$.

We begin with two observations: first, let $v \in \HH_{b,M,K}$ for some $M > 0$, i.e. those $v \in \HH_{b,K}$ satisfying \eqref{eq: l2 drift bound}. Then, by Sobolev embedding, there exists $C=C(M,N,K,\eta)>0$ such that
\begin{equs}
\sup_{0 \leq k \leq K} \| V_k \|_{L^6}^6	
\leq
C
\end{equs}
$\PP$-almost surely.

Second, recalling the iterated integral representation of the Wick powers $\<3>_k$ and $\<2>_k$ (see e.g. \eqref{eq: iterated integral third wick}), one can show $d\<3>_k = 3 \<2>_k d\<1>_k$ and $d(\Delta_j \<2>_k) = \Delta_j d\<2>_k = 2\Delta_j \<1>_k d\<1>_k$. Thus, we can write the stochastic integrals \eqref{eq: stoch integral 1} and \eqref{eq: stoch integral 2} in terms of stochastic integrals with respect to $d\<1>_k$. It suffices to show that their quadratic variations are finite in expectation.
 
Using that $d\llangle \<1>,\<1>\rrangle_k = \cJ^2_k(1) dk = \cJ^2_k dk$ and by Young's inequality,
\begin{equs}
\EE \Bigg[ \llangle \int_0^\bullet  V_{k} \cdot d\<3>_{k} \rrangle_K \Bigg]
&=
3^2 \EE \Bigg[ \intt \intx V_k^2 \<2>_k^2 \cJ_k^2 dx dk \Bigg]
\\
&\leq
3^2\EE \Bigg[ \frac 12 \intt \intx V_k^4 + \<2>_k^4 \cJ_k^4 dx dk \Bigg] < \infty.
\end{equs}
Hence, \eqref{eq: stoch integral 1} is a martingale.

Now consider \eqref{eq: stoch integral 2}. By  \eqref{eq: intermediate integrated ansatz},
\begin{equs}
\sum_{i < j-1} \intt \intx U_k \Delta_i(V_k^\flat + g_k^\flat) \cdot d(\Delta_j \<2>_k)
=
Z^a_K + Z^b_K
\end{equs}
where
\begin{equs}
Z^a_K
&=
2 \sum_{i < j-1} \intt \intx V_k \Delta_i V_k^\flat \Delta_j \<1>_k d\<1>_k
\\
Z^b_K
&= 
2\sum_{i < j-1} \intt \intx \Big(\frac 4\beta \<30>_k \Delta_i V_k^\flat  + U_k \Delta_i g_k^\flat \Big) \Delta_j \<1>_k d\<1>_k.
\end{equs}
Arguing as for \eqref{eq: stoch integral 1}, one can show $\EE \llangle Z^b_\bullet \rrangle_K < \infty$. 

By Young's inequality and using that Littlewood-Paley blocks and the $\flat$ operator are $L^p$ multipliers, we have
\begin{equs}
\EE \llangle Z^a_\bullet \rrangle_K
&=
2^2 \EE \Bigg[ \sum_{i < j-1} \intt \intx V_k^2 (\Delta_i V_k^\flat)^2 (\Delta_j \<1>_k)^2 \cJ_k^2 dk \Bigg]
\\
&\leq
2^2\EE \Bigg[ \frac 23 \intt \intx V_k^6 + \frac 13 (\Delta_k \<1>_k)^6 \cJ_k^6 dx dk \Bigg] 
< 
\infty	
\end{equs}
thus establishing that \eqref{eq: stoch integral 2} is a martingale.
\end{proof}

\subsubsection{Energy renormalisation} \label{subsec: polynomial divergences}

In the next lemma, we cancel the leading order divergence using the change of variables \eqref{eq: intermediate ansatz} in the drift entropy. The error term does not depend on the drift and is divergent in expectation (as $K \rightarrow \infty$); it is cancelled by one part of the energy renormalisation $\delta_K$ (see \eqref{def: energy renorm}).

\begin{lem}\label{lem: perturb 1}
\begin{equs}
\intx \frac 4\beta \<3>_K V_K dx + \frac 12 \intt \intx v_k^2 dx dk
\approx
\intt \intx - \frac{8}{\beta^2} (\cJ_k \<3>_k)^2 + \frac 12  u_k^2 dx dk.
\end{equs}
\end{lem}

\begin{proof}
By It\^o's formula, Lemma \ref{lem: martingales}, and the self-adjointness of $\cJ_k$,
\begin{equs}
\intx \frac 4\beta \<3>_K V_K dx
=
\intx \Big(\intt \frac 4\beta \<3>_k \partial_k V_k dk + \frac 4\beta V_k d\<3>_k \Big) dx	
\approx
\intx \intt \frac 4\beta \cJ_k \<3>_k v_k dk dx.
\end{equs}

Hence, by \eqref{eq: intermediate ansatz},
\begin{equs}
\intx \frac 4\beta \<3>_K V_K dx &+ \frac 12 \intt \intx v_k^2 dx dk
\\
&\approx
\intx \intt \frac 4\beta \cJ_k \<3>_k \big( -\frac 4\beta \<3>_k + u_k) + \frac 12 \big( -\frac 4\beta \<3>_k + u_k)^2 dk dx
\\
&=
\intx \intt -\frac{8}{\beta^2} (\cJ_k \<3>_k)^2 + \frac 12 u_k^2 dk dx. 	
\end{equs}
\end{proof}

As a consequence of \eqref{eq: intermediate ansatz}, the remaining (non-counterterm) integrals in \eqref{perturb0} acquire additional divergences that are independent of the drift. We isolate them in the next lemma; they are also renormalised by parts of the energy renormalisation (see \eqref{def: energy renorm}).
\begin{lem}\label{lem: perturb 1.5}
\begin{equs}\label{perturb1}
\intx \frac 4\beta \<1>_K (V_K+g_K)^3 - (4+\eta)\<1>_K V_K dx
\approx
\cR_K^2 - \intx \frac{256}{\beta^4}\<1>_K\<30>_K^3 dx
\end{equs}
and
\begin{equs}
\begin{split} \label{perturb1.5}
\intx \frac{6}{\beta} \<2>_K (V_K+g_K)^2 dx
&\approx	
\intx \frac{96}{\beta^3} \<2>_K \<30>_K^2 - \frac{48}{\beta^2}\<2>_K\<30>_K U_K
\\
&\quad\quad\quad
+ \frac 6\beta \<2>_K U_K^2 + \frac {12}\beta \<2>_K g_K U_K dx.
\end{split}
\end{equs}
\end{lem}

\begin{proof}
By \eqref{eq: intermediate integrated ansatz},
\begin{equs} \label{perturb1 eq2}
\begin{split}
\intx \frac{4}{\beta} &\<1>_K (V_K+g_K)^3 dx
\\
&=
\intx \frac{4}{\beta} \<1>_K (- \frac{4}{\beta} \<30>_K + U_K + g_K)^3 dx
\\
&=
\intx \frac 4\beta \<1>_K \Bigg( - \frac{64}{\beta^3} \<30>_K^3 + \frac{48}{\beta^2} \<30>_K^2(U_K + g_K) 
\\
&\quad\quad\quad
- \frac{12}{\beta} \<30>_K U_K^2 - \frac{24}\beta \<30>_K U_K g_K - \frac{12}\beta \<30>_K g_K^2
\\
&\quad\quad\quad
+ U_K^3 + 3 U_K^2 g_K + 3 U_K g_K^2 + g_K^3 \Bigg) dx
\\
&\approx
\intx - \frac{256}{\beta^4}\<1>_K\<30>_K^3 + \frac{192}{\beta^3}\<1>_K\<30>_K^2U_K 
\\
&\quad - \frac{48}{\beta^2}\<1>_K\<30>_KU_K^2 - \frac{96}{\beta^2}\<1>_K\<30>_Kg_KU_K
\\
&\quad
+ \frac{4}{\beta}\<1>_K U_K^3 + \frac{12}{\beta}\<1>_K g_KU_K^2 + \frac{12}{\beta}\<1>_K g_K^2 U_K dx.
\end{split}
\end{equs}
Above we have used Wick's theorem and the fact that $\<30>_K$ is Wick ordered to conclude $\EE \big[\<1>_K \<30>_K^2 g_K\big] = \EE \big[\<1>_K \<30>_K g_K^2\big] = 0$.

Similarly, $\EE \<1>_K \<30>_K = 0$ $\<30>_K$. Hence, by \eqref{eq: intermediate integrated ansatz}
\begin{equs} \label{perturb1 eq3}
\intx (4+\eta)\<1>_K V_K 
\approx
\intx (4+\eta)\<1>_K U_K dx. 
\end{equs}
Combining \eqref{perturb1 eq2} and \eqref{perturb1 eq3} establishes \eqref{perturb1}.

By \eqref{eq: intermediate integrated ansatz},
\begin{equs}
\begin{split}
\intx \frac{6}{\beta} \<2>_K (V_K+g_K)^2 dx
&=
\intx \frac{6}{\beta} \<2>_K \Big(- \frac{4}{\beta}\<30>_K + U_K + g_K\Big)^2 dx
\\
&=
\intx \frac{6}{\beta} \<2>_K \Big( \frac{16}{\beta^2}\<30>_K^2 - \frac{8}{\beta} (U_K+g_K)\<30>_K 
\\
&\quad\quad\quad
U_K^2 + 2U_Kg_K + g_K^2 \Big) dx
\\
&\approx
\intx \frac{96}{\beta^3} \<2>_K \<30>_K^2 + \frac{12}{\beta}\<2>_K\Big(-\frac{4}{\beta}\<30>_K \Big)U_K
\\
&\quad\quad\quad
+ \frac 6\beta \<2>_K U_K^2 + \frac {12}\beta \<2>_K g_K U_K dx
\end{split}
\end{equs}
where we have used that $\EE [\<2>_K g_K] = 0$ and, by Wick's theorem, $\EE \big[ \<2>_K \<30>_K \big] = 0$. This establishes \eqref{perturb1.5}.
\end{proof}

The divergences encountered in Lemmas \ref{lem: perturb 1} and \ref{lem: perturb 1.5} that are independent of the drift are killed by the energy renormalisation $\delta_K$ since, by definition,
\begin{equs}\label{eq: energy renorm cancellation}
	\delta_K 
	\approx 
	\intx  - \int_0^K \frac{8}{\beta^2}(\cJ_k\<3>_k)^2 dk -\frac{256}{\beta^4}\<1>_K \Big( \<30>_K \Big)^3 + \frac{96}{\beta^3}\Big(\<30>_K\Big)^2\<2>_K dx.
\end{equs}

\subsubsection{Mass renormalisation} \label{subsec: log divergences}
The integrals on the righthand side of \eqref{perturb1.5} that involve the drift cannot be bounded uniformly as $K \rightarrow \infty$. We isolate divergences using a paraproduct decomposition and expand the drift entropy using \eqref{eq: full paracontrolled ansatz} to mostly cancel them. This is done in Lemma \ref{lem: perturb 2}. The remaining divergences are then killed in Lemma \ref{lem: mass renorm} using the mass renormalisation.

\begin{lem}\label{lem: perturb 2}
	\begin{equs}\label{perturb2}
	\begin{split}
	\intx - \frac{48}{\beta^2}\<2>_K\<30>_K U_K &+ \frac 6\beta \<2>_K U_K^2 + \frac {12}\beta \<2>_K g_K U_K dx + \frac 12\intt u_k^2 dk dx
	\\
	&\approx
	\cR^3_K + \frac 12 \intx \intt r_k^2 dk dx 
	\\
	&\quad +\intx \frac{96}{\beta^3} \<2>_K \<30>_K^2 - \frac{48}{\beta^2} \<2>_K \pe \<30>_K U_K dx 
	\\ 
	&\quad
	- \intx \intt \frac{72}{\beta^2} (\cJ_k \<2>_k \pe \cJ_k \<2>_k)(V_k^\flat + g_k)^2 dkdx.
	\end{split}
	\end{equs}
\end{lem}

\begin{proof}

We write
\begin{equs}
- \frac{48}{\beta^2} \<2>_K \<30>_K U_K + \frac {12}{\beta}\<2>_K U_K g_K
=
\frac{12}{\beta} \<2>_K \Big( - \frac 4\beta \<30>_K + g_K \Big) U_K.	
\end{equs}
Thus, using \eqref{eq: intermediate integrated ansatz} and a paraproduct decomposition on the most singular products,
\begin{equs} \label{perturb2 eq2}
\begin{split}
\intx &\frac{12}{\beta} \<2>_K \Big(-\frac{4}{\beta}\<30>_K + g_K\Big)U_K + \frac{6}{\beta}\<2>_KU_K^2 dx
\\
&=
\intx \frac{12}{\beta} \Bigg( \<2>_K \pg \Big( - \frac 4\beta \<30>_K + g_K \Big) \Bigg) U_K + \frac 6\beta (\<2>_K \pg U_K) U_K
\\
&\quad\quad\quad
+ \frac {12}\beta \Bigg( \<2>_K \pe \Big( - \frac 4\beta \<30>_K + g_K \Big) \Bigg) U_K
+ \frac 6\beta (\<2>_K \pe U_K) U_K
\\
&\quad\quad\quad
+ \frac{12}\beta \Bigg( \<2>_K \pl \Big( - \frac 4\beta \<30>_K + g_K \Big) \Bigg) U_K
+ \frac 6\beta (\<2>_K \pl U_K) U_K
\\
&=
\intx \frac{12}{\beta} \Big(\<2>_K \pg (V_K+g_K)\Big) U_K - \frac{48}{\beta^2}\Big( \<2>_K \pe \<30>_K \Big) U_K 
\\
&\quad\quad\quad
+\frac{12}{\beta} (\<2>_K \pe g_K) U_K +\frac 6\beta (\<2>_K \pe U_K - \<2>_K \pg U_K)U_K
\\
&\quad\quad\quad
- \frac{48}{\beta^2} \Big(\<2>_K \pl \<30>_K\Big) U_K + \frac{12}\beta (\<2>_K \pl g_K) U_K + \frac 6\beta (\<2>_K \pl U_K)U_K dx.
\end{split}
\end{equs}
All except the first two integrals are absorbed into $\cR_K^3$.

For the first integral, we use the (drift-dependent) change of variables \eqref{eq: full paracontrolled ansatz} in the drift entropy of $u$ to mostly cancel the divergence. Due to the paraproduct term, using It\^o's formula to decompose into scales requires us to control $V_k + g_k$ for $k < K$. In order to be able to do this, we replace $V_K + g_K$ by $V_K^\flat + g_K^\flat$ first. Then, applying It\^o's formula, Lemma \ref{lem: martingales}, and using the self-adjointness of $\cJ_k$,
\begin{equs} \label{perturb2 eq3}
\begin{split}
&\intx \frac{12}{\beta} \Big(\<2>_K \pg (V_K+g_K)\Big) U_K dx
\\
&=
\intx \frac{12}{\beta} \Big(\<2>_K \pg (V_K^\flat+g_K^\flat)\Big)U_K +  \frac{12}{\beta}\Big(\<2>_K \pg (V_K +g_K - V_K^\flat - g_K^\flat)\Big) U_K  dx
\\
&\approx
\intx \intt \frac{12}{\beta} \cJ_k \Big(\<2>_k \pg (V_k^\flat+g_k^\flat)\Big)u_k + \frac{12}{\beta} \Big(\<2>_k \pg (\partial_k V_k^\flat + \partial_k g_k^\flat)\Big)U_k dk dx 
\\
&
\quad\quad\quad+ 
\intx \frac{12}{\beta}\Big(\<2>_K \pg (V_K+g_K - V_K^\flat-g_K^\flat)\Big) U_K \Bigg] dx.
\end{split}
\end{equs}

From \eqref{eq: full paracontrolled ansatz} and \eqref{eq: intermediate ansatz}
\begin{equs}\label{perturb2 eq4}
\begin{split}
\intx \intt &\frac{12}{\beta} \cJ_k(\<2>_k \pg (V_k^\flat+g_k^\flat)) u_k + \frac 12 u_k^2 dk dx
\\
&= 
\intx \intt - \frac{72}{\beta^2}\Big(\cJ_k (\<2>_k \pg (V_k^\flat+g_k^\flat)) \Big)^2 + \frac 12 r_k^2 dk dx
\\
&=
 \intx \intt  - \frac{72}{\beta^2}\Big( \cJ_k \<2>_k \pe \cJ_k \<2>_k \Big) (V_k^\flat+g_k^\flat)^2 + \frac 12 r_k^2 
\\
&\quad\quad\quad
-
\frac{72}{\beta^2}\Bigg( \Big(\cJ_k (\<2>_k \pg (V_k^\flat+g_k^\flat)) \Big)^2 - \Big(\cJ_k \<2>_k \pe \cJ_k \<2>_k \Big)(V_k^\flat+g_k^\flat)^2 \Bigg) dk dx.
\end{split}
\end{equs}

Combining \eqref{perturb2 eq2}, \eqref{perturb2 eq3}, and \eqref{perturb2 eq4} yields \eqref{perturb2}.
\end{proof}

We now cancel the divergences in the last two terms of \eqref{perturb2} using the mass renormalisation.

\begin{lem}\label{lem: mass renorm}
	\begin{equs}\label{renorm}
	\begin{split}
	 \cR_K^4
	 &\approx
	 \intx - \frac{48}{\beta^2} \<2>_K \pe \<30>_K U_K dx
	 \\
	 & \quad
	 - \intx \intt \frac{72}{\beta^2}\Big( \cJ_k \<2>_k \pe \cJ_k \<2>_k \Big) (V_k^\flat+g_k^\flat)^2 dk dx
	 \\
	 &\quad
	 - \intx \frac{2\gamma_K}{\beta^2} \<1>_K V_K - \frac{\gamma_K}{\beta^2}(V_K+g_K)^2 dx
	 \end{split}
	\end{equs}
\end{lem}

\begin{proof}
By the definition of $\<32>_K$ (see Section \ref{subsec: diagrams}),
\begin{equs}\label{renorm eq1}
\begin{split}
- \intx \frac{48}{\beta^2}&\<2>_K\pe\<30>_K U_K  - \frac{2\gamma_K}{\beta^2} \<1>_KV_K dx
\\
&=
-\intx \frac{48}{\beta^2} \<32>_KU_K + \frac{8\gamma_K}{\beta^3} \<1>_K\<30>_K dx
\\
&\approx
-\intx \frac{48}{\beta^2} \<32>_K U_K dx
\end{split}
\end{equs}
where we have used that, by Wick's theorem, $\EE \big[ \<1>_K \<30>_K \big] = 0$.

To renormalise the second integral in \eqref{renorm}, we need to rewrite the remaining counterterm in terms of $V_K^\flat$:
\begin{equs} \label{renorm eq2}
\begin{split}
-&\intx \frac{\gamma_K}{\beta^2} (V_K+g_K)^2 dx 
\\
&=
-\intx \frac{\gamma_K}{\beta^2} (V_K^\flat+g_K^\flat)^2  + 2(V_K^\flat+g_K^\flat)(V_K + g_K - V_K^\flat- g_K^\flat) 
\\
&\quad\quad\quad
+ (V_K + g_K -V_K^\flat - g_K^\flat)^2  dx.
\end{split}
\end{equs}

Using It\^o's formula on the first integral of the right hand side of \eqref{renorm eq2},
\begin{equs}
-&\intx \frac{\gamma_K}{\beta^2} (V_K^\flat+g_K^\flat)^2 dx
\\
&=
- \intx \intt \frac{\partial_k{\gamma}_k}{\beta^2} (V_k^\flat+g_k^\flat)^2 + \frac{2\gamma_k}{\beta^2}( \partial_k V_k^\flat + \partial_k g_k^\flat)(V_k^\flat+g_k^\flat) dk dx.
\end{equs}

By the definition of $\<202>_k$ (see Section \ref{subsec: diagrams}),
\begin{equs} \label{renorm eq3}
\begin{split}
\intx \intt \Big( - \frac{72}{\beta^2} \cJ_k \<2>_k \pe \cJ_k \<2>_k &- \frac{\partial_k \gamma_k}{\beta^2} \Big) (V_k^\flat+g_k^\flat)^2 dk dx 
\\
&= 
\intx \intt \frac{72}{\beta^2}\<202>_k (V_k^\flat+g_k^\flat)^2 dk dx.
\end{split}
\end{equs}

Hence, combining \eqref{renorm eq1}, \eqref{renorm eq2}, and \eqref{renorm eq3} establishes \eqref{renorm}.
\end{proof}

\begin{proof}[Proof of \eqref{eq: all the UV is here}]
Lemmas \ref{lem: perturb 1}, \ref{lem: perturb 1.5}, \ref{lem: perturb 2}, and \ref{lem: mass renorm}, together with \eqref{eq: energy renorm cancellation}, establish \eqref{eq: all the UV is here}.
\end{proof}

\subsection{Proof of \eqref{eq: remainder estimates}: Estimates on remainder terms}\label{subsec: estimates}

Define
\begin{equs}
\cR^a_K
&=
\cR^{a,1}_K + \cR^{a,2}_K + \cR^{a,3}_K
\end{equs}
where
\begin{equs}
\cR^{a,1}_K
&=
\intx - \frac 4\beta \vec{\<1>}_K V_K^3 + \frac 4\beta \<1>_K U_K^3 dx
\\
\cR^{a,2}_K
&=
\intx \intt \frac{12}\beta \Big( \<2>_k \pg (\partial_k V_k^\flat + \partial_k g_k^\flat) \Big) U_k + \frac{2\gamma_k}{\beta^2}( \partial_k V_k^\flat + \partial_k g_k^\flat)(V_k^\flat+g_k^\flat) dkdx
\\
\cR^{a,3}_K
&=
\intx \intt \frac{72}{\beta^2} \Bigg( \Big(\cJ_k (\<2>_k \pg (V_k^\flat+g_k^\flat)) \Big)^2 - \Big(\cJ_k \<2>_k \pe \cJ_k \<2>_k \Big)(V_k^\flat+g_k^\flat)^2 \Bigg) dk dx
\end{equs}
and let $\cR^b_K = \sum_{i=1}^4 \cR^i_K - \cR^a_K$.

$\cR^a_K$ contains the most difficult terms to bound, either due to analytic considerations or $\beta$-dependence; $\cR^b_K$ contains the terms that follow almost immediately from \cite[Lemmas 18-23]{BG19}.

\begin{prop}\label{prop: r cgw}
For any $\varepsilon > 0$, there exists $C=C(\varepsilon,\eta)>0$ such that, for all $\beta > 1$,
\begin{equs}
\begin{split}\label{eq: cgw cubic}
|\cR^{a,1}_K|	
&\leq
C N^\Xi_K + \varepsilon \Big( \intx \cV_\beta(V_K+g_K) dx + \frac 12 \intt \intx r_k^2 dx dk \Big)
\end{split}
\\
\begin{split}\label{eq: cgw scalespace}
|\cR^{a,2}_K|
&\leq
C N^\Xi_K + \varepsilon \Big( \intx \cV_\beta(V_K+g_K) dx + \frac 12 \intt \intx r_k^2 dx dk \Big)
	\end{split}
\\
\begin{split}\label{eq: cgw commutator}
|\cR^{a,3}_K|
&\leq
C N^\Xi_K + \varepsilon \Big( \intx \cV_\beta(V_K+g_K) dx + \frac 12 \intt \intx r_k^2 dx dk \Big).
\end{split}
\end{equs}
\end{prop}

\begin{proof}
The estimates \eqref{eq: cgw cubic}, \eqref{eq: cgw scalespace}, and \eqref{eq: cgw commutator}	are established in Sections \ref{subsec: cubic bounds}, \ref{subsec: scalespace bounds}, and \ref{subsec: commutator bound} respectively. \eqref{eq: cgw scalespace} and \eqref{eq: cgw commutator} are established by a relatively straightforward combination of techniques in \cite[Lemmas 18-23]{BG19} together with Lemmas \ref{lem: potential} and \ref{lem: bounds on translation}. On the other hand, the terms with cubic dependence in the drift \eqref{eq: cgw cubic} require a slightly more involved analysis. 

Note that, since our norms on functions/distributions were defined using $\dbar x = \frac{dx}{N^3}$ instead of $dx$ to track $N$ dependence, in the proof we rewrite the integrals above in terms of $\dbar x$ by dividing both sides by $N^3$. 
\end{proof}

\begin{prop}\label{prop: bg bound}
For any $\varepsilon > 0$, there exists $C=C(\varepsilon,\eta)>0$ such that, for all $\beta > 1$,
\begin{equs}
|\cR^b_K|
\leq
C N^\Xi_K + \varepsilon \Big( \intx \cV_\beta(V_K+g_K) dx + \frac 12 \intt \intx r_k^2 dx dk \Big).
\end{equs}
\end{prop}

\begin{proof}
Follows from a direct combination of arguments in \cite[Lemmas 18-23]{BG19} with Lemmas \ref{lem: potential} and \ref{lem: bounds on translation}. We omit it.	
\end{proof}

\begin{proof}[Proof of \eqref{eq: remainder estimates}]
Since $\sum_{i=1}^4 \cR_K^i = \cR^a_K + \cR^b_K$, Propositions \ref{prop: r cgw} and \ref{prop: bg bound} establish \eqref{eq: remainder estimates}.
\end{proof}

The proofs of Propositions \ref{prop: r cgw} and \ref{prop: bg bound} rely heavily on bounds on the drift established in Proposition \ref{prop: u and v bound}, so we prove this first in the next subsection. Throughout the remainder of this section, we use the notation $a \lesssim b$ to mean $a \leq Cb$ for some $C=C(\varepsilon,\eta)$, and we also allow for this constant to depend on other inessential parameters (i.e. not $\beta$, $N$, or $K$).

\subsubsection{Proof of Proposition \ref{prop: u and v bound}} \label{subsec: estimates on drift}

First, note that \eqref{eq: v estimate} is a direct consequence of \eqref{eq: u estimate} along with \eqref{eq: intermediate integrated ansatz} and bounds contained in Proposition \ref{prop: diagrams}.

We now prove \eqref{eq: u estimate}. Fix any $k' \in [0,K]$. As a consequence of \eqref{eq: intermediate integrated ansatz},
\begin{equs} \label{eq: u estimate 0}
\|U_{k'}\|_{H^{1-\kappa}}^2 
&\leq 
\frac{288}{\beta^2} \Big\| \int_0^{k'} \cJ_k^2 (\<2>_k \pg (V_k^\flat + g_k^\flat)) dk \Big\|_{H^{1-\kappa}}^2 + 2\|R_{k'} \|_{H^{1-\kappa}}^2.
\end{equs}

By Minkowski's integral inequality, Bernstein's inequality \eqref{tool: bernstein annulus}, the multiplier estimate on $\cJ_k$ \eqref{tool: multiplier estimate}, the paraproduct estimate \eqref{tool: paraproduct}, and the $\flat$-estimates \eqref{eq: a priori flat drift}, 
\begin{equs}
\begin{split}
\Big\| \int_0^{k'} \cJ_k^2 (\<2>_k \pg (V_k^\flat + g_k^\flat)) dk \Big\|_{H^{1-\kappa}}
&\lesssim
\int_0^{k'} \frac{\| \cJ_k^2 (\<2>_k \pg (V_k^\flat + g_k^\flat)) \|_{H^{-1-2\kappa}}}{\langle k \rangle^\kappa} dk
\\
&\lesssim
\int_0^{k'} \frac{\|\<2>_k \pg (V_k^\flat + g_k^\flat)\|_{H^{-1-2\kappa}}}{ \langle k \rangle^{1+\kappa}} dk
\\
&\lesssim
\Bigg( \int_0^{k'} \frac{\|\<2>_k\|_{B^{-1-2\kappa}_{4,\infty}}}{\langle k \rangle^{1+\kappa}} dk \Bigg) \|V_K+g_K\|_{L^4}.
\end{split}
\end{equs}

Hence, by Cauchy-Schwarz with respect to the finite measure $\frac{dk}{\langle k \rangle^{1+\kappa}}$, the potential bound \eqref{tool: potential bound norm}, and Young's inequality,
\begin{equs} \label{eq: u estimate eq2}
\begin{split}
\frac{1}{\beta^2} \Big\| &\int_0^{k'} \cJ_k^2 (\<2>_k \pg (V_k^\flat + g_k^\flat)) dk \Big\|_{H^{1-\kappa}}^2
\\
&\lesssim
\frac{1}{\beta^2} \Bigg(\int_0^{k'} \frac{ \|\<2>_k\|_{B^{-1-2\kappa}_{4,\infty}}}{\langle k \rangle^{1+\kappa}}dk \Bigg)^2\|V_K+g_K\|_{L^4}^2
\\
&\lesssim
\frac{1}{\beta^2} \Bigg( \int_0^{k'} \frac{ \|\<2>_k\|_{B^{-1-2\kappa}_{4,\infty}}^2}{\langle k \rangle^{1+\kappa}}dk \Bigg)\|V_K+g_K\|_{L^4}^2
\\
&\lesssim
 \int_0^{k'} \frac{\|\<2>_k\|_{B^{-1-2\kappa}_{4,\infty}}^2}{\langle k \rangle^{1+\kappa}}dk \Bigg( \frac{\Big(\intx \cV_\beta(V_K+g_K) \dbar x\Big)^\frac 12}{\beta^\frac 32} + \frac{1}{\beta} \Bigg)
 \\
 &\leq
 \frac 1\beta  \int_0^{k'} \frac{\|\<2>_k\|_{B^{-1-2\kappa}_{4,\infty}}^2}{\langle k \rangle^{1+\kappa}}dk  + \frac{1}{4\varepsilon}\Bigg(\int_0^{k'} \frac{\|\<2>_k\|_{B^{-1-2\kappa}_{4,\infty}}^2}{\langle k \rangle^{1+\kappa}}dk \Bigg)^2
 \\
 &\quad\quad\quad
 + \frac{\varepsilon}{\beta^3}\intx \cV_\beta(V_K+g_K) \dbar x.
\end{split}
\end{equs}

For the remaining term in \eqref{eq: u estimate 0}, note that by the trivial embedding $H^1 \hookrightarrow H^{1-\kappa}$ and the bound \eqref{eq: l2 drift bound} applied to $R_{k'}$,
\begin{equs} \label{eq: u estimate eq3}
\| R_{k'} \|_{H^{1-\kappa}}^2
\lesssim
\intt \intx r_k^2 \dbar x dk.	
\end{equs}

Inserting \eqref{eq: u estimate eq2} and \eqref{eq: u estimate eq3} into \eqref{eq: u estimate 0} establishes \eqref{eq: u estimate}.


\subsubsection{Proof of \eqref{eq: cgw cubic}} \label{subsec: cubic bounds}

We start with the first integral in $\cR^{a,1}_K$. Fix $\kappa > 0$ and let $q$ be such that $(1+\kappa)^{-1} + q^{-1} = 1$. Then, by Young's inequality (and remembering $\beta > 1$),
\begin{equs} \label{eq:coarsecube1}
\Big| \intx \frac{2}{\beta} \vec{\<1>}_K V_K^3 \dbar x \Big|
&\leq
C_\varepsilon \intx |\vec{\<1>}_K|^q \dbar x + \varepsilon \intx \Big(\frac{V_K}{\hfbeta}\Big)^{2+2\kappa} |V_K|^{1+\kappa} \dbar x.
\end{equs}
Adding and subtracting $g_K$ into the second term on the righthand side and using the pointwise potential bound \eqref{tool: potential bound}, 
\begin{equs} \label{eq: coarsecube1.5}
\begin{split}
\intx &\Big(\frac{|V_K|}{\hfbeta}\Big)^{2+2\kappa} |V_K|^{1+\kappa} \dbar x
\\
&\lesssim
\intx \Big( \frac{|V_K+g_K|^4}{\beta^2}^{(\frac{1+\kappa}{2})} +  \big| \frac{g_K}{\hfbeta}\big|^{2+2\kappa} \Big) |V_K|^{1+\kappa} \dbar x
\\
&\lesssim
\intx \Bigg(\Big( \frac{\cV_{\beta}(V_K+g_K)}{\beta}\Big)^\frac{1+\kappa}{2} + 1 + \big| \frac{g_K}{\hfbeta}\big|_\infty^{2+2\kappa} \Bigg) |V_K|^{1+\kappa} \dbar x
\end{split}
\end{equs}
where we recall that $|\cdot|_\infty$ is the supremum norm.

By the bounds on $g_K$ \eqref{tool: g bounds} and $V_K$ \eqref{eq: v estimate}, taking $\kappa<1$ yields
\begin{equs} \label{eq:coarsecube2}
\begin{split}
\intx \Big( 1+&\big| \frac{g_K}{\hfbeta}\big|_\infty^{2+2\kappa} \Big)|V_K|^{1+\kappa} \dbar x 
\\
&\leq
C(\varepsilon,\kappa,\eta) + \varepsilon \| V_K\|_{L^2}^2
\\
&\leq
C\frac{N^\Xi_K}{N^3} + \varepsilon \Big( \intx \cV_\beta(V_K+g_K)\dbar x + \frac 12 \intt \intx r_k^2 \dbar x dk \Big).
\end{split}
\end{equs}
Above, we recall that $N^\Xi_K$ can contain constants $C=C(\eta)>0$.

For the remaining term on the righthand side of \eqref{eq: coarsecube1.5}, we reorganise terms and iterate the preceding argument:
\begin{equs} \label{eq:coarsecube3}
\begin{split}
\intx &\cV_{\beta}(V_K+g_K)^\frac{1+\kappa}2 \Big( \frac{|V_K|}{\hfbeta} \Big)^{1+\kappa} \dbar x
\\
&\lesssim
\intx \cV_{\beta}(V_K+g_K)^\frac{1+\kappa}2 \Bigg(  \big| \frac{g_K}{\hfbeta}\big|_\infty^{1+\kappa} + 1 +\frac{ \cV_{\beta}(V_K+g_K)^\frac{1+\kappa}4 }{\beta^\frac{1+\kappa}4} \Bigg) \dbar x
\\
&\leq
C(\varepsilon,\kappa,\eta)+ \varepsilon \intx \cV_\beta(V_K+g_K)\dbar x 
\end{split}
\end{equs}
provided that $\kappa < \frac 14$. 

We now estimate the second integral in $\cR^{a,1}_K$. Let $\tilde\kappa > 0$ be sufficiently small. Let $q$ be such that $\frac{3-\tilde\kappa}{4(1+\tilde\kappa)(1-\tilde\kappa)} + \frac 1q = 1$. Moreover, let $\theta = \frac{2\tilde\kappa(1-\tilde\kappa)}{(1+\tilde\kappa)(1-2\tilde\kappa)}$. By duality \eqref{tool: duality}, the fractional Leibniz rule \eqref{tool: Leibniz} and interpolation \eqref{tool: interpolation},
\begin{equs} \label{eq:rem_cubic_1}
\begin{split}
\Big| \intx \frac{4}{\beta} \<1>_K U_K^3 \dbar x\Big|
&\lesssim
\frac{1}{\beta} \| \<1>_K \|_{B^{-\frac 12 - \kappa}_{q,\infty}} \|U_K^3\|_{B^{\frac 12 + \kappa}_{\frac{4(1+\tilde\kappa)(1-\tilde\kappa)}{3-\tilde\kappa},1}}	
\\
&\lesssim
\frac{1}{\beta} \| \<1>_K \|_{B^{-\frac 12 - \kappa}_{q,\infty}} \| U_K \|_{B^{\frac 12 + \kappa}_{2+\tilde\kappa,1}} \| U_K \|_{L^{4-2\tilde\kappa}}^2
\\
&\lesssim
\frac{1}{\beta} \| \<1>_K \|_{B^{-\frac 12 - \kappa}_{q,\infty}} \| U_K \|_{H^{1-\kappa}}^{1-\theta} \| U_K \|_{L^{4-2\tilde\kappa}}^{2+\theta}. 
\end{split}
\end{equs}

By the change of variables \eqref{eq: intermediate integrated ansatz} in reverse, reorganising terms, Young's inequality, the bound on $U_K$ \eqref{eq: u estimate}, and using $\varepsilon < 1$,
\begin{equs}\label{eq:rem_cubic_2}
\begin{split}
	\eqref{eq:rem_cubic_1}
	&\leq
	C\frac{N^\Xi_K}{N^3} + \varepsilon \| U_K \|_{H^{1-\kappa}}^2
	\\
	&\quad 
	+ \| \<1>_K \|_{B^{-\frac 12 - \kappa}_{q,\infty}} \| U_K \|_{H^{1-\kappa}}^{1-\theta} \Big(\frac{1}{\beta^\frac{1}{2+\theta}}\| V_K \|_{L^{4-2\tilde\kappa}}\Big)^{2+\theta}
	\\
	&\leq
	C\frac{N^\Xi_K}{N^3} + \varepsilon \| U_K \|_{H^{1-\kappa}}^2 + \frac{1}{\hfbeta^\frac{8}{2+\theta}}\intx V_K^4 \dbar x
	\\
	&\leq
	C \frac{N^\Xi_K}{N^3} + \varepsilon \Big( \intx \cV_\beta(V_K+g_K) \dbar x + \frac 12 \intt \intx r_k^2 \dbar x dk \Big)
	\\
	&\quad\quad\quad
	+ \frac{1}{\hfbeta^\frac{8}{2+\theta}}\intx V_K^4 \dbar x.
\end{split}
\end{equs}

For the last term on the righthand side of \eqref{eq:rem_cubic_2}, we iterate the potential bound \eqref{tool: potential bound} and bound on $g_K$ \eqref{tool: g bounds} as in the estimate of \eqref{eq: coarsecube1.5}:
\begin{equs} \label{eq:rem_cubic_3}
\begin{split}
	\frac{1}{\hfbeta^\frac{8}{2+\theta}}\intx V_K^4 \dbar x
	&=
	\intx \Big( \frac{|V_K|}{\hfbeta} \Big)^\frac{4}{1+\frac \theta 2} |V_K|^\frac{2\theta}{2+\frac \theta 2} \dbar x
	\\
	&\lesssim 
	\intx \Big( \big| \frac{g_K}{\hfbeta}\big|^\frac{4}{1+\frac\theta2} + 1 \Big) |V_K|^\frac{2\theta}{2+\frac\theta 2} \dbar x
	\\
	&\quad\quad\quad
	+ \intx \frac{\cV_\beta(V_K+g_K)^\frac{1}{1+\frac\theta 2}}{\beta^\frac{2}{1+\frac\theta 2}} |V_K|^\frac{2\theta}{2+\frac\theta 2} \dbar x
	\\
	&\lesssim
	C(\varepsilon,\eta) + \frac {\varepsilon\eta} 2 \intx |V_K|^2 \dbar x
	\\
	&\quad\quad\quad
	+ \intx \cV_\beta(V_K+g_K)^\frac{1}{1+\frac\theta2} \frac{|V_K|^\frac{2\theta}{2+\frac\theta2}}{\beta^\frac{2}{1+\frac\theta2}} \dbar x
	\\
	&\lesssim
	C(\varepsilon,\eta) + \varepsilon \| V_K\|_{H^{\frac 12 - \kappa}}^2 
	\\
	&\quad\quad\quad
	+ \intx \cV_\beta(V_K+g_K)^\frac{1}{1+\frac\theta 2}\Big(1+\big| \frac{g_K}{\hfbeta}\big|^\frac{2\theta}{2+\frac\theta2} \Big)\dbar x 
	\\
	&\quad\quad\quad
	+ \intx \cV_\beta(V_K+g_K)^{\frac{1}{1+\frac\theta2} + \frac{\theta}{4+\theta}} \dbar x
	\\
	&\leq
	C(\varepsilon,\eta) + \varepsilon \Big( \|V_K\|_{H^{\frac 12 - \kappa}}^2  + \intx \cV_\beta(V_K+g_K)\dbar x \Big)
	\\
	&\leq
	C \frac{N^\Xi_K}{N^3} + 2\varepsilon \Big( \intx \cV_\beta(V_K+g_K) \dbar x + \frac 12 \intt \intx r_k^2 \dbar x dk \Big)
\end{split}
\end{equs}
where in the penultimate line we used Young's inequality and in the last line we have used \eqref{eq: v estimate}.

Combining \eqref{eq:coarsecube1}, \eqref{eq:coarsecube2}, \eqref{eq:coarsecube3}, \eqref{eq:rem_cubic_2}, and \eqref{eq:rem_cubic_3} establishes \eqref{eq: cgw cubic}. 

\subsubsection{Proof of \eqref{eq: cgw scalespace}} \label{subsec: scalespace bounds}

For any $\theta \in (0,1)$ let $\frac 1p = \frac \theta 4 + \frac{1-\theta}2$ and let $\frac 1{p'} = 1- \frac 1p$. Then, by duality \eqref{tool: duality}, the paraproduct estimate \eqref{tool: paraproduct}, the Bernstein-type bounds on the derivatives of the drift \eqref{eq: a priori derivative drift}, and bounds on the $\partial_k g_k^\flat$ \eqref{tool: partial g flat},
\begin{equs} \label{scalespace eq1}
\begin{split}
\Big| \intx \intt &\frac{12}\beta \<2>_k \pg (\partial_k V_k^\flat + \partial_k g_k^\flat) U_k dk \dbar x \Big|
\\
&\lesssim
\frac 1\beta \intt \| \<2>_k \pg (\partial_k V_k^\flat + \partial_k g_k^\flat) \|_{H^{-1+\kappa}} \|U_k \|_{H^{1-\kappa}} dk
\\
&\lesssim
\frac 1\beta \intt \| \<2>_k \|_{B^{-1+\kappa}_{p',2}} \| \partial_k V_k^\flat + \partial_k g_k^\flat \|_{L^p} \| U_k \|_{H^{1-\kappa}} dk
\\
&\lesssim
\sup_{0 \leq k \leq K} \|U_k\|_{H^{1-\kappa}} \frac 1\beta \| V_K + g_K \|_{B^{3\kappa}_{p,1}} \intt \| \<2>_k \|_{B^{-1+\kappa}_{p',2}} \frac{dk}{\langle k \rangle^{1+3\kappa}}	
\end{split}
\end{equs}
where in the last inequality we have reordered terms.

Then,  
\begin{equs}\label{scalespace eq2}
\begin{split}
\eqref{scalespace eq1}
&\lesssim
\sup_{0 \leq k \leq K} \| U_k\|_{H^{1-\kappa}} \frac 1\beta \| V_K + g_K \|_{B^0_{4,\infty}}^\theta \|V_K + g_K \|_{B^{6\kappa}_{2,1}}^{1-\theta} 
\\
&\quad\quad\quad
\times \intt \| \<2>_k \|_{B^{-1-\kappa}_{p',2}} \frac{dk}{\langle k \rangle^{1+\kappa}}
\\
&\lesssim
\sup_{0 \leq k \leq K} \| U_k\|_{H^{1-\kappa}}\frac{\| V_K + g_K\|_{L^4}^\theta}{\beta^\theta} \Big(\frac {\|V_K\|^{1-\theta}_{H^{\frac 12 - \kappa}}}{\beta^{1-\theta}} + 1 \Big) 
\\
&\quad\quad\quad
\times \intt \| \<2>_k \|_{B^{-1-\kappa}_{p',2}} \frac{dk}{\langle k \rangle^{1+\kappa}}
\\
&\leq 
C(\varepsilon)\Bigg(1+ \Big( \intt \| \<2>_k \|_{B^{-1-\kappa}_{p',2}} \frac{dk}{\langle k \rangle^{1+\kappa}}\Big)^\frac{4}{4-\theta} \Bigg) 
\\
&\quad\quad\quad
+ \frac{\varepsilon}{2}\Bigg( \| V_K \|_{H^{\frac 12 - \kappa}}^2 + \sup_{0\leq k \leq K}\| U_k \|_{H^{1-\kappa}}^2 + \frac 1{\beta^4}\|V_K+g_K\|_{L^4}^4 \Bigg)
\\
&\leq
C(\varepsilon,\eta) \frac{N^\Xi_K}{N^3} + \varepsilon \Big( \intx \cV_\beta(V_K+g_K) \dbar x + \frac 12 \intt \intx r_k^2 \dbar x dk \Big)
\end{split}
\end{equs}
where in the first line we have used Bernstein's inequality \eqref{tool: bernstein ball}; in the second line we have used interpolation \eqref{tool: interpolation}; in the penultimate line we used Young's inequality; and in the last line we have used the bounds on $V_K$ \eqref{eq: v estimate}, $U_k$ \eqref{eq: u estimate}, together with the potential bound \eqref{tool: potential bound norm}.

In order to bound the second integrand in $\cR^{a,2}_K$, we use Fubini's theorem, the Cauchy-Schwarz inequality, the bounds on $V_k^\flat$ \eqref{eq: a priori flat drift} and $\partial_k V_K^\flat$ \eqref{eq: a priori derivative drift}, and the bounds on $g_K$ \eqref{tool: g bounds} to obtain
\begin{equs}\label{scalespace eq3}
\begin{split}
\Big| \intx \intt &\frac{2\gamma_k}{\beta^2} (\partial_k V_k^\flat + \partial_k g_k^\flat)(V_k^\flat + g_k^\flat) dk \dbar x \Big|
	\\
	&\lesssim
	\frac 1{\beta^2} \intt \gamma_k \| \partial_k V_k^\flat + \partial_k g_k^\flat \|_{L^2} \| V_k^\flat + g_k^\flat \|_{L^2} dk
	\\
	&\lesssim
	\frac{1}{\beta^2} \|V_K + g_K \|_{H^{2\kappa}} \| V_K + g_K \|_{L^2} \intt \frac{\gamma_k}{\langle k \rangle^{\kappa}} \frac{dk}{\langle k \rangle^{1+\kappa}}
	\\
	&\lesssim
	\Big( \frac{\| V_K \|_{H^{\frac 12 -\kappa}}}{\beta^2} + \frac{1}{\beta^\frac 32} \Big)\| V_K + g_K \|_{L^4} 
\end{split}
\end{equs}
where in the last inequality we have used the observation made in Remark \ref{rem: growth of renorm} that $|\gamma_k|\lesssim \log \langle k \rangle$.

Thus, by Young's inequality (applied to each term after expanding the sum), the potential bound \eqref{tool: potential bound norm}, and the bound on $V_K$ \eqref{eq: v estimate},
\begin{equs}\label{scalespace eq4}
\begin{split}
\eqref{scalespace eq3}
&\leq
C(\varepsilon,\eta) + \varepsilon \Bigg( \| V_K \|_{H^{\frac 12- \kappa}}^2 + \Big( \frac 1{\beta^8} + \frac 1{\beta^6} \Big) \| V_K+g_K\|_{L^4}^4 \Bigg)
\\
&\leq
C(\varepsilon,\eta) \frac{N^\Xi_K}{N^3} + \varepsilon \Big( \intx \cV_\beta(V_K+g_K) \dbar x + \frac 12 \intt \intx r_k^2 \dbar x dk \Big).
\end{split}	
\end{equs}

Combining \eqref{scalespace eq2} and \eqref{scalespace eq4} yields \eqref{eq: cgw scalespace}.

\subsubsection{Proof of \eqref{eq: cgw commutator}} \label{subsec: commutator bound}
We write $\cR^{a,3}_K = I_1 + I_2 + I_3$, where
\begin{equs}
I_1 
&=
\intx \intt \frac{72}{\beta^2} \Big( \cJ_k (\<2>_k \pg (V_k^\flat + g_k^\flat)) \Big)^2 - \Big( \cJ_k \<2>_k \pg (V_k^\flat + g_k^\flat) \Big)^2 dk dx
\\
I_2
&= 
\intx \intt \frac{72}{\beta^2}\Big(\cJ_k \<2>_k \pg (V_k^\flat + g_k^\flat) \Big)^2 
\\
&\quad\quad\quad\quad\quad
- \Big( \big( \cJ_k \<2>_k \pg (V_k^\flat + g_k^\flat) \big) \pe \cJ_k \<2>_k \Big) (V_k^\flat + g_k^\flat) dk dx
\\
I_3
&=
\intx \intt \frac{72}{\beta^2} \Big( \big(\cJ_k \<2>_k \pg (V_k^\flat + g_k^\flat) \big) \pe \cJ_k \<2>_k 
\\
&\quad\quad\quad\quad\quad
- (\cJ_k \<2>_k \pe \cJ_k \<2>_k)(V_k^\flat + g_k^\flat) \Big) \Big(V_k^\flat + g_k^\flat \Big) dk \dbar x.
\end{equs}

Let $\theta \in (0,1)$ be sufficiently small and let $\frac 1p = \frac \theta 4 + \frac{1-\theta}2$, $\frac 1q = \frac{1-\theta}2$ and $\frac 1{p'} = \frac 12-\frac 1p$, $\frac 1{q'} = \frac 12 - \frac 1 q$. Then,
\begin{equs}\label{commutator I_1 eq1}
\begin{split}
	|I_1|
	&\lesssim
	\frac{1}{\beta^2} \intt \| \cJ_k \big( \<2>_k \pg (V_k^\flat + g_k^\flat) \big) - \cJ_k \<2>_k \pg (V_k^\flat + g_k^\flat) \|_{H^{2\kappa}}
	\\ 
	&\quad\quad\quad\quad\quad
	\times \|\cJ_k \big( \<2>_k \pg (V_k^\flat + g_k^\flat) \big) + \cJ_k \<2>_k \pg (V_k^\flat + g_k^\flat)\|_{H^{-2\kappa}} dk
	\\
	&\lesssim
	\frac{1}{\beta^2} \intt \| \<2>_k \|_{B^{-1-\kappa}_{p',q'}} \|V_k^\flat + g_k^\flat \|_{B^{4\kappa}_{p,q}} 
	\\
	&\quad\quad\quad\quad\quad
	\times \Big( \| \cJ_k\big(\<2>_k \pg (V_k^\flat + g_k^\flat) \big)\|_{H^{-2\kappa}} + \| \cJ_k \<2>_k \pg (V_k^\flat + g_k^\flat)\|_{H^{-2\kappa}} \Big)dk
	\\
	&\lesssim 
	\frac{1}{\beta^2} \intt \| \<2>_k \|_{B^{-1-\kappa}_{p',q'}} \|V_k^\flat + g_k^\flat \|_{B^{4\kappa}_{p,q}} \| \<2>_k \|_{B^{-1-2\kappa}_{4,2}} \| V_k^\flat + g_k^\flat \|_{L^4} \frac{dk}{\langle k \rangle}
	\end{split}
\end{equs}
where the first inequality is by duality \eqref{tool: duality}; the second inequality is by the commutator estimate \eqref{tool: multiplier commutator} and the triangle inequality; and the third inequality is by the multiplier estimate \eqref{tool: multiplier estimate} and the paraproduct estimate \eqref{tool: paraproduct}.

Thus,
\begin{equs}\label{commutator I_1 eq2}
\begin{split}
	\eqref{commutator I_1 eq1}
	&\lesssim
	\frac{1}{\beta^2} \intt \| \<2>_k \|_{B^{-1-\kappa}_{p',q'}} \| V_k^\flat + g_k^\flat \|_{B^{4\kappa}_{p,q}} \| \<2>_k \|_{B^{-1-\kappa}_{4,2}} \| V_k^\flat + g_k^\flat \|_{L^4} \frac{dk}{\langle k \rangle^{1+\kappa}}
	\\
	&\lesssim
	\frac{1}{\beta^2} \|V_K + g_K\|_{H^\frac{4\kappa}{1-\theta}}^{1-\theta}  \| V_K + g_K\|_{L^4}^{1+\theta} \intt \| \<2>_k \|_{B^{-1-\kappa}_{p',q'}} \| \<2>_k \|_{B^{-1-\kappa}_{4,2}} \frac{dk}{\langle k \rangle^{1+\kappa}}
\end{split}
\end{equs}
where the first inequality is by Bernstein's inequality \eqref{tool: bernstein ball}; and the second inequality is by the $\flat$-bounds applied to $V_k^\flat + g_k^\flat$ \eqref{eq: a priori flat drift}, interpolation \eqref{tool: interpolation}, and the trivial bound $\| V_K + g_K \|_{B^{4\kappa \theta}_{4,\infty}} \lesssim \| V_K+g_K\|_{L^4}$.

By applying Young's inequality, the potential bound \eqref{tool: potential bound norm}, and the bound on $V_K$ \eqref{eq: v estimate}, we have
\begin{equs}\label{commutator I_1 eq3}
\begin{split}
\eqref{commutator I_1 eq2}
&\leq
C\frac{N^\Xi_K}{N^3} + \varepsilon \Bigg( \|V_K + g_K\|_{H^{\frac{4\kappa}{1-\theta}}}^2 + \frac{1}{\beta^{\frac{8}{1+\theta}}} \|V_K + g_K\|_{L^4}^4 \Bigg)
\\
&\leq
C \frac{N^\Xi_K}{N^3} + \varepsilon \Big( \intx \cV_\beta(V_K+g_K) \dbar x + \frac 12 \intt \intx r_k^2 \dbar x dk \Big).
\end{split}
\end{equs}

Now consider $I_2$. Using the commutator estimate \eqref{tool: commutator 1} with $f = \cJ_k \<2>_k$, $g = V_k^\flat + g_k^\flat$ and $h = \cJ_k \<2>_k \pg (V_k^\flat + g_k^\flat)$, followed by the paraproduct estimate \eqref{tool: paraproduct}, we obtain 
\begin{equs}\label{commutator I_2 eq1}
\begin{split}
I_2
&\lesssim
\frac{1}{\beta^2} \intt \| \cJ_k \<2>_k \|_{B^{-2\kappa}_{6,\infty}} \|V_k^\flat + g_k^\flat \|_{H^{4\kappa}} \|\cJ_k \<2>_k \pg (V_k^\flat + g_k^\flat) \|_{B^{-2\kappa}_{3,2}} dk
\\
&\lesssim
\frac{\| V_K + g_K \|_{H^{4\kappa}} \| V_K + g_K \|_{L^4}}{\beta^2} \intt \| \<2>_k \|_{B^{-\kappa}_{12,2}}^2 \frac{dk}{\langle k \rangle^{1+2\kappa}}.
\end{split}
\end{equs}
By applying Young's inequality, the potential bound \eqref{tool: potential bound norm}, and the a priori bound on $V_K$ \eqref{eq: v estimate},
\begin{equs} \label{commutator I_2 eq2}
\begin{split}
\eqref{commutator I_2 eq1}
&\leq
C(\varepsilon,\eta)\frac{N^\Xi_K}{N^3} + \varepsilon \Bigg( \| V_K + g_K \|_{H^{4\kappa}}^2 + \frac 1{\beta^8}\|V_K+g_K\|_{L^4}^4 \Bigg)
\\
&\leq
C(\varepsilon,\eta) \frac{N^\Xi_K}{N^3} + \varepsilon \Big( \intx \cV_\beta(V_K+g_K) \dbar x + \frac 12 \intt \intx r_k^2 \dbar x dk \Big).
\end{split}	
\end{equs}
where the final inequality uses the multiplier estimate \eqref{tool: multiplier estimate}, the $\flat$-bounds applied to $V_K+g_K$ \eqref{eq: a priori flat drift}, and Bernstein's inequality \eqref{tool: bernstein annulus}.

For $I_3$, we apply duality \eqref{tool: duality}, the commutator estimate \eqref{tool: commutator 2} with $f=h = \cJ_k \<2>_k$ and $g=V_k^\flat + g_k^\flat$, followed by the $\flat$-bounds applied to $V_K+g_K$ \eqref{eq: a priori flat drift}, to obtain
\begin{equs} \label{commutator I_3 eq1}
\begin{split}
I_3
&\lesssim
\frac{1}{\beta^2} \intt \| \big(\cJ_k \<2>_k \pg (V_k^\flat + g_k^\flat) \big) \pe \cJ_k \<2>_k - (\cJ_k \<2>_k \pe \cJ_k \<2>_k)(V_k^\flat + g_k^\flat)\|_{B^\kappa_{\frac 43, \infty}} 
\\
&\quad\quad\quad\quad\quad
\times \| V_k^\flat + g_k^\flat \|_{B^{-\kappa}_{4,1}} dk
\\
&\lesssim
\frac{1}{\beta^2} \intt \| \cJ_k \<2>_k \|_{B^{-2\kappa}_{8,\infty}}^2 \| V_k^\flat + g_k^\flat \|_{B^{5\kappa}_{2,\infty}} \| V_k^\flat + g_k^\flat \|_{L^4} dk
\\
&\lesssim
\frac{1}{\beta^2} \| V_K + g_K \|_{H^{5\kappa}}\| V_K + g_K \|_{L^4} \intt \|\<2>_k\|^2_{B^{-\kappa}_{8,\infty}} \frac{dk}{\langle k \rangle^{1+2\kappa}}
\\
&\leq
C(\varepsilon,\eta) \frac{N^\Xi_K}{N^3} + \varepsilon \Big( \intx \cV_\beta(V_K+g_K) \dbar x + \frac 12 \intt \intx r_k^2 \dbar x dk \Big)
\end{split}
\end{equs}
where in the last line we have used Young's inequality, the potential bound \eqref{tool: potential bound norm}, and the bound on $V_K$ \eqref{eq: v estimate} as in \eqref{commutator I_2 eq2}.

Using that $\cR^{a,3}_K = I_1 + I_2 + I_3$, the estimates \eqref{commutator I_1 eq3}, \eqref{commutator I_2 eq2}, and \eqref{commutator I_3 eq1} establish \eqref{eq: cgw commutator}.

\subsection{A lower bound on the effective Hamiltonian} \label{subsec: pointwise bound on effective}

The following lemma, based on \cite[Theorem 3.1.1]{GJS76-3}, gives a $\beta$-independent lower bound on $\cH^{\rm eff}_K(Z_K)$ in terms of the $L^2$-norm of the fluctuation field $Z_K^\perp = Z_K - \vec Z_K$, where we recall $Z_K = \vec{\<1>}_K + V_K + g_K$ and $\vec Z_K (x) = Z_K(\Box)$ for $x \in \Box \in \BBN$. This is useful for us because the latter can be bounded in a $\beta$-independent way (see Section \ref{subsec: proof z lower}).

\begin{lem}\label{lem:pointwise}	
There exists $C>0$ such that, for any $\zeta > 0$ and $K \in (0,\infty)$,
\begin{equs} \label{eq: pointwise lem}
\cH^{\rm eff}_K (Z_K)
\geq
-CN^3 -\zeta \intx \big(Z_K^\perp \big)^2 dx	
\end{equs}
provided $\eta < \min\Big( \frac 1{32}, \frac {2\zeta}{49} \Big)$.
\end{lem}

\begin{proof}

First, we write
\begin{equs}
\cH^{\rm eff}_K(Z_K)
=
\sum_{\sBox \in \BBN} \int_\sBox \frac 12 \cV_{\beta,N,K}(Z_K) - \frac \eta 2 (Z_K-h)^2 - \log \Big(\chi_{\sigma(\sBox)} \big(Z_K(\Box)\big) \Big) dx.	
\end{equs}
Fix $x \in \Box \in \BBN$. Without loss of generality, assume $\sigma(x) = 1$ and, hence, $h(x) = \hfbeta$. Define
\begin{equs}
I(x) 
= 
\frac 12 \cV_\beta(Z_K(x)) - \frac{\eta}{2} (Z_K(x)-\hfbeta)^2 - \log \chi_{+}(\vec Z_K(x)).
\end{equs}
In order to show \eqref{eq: pointwise lem}, it suffices to show that, for some $C > 0$,
\begin{equs}
I(x) + \zeta Z_K^\perp(x)^2 \geq -C.	
\end{equs}

The fundamental observation is that  $Z_K(x) \mapsto \frac 12\cV_\beta(Z_K(x))$ can be approximated from below near the minimum at $Z_K(x) = \hfbeta$ by the quadratic $Z_K(x) \mapsto \frac \eta 2 (Z_K(x)-\hfbeta)^2$ provided $\eta$ is taken sufficiently small. Indeed, we have
\begin{equs}
\frac 12\cV_\beta(Z_K(x)) - \frac \eta 2 (Z_K(x) - \hfbeta)^2 
&=
\frac 1{2\beta}(Z_K(x) -\hfbeta)^2 \Big( (Z_K(x) + \hfbeta)^2 - \eta \beta \Big)	
\end{equs}
which is non-negative provided $|Z_K(x) + \hfbeta| \geq \sqrt{\eta \beta}$. Thus, this approximation is valid except for the region near the opposite potential well satisfying $(-1-\sqrt\eta)\hfbeta < Z_K(x) < (-1+\sqrt\eta)\hfbeta$ (see Figure \ref{fig: potential}). When $Z_K(x)$ sits in this region, we split $Z_K(x) = \vec Z_K(x) + Z_K^\perp(x)$ and observe that:
\begin{itemize}
\item either the deviation to the opposite well is caused by $\vec Z_K(x)$, which is penalised by the logarithm in $I(x)$;
\item or, the deviation is caused by $Z_K^\perp(x)$, which produces the integral involving $Z_K^\perp$ in \eqref{eq: pointwise lem}.
\end{itemize}

\begin{figure}[!htb]
\centering
	\begin{tikzpicture}
\begin{axis}[
	axis lines = middle,
	width = 10cm, height=8cm,
	xmax=3.6,
	ymax=1.2,
    samples=100,
    xtick={-1.41421,  1.41421},
    xticklabels ={-$\sqrt{\beta}$, $\sqrt{\beta}$},
    ytick={1.1},
    ytick style={draw=none},
    yticklabel={$\frac \beta 2$},
    enlarge y limits={rel=0.7},
    enlarge x limits = {rel=0.05}
	]
\addplot [domain=-2.2:2.15, color=black,thick] {0.25*(x^4) - 1*(x^2) + 1} node[right, font=\small] {$\frac 12 \mathcal{V}_\beta (Z_K(x))$};
\addplot [domain = 2.02:2.2, color=black,thick] {0.25*(x^4) - 1*(x^2) + 1};
\addplot [domain=-3.3:2.2, color=blue,thick] {0.1*(x-1.41421)^2} node[above, pos=1.09, font=\small] {\color{blue} $\frac \eta 2 (Z_K(x) - \sqrt{\beta})^2$};
\draw[gray, dashed, name path = A] (-2.04666, 0) -- (-2.04666, 3);
\draw[gray, dashed, name path = B] (-0.781758,0) -- (-0.781758, 3);
\addplot[gray, fill opacity=0.20] fill between [of=A and B,soft clip= 
{domain=-12:2}];

\end{axis}
\end{tikzpicture}

	\caption{Plot of $\cV_\beta(Z_K(x))$ and $\frac \eta 2 (Z_K(x) - \hfbeta)^2$.} 
	\label{fig: potential}
\end{figure}

Motivated by these observations, we split the analysis of $I(x)$ into two cases. First we treat the case $Z_K(x) \in \RR\setminus\Big( -\frac{4\hfbeta}3, -\frac{2\hfbeta}3 \Big)$. Under this condition, we have 
\begin{equs}
\frac 12 \cV_\beta(Z_K(x)) \geq \eta (Z_K(x)-\hfbeta)^2
\end{equs}
provided that $\eta \leq \frac 1 9$. 
Since $\chi_+(\cdot) \leq 1$, $-\log\chi_+(\cdot) \geq 0$. It follows that $I(x) \geq 0$.

Now let $Z_K(x) \in \Big( -\frac{4\hfbeta}3, -\frac{2\hfbeta}3  \Big)$. Necessarily, either $\vec Z_K(x) \leq - \frac \hfbeta 3$ or $Z_K^\perp(x) \leq - \frac\hfbeta 3$. 

We first assume that $\vec Z_K(x) \leq - \frac \hfbeta 3$. By standard bounds on the Gaussian error function (see e.g. \cite[Lemma 2.6.1]{GJS76-3}), for any $\theta \in (0,1)$ there exists $C=C(\theta)>0$ such that 
\begin{equs}
-\log \chi_+\big(Z_K(\Box)\big) \geq - \theta (\vec Z_K(x))^2 + C.
\end{equs}
Applying this with $\theta \in (\frac 12, 1)$ and that, by our assumption, $\vec Z_K(x) - \hfbeta > 4\vec Z_K(x)$,
\begin{equs}
I(x) + \zeta (Z_K^\perp(x))^2 
&\geq
-\frac \eta 2 ( Z_K^\perp(x) + \vec Z_K (x) - \hfbeta)^2 - \log\chi_+\big(Z_K(\Box)\big) + \zeta (Z_K^\perp(x))^2
\\
&\geq
(\zeta-\eta) (Z_K^\perp(x))^2 - 16\eta(\vec Z_K(x))^2 - \tilde\theta (\vec Z_K(x)^2) - C
\\
&\geq
-C
\end{equs}
provided $\eta < \min\Big( \zeta, \frac{1}{32} \Big)$. 

Finally, assume that $Z_K^\perp(x) < - \frac\hfbeta 3$. Since $Z_K(x) - \hfbeta \in \Big( -\frac{7\hfbeta}3, -\frac{-5\hfbeta}3 \Big)$, we have
\begin{align}
\begin{split}
I(x) + \zeta (Z_K^\perp(x))^2
&\geq
- \frac{49\eta}{18} \beta + \zeta (Z_K^\perp(x))^2
\geq 0
\end{split}
\end{align}
provided that $\eta \leq \frac{2\zeta}{49}$.

\end{proof}

\subsection{Proof of Proposition \ref{prop:zbounds}}

\subsubsection{Proof of the lower bound on the free energy \eqref{eq:zboundlower}} \label{subsec: proof z lower}

We derive bounds uniform in $\sigma$ for each term in the expansion \eqref{partition expansion}. Since there are $2^{N^3}$ terms, this is sufficient to establish \eqref{eq:zboundlower}. Fix $\sigma \in \{\pm 1\}^\BBN$.

Recall 
\begin{equs}\label{eq:recall}
-\log \sZ_{\beta,N,K}^\sigma = -\log \EE_N e^{-\cH_{\beta,N,K}^\sigma} + F^\sigma_{\beta,N,K}.
\end{equs}

Let $C_P > 0$ be the sharpest constant in the Poincar\'e inequality \eqref{tool: poincare} on unit boxes. Note that $C_P$ is independent of $N$. Fix $\zeta < \frac 1{8C_P}$ and let $\varepsilon = 1-8C_P\zeta > 0$. By Proposition \ref{prop: killing} and Lemma \ref{lem:pointwise} there exists $C=C(\zeta,\eta)>0$ such that, for every $v \in \HH_{b,K}$,
\begin{equs}
\Psi_K(v)
&=
\cH_{\beta,N,K}^\sigma(\<1>_K + V_K) + \frac 12 \intx\intt v_k^2 dk dx 
\\
&\approx
\sum_{i=1}^4 \cR^i_K + \cH^\mathrm{eff}_K(Z_K) + \frac 12 \intx \cV_\beta(V_K+g_K) dx + \frac 12 \intx \intt r_k^2 dk dx
\\
&\geq
-C(\varepsilon) N^\Xi_K + \cH^\mathrm{eff}_K + \frac{1-\varepsilon}2 \Big( \intx \cV_\beta(V_K+g_K) dx + \frac 12\intx\intt r_k^2 dk dx \Big).
\\
&\geq
-C(\zeta) N^\Xi_K - \zeta \intx (Z_K^\perp)^2 dx 
\\
&\quad\quad\quad\quad\quad
+ 4\zeta C_P \Big( \intx \cV_\beta(V_K+g_K) dx + \frac 12\intx\intt r_k^2 dk dx \Big)
\end{equs}
provided $\eta < \frac{2\zeta}{49} < \frac 1{196 C_P}$.

Note that for any $f \in L^2$, $\intx (f^\perp)^2 dx \leq \intx f^2 dx$. Therefore, using the inequality $(a_1 + a_2 + a_3 + a_4)^2 \leq 4(a_1^2 + a_2^2 + a_3^2 + a_4^2)$ and that $Z_K^\perp(x) = (V_K+g_K)^\perp(x)$, we have
\begin{equs}
\intx (Z_K^\perp)^2 dx 
&\leq
4\intx \frac{16}{\beta^2} \Big(\<30>_K\Big)^2  + \frac{144}{\beta^2} \Bigg( \intt \cJ_k^2 \<2>_k \pg (V_k^\flat + g_k)) dk \Bigg)^2 
\\
&\quad\quad\quad
+ (R_K^\perp)^2 + (g_K^\perp)^2 dx.
\end{equs}

Arguing as in \eqref{eq: u estimate eq2},
\begin{equs}
4\intx \frac{16}{\beta^2} \Big(\<30>_K\Big)^2  + \frac{144}{\beta^2} \Bigg( \intt \cJ_k^2 \<2>_k \pg (V_k^\flat + g_k)) dk \Bigg)^2  dx
\\
\leq
C(\zeta, C_P) N^\Xi_K + \frac{4\zeta C_P}{\beta^3} \intx \cV_\beta(V_K+g_K)dx.
\end{equs}

By the Poincar\'e inequality \eqref{tool: poincare} on unit boxes,
\begin{equs}
\intx (R_K^\perp)^2 dx
&=
\sum_{\Box \in \BBN} \int_\Box \Big( R_K - \int_\Box R_K dx \Big)^2 dx
\\
&\leq
C_P\sum_{\Box \in \BBN} \int_\Box |\nabla R_K|^2 dx
\\
&\leq
C_P\intx \intt r_k^2 dk dx
\end{equs}
where in the last inequality we used that $\intx |\nabla R_K|^2 dx \leq \| R_K\|_{H^1}^2$ and Lemma \ref{lem: drift bound} (applied to $R_K$).

Similarly, by the Poincar\'e inequality \eqref{tool: poincare} and the (trivial) bound $\|\nabla g_K\|_{L^2}^2 \leq \| \nabla \tilde g_K \|_{L^2}^2$ \eqref{tool: grad g g tilde},
\begin{equs}
\intx (g_K^\perp)^2 dx 
\leq 
C_P\intx |\nabla g_K|^2 dx
\leq
C_P\intx |\nabla \tilde g_K|^2 dx.
\end{equs}

Then, recalling that $\beta > 1$,
\begin{equs}
\EE \Psi_K(v)
&\geq
\EE \Bigg[ - C N^\Xi_K + 4\zeta C_P\Big( 1 - \frac{1}{\beta^3} \Big) \intx \cV_\beta(V_K+g_K) dx
\\
&\quad\quad\quad
+ \Big( 4\zeta C_P - 4\zeta C_P \Big) \intx \intt r_k^2 dk dx - 4\zeta C_P \intx |\nabla \tilde g_K|^2 dx \Bigg]
\\
&\geq 
\EE \Bigg[ - C N^\Xi_K - 4\zeta C_P \intx |\nabla \tilde g_K|^2 dx \Bigg]
\end{equs}
from which, by Proposition \ref{prop: bd}, we obtain
\begin{equs}
-\log\EE_N e^{-\cH_{\beta,N,K}^\sigma} 
\geq 
-CN^3 -4\zeta C_P \intx |\nabla \tilde g_K|^2 dx.
\end{equs}

Inserting this into \eqref{eq:recall} and using that $F_{\beta,N,K}^\sigma \geq \frac 12 \intx |\nabla \tilde g_K|^2 dx$ (see \eqref{def: f}) yields:
\begin{equs}
-\log\sZ_{\beta,N,K}^\sigma 
\geq 
-CN^3 + \Big( \frac 12 - 4\zeta C_P \Big) \intx |\nabla \tilde g_K|^2 dx
\geq 
-CN^3
\end{equs}
which establishes \eqref{eq:zboundlower}.

\subsubsection{Proof of the upper bound on the free energy \eqref{eq:zboundupper}} \label{subsec: proof upper bound}

We (globally) translate the field to one of the minima of $\cV_\beta$: this kills the constant $\beta$ term. Thus, under the translation $\phi = \psi + \hfbeta$,
\begin{equs}
\sZ_{\beta,N,K} 
= 
\EE_N e^{-\cH_{\beta,N,K}^+(\psi_K)}
\end{equs}
where
\begin{equs}
\cH_{\beta,N,K}^+(\psi_K) 
&=
\intx \cV_\beta^+(\psi_K) - \frac{\gamma_K}{\beta^2}:(\psi_K + \sqrt \beta)^2: - \delta_K - \frac \eta 2 :\psi_K^2:dx 
\end{equs}
and
\begin{equs}
\cV_\beta^+(a) = \frac{1}{\beta} a^2(a+2\hfbeta)^2 =\frac{1}{\beta}a^4 + \frac{4}{\hfbeta}a^3 + 4a^2. 
\end{equs}

We apply the Proposition \ref{prop: bd} to $\sZ_{\beta,N,K}$ with the infimum taken over $\HH_K$. In order to obtain an upper bound, we choose a particular drift in the corresponding stochastic control problem \eqref{eq: boue dupuis}. Following \cite{BG19}, we seek a drift that satisfies sufficient moment/integrability conditions with estimates that are extensive in $N^3$, as formalised in Lemma \ref{lem: bg drift} below. Such a drift is constructed using a fixed point argument, hence the need to work in the Banach space $\HH_K$ as opposed to $\HH_{b,K}$.

\begin{lem}\label{lem: bg drift}
There exist processes $\cU_{\leq} \<2>_\bullet$ and $\cU_{>} \<2>_\bullet$ satisfying $\cU_{\leq>} \<2>_\bullet + \cU_{\geq} \<2>_\bullet = \<2>_\bullet$ and a unique fixed point $\check{v} \in \HH_K$ of the equation 
\begin{equs} \label{eq: the bg drift}
\check{v}_k
=	
-\frac 4\beta \cJ_k \<3>_k - \frac{12}\hfbeta \cJ_k \<2>_k - \frac{12}{\beta} \cJ_k ( \cU_{>}\<2>_k \pg \check{V}_k^\flat)
\end{equs}
where $\check{V}_K = \intt \cJ_k \check{v}_k dk$, such that the following estimate holds: for all $p \in [1,\infty)$, there exists $C=C(p,\eta)>0$ such that, for all $\beta > 1$,
\begin{equs} \label{bg drift nice}
	\EE \Bigg[ \intx |\check{V}_K|^p dx + \frac 12 \intx \intt \check{r}_k^2 dk dx \Bigg]
	\leq
	CN^3
\end{equs}
where $\check{r}_k = -\frac{12}{\beta} \cJ_k(\cU_{\leq} \<2>_k \pg \check{V}_k^\flat)$.
\end{lem}

\begin{proof}
See \cite[Lemma 6]{BG19}. Note that the key difficulty lies in obtaining the right $N$ dependence in \eqref{bg drift nice}. Due to the paraproduct in the definition of \eqref{eq: the bg drift}, one can show that this requires finding a decomposition of $\<2>_k$ such that $\cU_{>} \<2>_k$ has Besov-H\"older norm that is uniformly bounded in $N^3$ (see Proposition \ref{prop:paraproduct}). Such a bound is not true for $\<2>_k$ (see Remark \ref{rem: stochastic norms}). This is overcome by defining $\cU_{\leq} \<2>_k$ to be a random truncation of the Fourier series of $\<2>_k$, where the location of the truncation is chosen to depend on the Besov-H\"older norm of $\<2>_k$.  
\end{proof}

For $v \in \HH_K$, let
\begin{equs}
\Psi_K^+(v) 
=
\cH^+_{\beta,N,K}(\<1>_K + V_K) + \frac 12 \intx\intt v_k^2 dk dx
\end{equs}
and define $\cR_K^+$ by
\begin{equs}
\Psi_K^+(v)
=
\cR_K^+ - \frac \eta 2 \intx V_K^2 dx + \intx \cV_\beta^+(V_K) dx + \frac 12 \intx\intt v_k^2\ dk dx.
\end{equs}
We observe
\begin{equs} \label{upperbound eq1}
\Psi_K^+(v)
\leq
\cR_K^+ + \intx \cV_\beta^+(V_K) dx + \frac 12 \intx\intt v_k^2 dkdx.	
\end{equs}
Thus, unlike the lower bound, the negative mass $-\frac \eta 2 \intx V_K^2 dx$ can be ignored in bounding the upper bound on the free energy. 

Now fix $\check v$ as in \eqref{eq: the bg drift}. Arguing as in Proposition \ref{prop: killing}, there exists $\tilde \cR^+_K$ such that
\begin{equs} \label{upperbound eq2}
\cR^+_K + \frac 12 \intx\intt \check v_k^2 dk dx 
\approx
\tilde\cR_K^+ + \frac 12 \intx \intt \check r_k^2 dk dx
\end{equs}
and $\tilde \cR_K^+$ satisfies the following estimate: for every $\varepsilon > 0$, there exists $C=C(\varepsilon,\eta) > 0$ such that, for all $\beta > 1$,
\begin{equs} \label{upperbound eq3}
|\tilde\cR_K^+| 
\leq 
C N^\Xi_K + \varepsilon \Big( \intx \cV_\beta^+(\check V_K) dx + \frac 12 \intx \intt \check r_k^2 dk dx \Big).
\end{equs}
Above, we have used that the moment conditions \eqref{bg drift nice} are sufficient for conclusions of Lemma \ref{lem: martingales} to apply to $\check v$. 

Thus, by \eqref{upperbound eq1}, \eqref{upperbound eq2}, and \eqref{upperbound eq3},
\begin{equs} \label{upperbound eq4}
\EE [\Psi_K^+(\check v)] 
\leq 
CN^3  + (1+\varepsilon) \EE \Big[ \intx \cV_{\beta}^+(\check V_K) + \frac 12 \intx \intt \check r_k^2 dk dx \Big].
\end{equs}

By Young's inequality, $\frac{1}{\beta}a^4 + \frac{4}{\hfbeta}a^3 + 4a^2 \leq 3a^4 + 6a^2 \leq 9a^4 + 9$ for all $\beta > 1$ and $a \in \RR$. Thus,
\begin{equs} 
\intx \cV_\beta^+(\check V_K) dx 
\leq 
9 \intx \check V_K^4 dx + 9N^3.
\end{equs}
Inserting this into \eqref{upperbound eq4} and using the moment estimates on the drift \eqref{bg drift nice} yields
\begin{equs}
\EE [\Psi_K^+(\check v)]
\leq
CN^3 + (1+\varepsilon) \EE \Big[ 9 \intx \check V_K^4 dx + \frac 12 \intx \intt \check r_k^2 dk dx \Big]
\leq
CN^3.
\end{equs}
Hence, by Proposition \ref{prop: bd},
\begin{equs}
-\log\sZ_{\beta,N,K}
=
\inf_{v \in \HH_K} \EE \Psi^+_K(v)
\leq 
\EE \Psi^+_K(\check v)
\leq
CN^3
\end{equs}
thereby establishing \eqref{eq:zboundupper}.

\subsection{Proof of Proposition \ref{prop: q bound main}} \label{subsec: proof of q bound}

We begin with two propositions, the first of which is a type of It\^o isometry for fields under $\nubn$ and the second of characterises functions against which the Wick square field can be tested against. Together, they imply that the random variables in Proposition \ref{prop: q bound main} are integrable and that these expectations can be approximated using the cutoff measures $\nu_{\beta,N,K}$. Recall also Remarks \ref{rem:block_av} and \ref{rem: testing wick square}.

\begin{prop}\label{prop: testing phi4}
Let $f \in H^{-1+\delta}$ for some $\delta > 0$. For every $K \in (0,\infty)$, let $\phi^{(K)} \sim \nu_{\beta,N,K}$ and $\phi \sim \nubn$.
	
The random variables $\{ \intx \phi^{(K)} f dx \}_{K > 0}$ converge weakly as $K \rightarrow \infty$ to a random variable
	\begin{equs}
	\phi(f) 
	= 
	\intx \phi f dx \in L^2(\nubn).
	\end{equs}
	
Moreover, for every $c> 0$,
	\begin{equs}
	\big\langle \exp{\big(c\phi(f)^2\big)} \big\rangle_{\beta,N} 
	< 
	\infty.	
	\end{equs}
 
\end{prop}

\begin{proof}
Let $\{ f_n \}_{n \in \NN} \subset C^\infty(\TTN)$ such that $f_n \rightarrow f$ in $H^{-1+\delta}$. We first show that $\{ \phi(f_n) \}$ is Cauchy in $L^2(\nubn)$. 

Let $\varepsilon > 0$. Choose $n_0$ such that, for all $n,m > n_0$, $\| f_n - f_m \|_{H^{-1+\delta}} < \frac{\varepsilon}{N^3}$.

Fix $n,m > n_0$ and let $\delta f = f_n - f_m$. Then,
\begin{equs} \label{eq:cauchy}
|\phi(f_n) - \phi(f_m)|^2
=
\varepsilon \cdot \frac 1\varepsilon \phi(\delta f)^2
\leq
\varepsilon e^{\frac 1\varepsilon \phi(\delta f)^2 }.
\end{equs}

By Proposition \ref{prop:zbounds}, there exists $C=C(\eta)>0$ such that
\begin{equs}
\Big\langle e^{\frac 1\varepsilon \phi(\delta f)^2 } \Big\rangle_{\beta,N}
&=
\lim_{K \rightarrow \infty} \frac{1}{\sZ_{\beta,N,K}} \EE_N e^{-\cH_{\beta,N,K}(\phi_K) + \frac 1\varepsilon \phi_K(\delta f)^2}
\\
&\leq
e^{CN^3} \limsup_{K \rightarrow \infty}\EE_N e^{-\cH_{\beta,N,K}(\phi_K) + \frac 1\varepsilon\phi_K(\delta f)^2}.	
\end{equs}
We apply Proposition \ref{prop: bd} to the expectation on the righthand side (with total energy cutoff suppressed, see Remark \ref{rem: local martingale issue} and the paragraph that precedes it).

For $v \in \HH_{b,K}$, define
\begin{equs}
\Psi^{\delta f}_K(v)
=
\cH_{\beta,N,K}(\<1>_K + V_K) - \frac 1\varepsilon \Big( \intx (\<1>_K + V_K) \delta f dx \Big)^2 + \frac 12 \intx\intt v_k^2 dk dx.  	
\end{equs}

Expanding out the second term (and ignoring the prefactor $\frac 1\varepsilon$ for the moment), we obtain:
\begin{equs} \label{eq: phi4 itoproof 1}
	\EE \Bigg[ \Big(\intx \<1>_K \delta f dx \Big)^2 + \Big( \intx \<30>_K \delta f dx \Big)^2 + \Big( \intx U_K \delta f dx \Big)^2 \Bigg]. 
\end{equs}

Consider the first integral in \eqref{eq: phi4 itoproof 1}. By Parseval's theorem, the Fourier coefficients of $\<1>_K$ (see \eqref{def: lollipop}), and It\^o's isometry,
\begin{equs} \label{eq: phi4 itoproof 1 a}
\begin{split}
\EE \Big[\intx \<1>_K \delta f dx\Big]^2
&=
\frac{1}{N^6} \sum_{n,m} \EE [\cF \<1>_K (n) \cF \<1>_K (m)] \cF\delta f(m) \cF \delta f(n)
\\
&\lesssim
\frac{1}{N^3}\sum_{n} \frac{|\cF \delta f(n)|^2}{\langle n \rangle^2} 
\lesssim
N^3 \| \delta f \|_{H^{-1+\delta}}^2
\end{split}
\end{equs}
where sums are taken over frequencies $n_i \in (N^{-1}\ZZ)^3$. Above, the $N$ dependency in the last inequality is due to our Sobolev spaces being defined with respect to normalised Lebesgue measure $\dbar x$.  

For the second term in \eqref{eq: phi4 itoproof 1}, by Parseval's theorem, It\^o's isometry, and the Fourier coefficients of $\<30>_K$ (see \eqref{eq: trident fourier bound}), we obtain
\begin{equs} \label{eq: phi4 itoproof 1 b}
\begin{split}
\EE \Big( \intx \<30>_K \delta f dx \Big)^2
&=
\frac{1}{N^6}\EE \Big( \sum_{n} \cF \<30>_K (n) \cF \delta f (n) \Big)^2
\\
&=
\frac{1}{N^6}\sum_n |\cF \delta f(n)|^2 \EE \Big|\cF\<30>_K(n)\Big|^2
\\
&\lesssim
\sum_n \frac{|\cF \delta f(n)|^2}{\langle n \rangle^4}
\lesssim 
N^6 \| \delta f \|_{H^{-1+\delta}}^2.
\end{split}
\end{equs}

For the final term in \eqref{eq: phi4 itoproof 1}, by duality \eqref{tool: duality} 
\begin{equs} \label{eq: phi4 itoproof 1 c}
\Big(\intx U_K \delta f dx \Big)^2	
\leq
N^6 \| \delta f \|_{H^{-1+\delta}}^2 \| U_K \|_{H^{1-\delta}}^2.
\end{equs}

Therefore, using that $\| \delta f \|_{H^{1-\delta}}^2 \leq \frac{\varepsilon^2}{N^6}$, the estimates \eqref{eq: phi4 itoproof 1 a}, \eqref{eq: phi4 itoproof 1 b}, and \eqref{eq: phi4 itoproof 1 c} yield: 
\begin{equs} \label{eq: phi4 ito c}
\begin{split}
\EE  &\Bigg[ \frac{1}{\varepsilon} \Big( \intx (\<1>_K +V_K)\delta f dx \Big)^2 \Bigg] 
\\
&\leq
C(\eta)N^6( N^{-3} + 1) \frac{\|\delta f \|_{H^{-1+\delta}}^2}{\varepsilon} 
\\
&\quad\quad\quad
+ C(\eta)N^6\frac{\|\delta f \|_{H^{-1+\delta}}^2}{\varepsilon}\EE \Big[ \|U_K\|_{H^{1-\delta}}^2 \Big]
\\
&\leq
C(\eta)\varepsilon (N^{-3} + 1 + \EE \| U_K\|_{H^{1-\delta}}^2).	
\end{split}
\end{equs}

Using arguments in Section \ref{subsec: proof z lower}, it is straightforward to show that there exists $C=C(\eta,\beta)>0$ such that, for $\varepsilon$ sufficiently small,
\begin{equs}
\EE \Psi_K^{\delta f}(v) 
\geq 
- CN^3
\end{equs}
for every $v \in \HH_{b,K}$ (note that $\beta$ dependence is not important here).

Inserting this into Proposition \ref{prop: bd} gives
\begin{equs} \label{eq:cauchy2}
\limsup_{K \rightarrow \infty} \langle e^{-\cH_{\beta,N,K}(\phi_K) + \frac{1}{\varepsilon} \phi_K(\delta f)^2} \rangle_{\beta,N,K}
\leq
e^{CN^3}.
\end{equs}
Taking expectations in \eqref{eq:cauchy} and using \eqref{eq:cauchy2} finishes the proof that $\{ \phi(f_n) \}$ is Cauchy in $L^2(\nubn)$. 

Similar arguments can be used to show exponential integrability of the limiting random variable, $\phi(f)$ and that,
\begin{equs}
\sup_{K > 0} |\langle |\phi^{(K)}(f_n) - \phi^{(K)}(f)| \rangle_{\beta,N,K} \rightarrow 0 \quad \text{as }	n \rightarrow \infty.
\end{equs}

We now show that $\phi^{(K)}(f)$ converges weakly to $\phi(f)$ as $K \rightarrow \infty$. Let $G : \RR \rightarrow \RR$ be bounded and Lipschitz with Lipschitz constant $|G|_{\rm Lip}$, and let $\varepsilon > 0$. Choose $n$ sufficiently large so that
\begin{equs}
	\sup_{K > 0} |\langle |\phi^{(K)}(f_n) - \phi^{(K)}(f)| \rangle_{\beta,N,K} 
	< 
	\frac{\varepsilon}{2|G|_{\rm Lip}}
\end{equs}
and
\begin{equs}
\langle |\phi(f_n)) - \phi(f)| \rangle_{\beta,N}	
<
\frac{\varepsilon}{2|G|_{\rm Lip}}.
\end{equs}

Then,
\begin{equs}
|\langle G(\phi^{(K)}(f)) \rangle_{\beta,N,K} - \langle G(\phi(f)) \rangle_{\beta,N}
&\leq
\sup_{K > 0} |\langle G(\phi^{(K)}(f_n)) - G(\phi^{(K)}(f)) \rangle_{\beta,N,K} |
\\
&\quad
+ 	| \langle G(\phi^{(K)}(f_n)) \rangle_{\beta,N,K} - \langle G(\phi(f_n)) \rangle_{\beta,N}|
\\
&\quad
+ |\langle G(\phi(f_n)) - G(\phi(f)) \rangle_{\beta,N}|
\\
&\leq 
| \langle G(\phi^{(K)}(f_n)) \rangle_{\beta,N,K} - \langle G(\phi(f_n)) \rangle_{\beta,N}| + \varepsilon.
\end{equs}
The first term on the righthand side goes to zero as $K \rightarrow \infty$ since $f_n \in C^\infty$. Thus,
\begin{equs}
	\lim_{K \rightarrow \infty} |\langle G(\phi(f)) \rangle_{\beta,N,K} - \langle G(\phi(f)) \rangle_{\beta,N}
	\leq
	\varepsilon.
\end{equs}

Since $\varepsilon$ is arbitrary, we have shown that $\phi^{(k)}(f)$ converges weakly to $\phi(f)$. 
\end{proof}

\begin{prop} \label{prop: testing wick square}
Let $f \in B^s_{\frac 43,1} \cap L^{2}$ for some $s > \frac 12$. For every $K \in (0,\infty)$, let $\phi^{(K)} \sim \nu_{\beta,N,K}$ and $\phi \sim \nubn$.

The random variables $\{ \intx :(\phi^{(K)})^2: f dx \}_{K>0}$ converge weakly as $K \rightarrow \infty$ to a random variable
\begin{equs}
:\phi^2:(f)	
= 
\intx :\phi^2: f dx \in L^2(\nubn).
\end{equs}

Moreover, for $c > 0$,
\begin{equs}
\big\langle \exp \big( c:\phi^2:(f) \big) \big\rangle_{\beta,N}
<
\infty.	
\end{equs}
\end{prop}

\begin{proof}
The proof of Proposition \ref{prop: testing wick square} follows the same strategy as the proof of Proposition \ref{prop: testing phi4}, so we do not give all the details. The only real key difference is the analytic bounds required in the stochastic control problem. Indeed, these require one to tune the integrability assumptions on $f$ in order to get the required estimates. 

It is not too difficult to see that the term we need to control is the integral
\begin{equs} \label{testing wick square: eq 1}
\intx \<2>_K f + 2 \<1>_K V_K f + V_K^2 f dx.	
\end{equs}
Strictly speaking, we need to control the above integral with $f$ replaced by $\delta f = f_n - f_m$, where $\{ f_n \}_{n \in \NN}\subset C^\infty(\TTN)$ such that $f_n \rightarrow f$ in $B^s_{\frac 34, 1} \cap L^2$, but the analytic bounds are the same.

Note that $\EE \intx \<2>_K f dx = \intx \EE \<2>_K f dx = 0$. Moreover by Young's inequality and the additional integrability assumption $f \in L^2$, for any $\varepsilon > 0$ we have
\begin{equs}
\intx V_K^2 f dx
\lesssim 
\frac{1}{\varepsilon} \intx f^2 dx + \varepsilon \intx V_K^4 dx
\end{equs}
which can be estimated as in the proof of Proposition \ref{prop: testing phi4}. Thus, we only need to estimate the second integral in \eqref{testing wick square: eq 1}. Note that the product $\<1>_K f$ is a well-defined distribution from a regularity perspective as $K \rightarrow \infty$ since $f \in B^{s}_{\frac 43,1}$ for $s > \frac 12$. The difficulty in obtaining the required estimates comes from integrability issues.

We split the integral into three terms by using the paraproduct decomposition $\<1> f = \<1> \pg f + \<1> \pe f + \<1> \pl f$. The integral associated to $\<1> \pl f$ is straightforward to estimate, so we focus on the first two terms. Since $f \in L^2$ and $\<1>_K \in \cC^{- \frac 12 - \kappa}$, by the paraproduct estimate \ref{tool: paraproduct} we have $\<1> \pg f \in H^{-\frac 12 -\kappa}$. Thus, the integral $\intx (\<1>_K \pg f)V_K dx$ can be treated similarly as in the proof of Proposition \ref{prop: testing phi4}. Note that, in this proposition the use of H\"older-Besov norms is fine because we are not concerned with issues of $N$ dependency. Moreover, note that if we just used that $f \in B^s_{\frac 43,1}$ the resulting integrability of $\<1>_K \pg f$ is not sufficient to justify testing against $V_K$, which can be bounded in $L^2$-based Sobolev spaces. For the final integral, by the resonant product estimate \eqref{tool: resonant} we have $\<1>_K \pe f \in L^{\frac 43}$. Hence, we can use Young's inequality to estimate $\intx (\<1>_K \pe f)V_K dx$ and then argue similarly as in the proof of Proposition \ref{prop: testing phi4}.
\end{proof}

Without loss of generality, we assume $a_0 = a = 1$ in Proposition \ref{prop: q bound main} and we split its proof into Lemmas \ref{lem: Q1}, \ref{lem: Q2}, and \ref{lem: Q3}.

\begin{lem} \label{lem: Q1}
There exists $\beta_0 > 1$ and $C_Q > 0$ such that, for any $\beta > \beta_0$, 
\begin{equs}
-\frac{1}{N^3} \log \Big\langle \prod_{\sBox \in \BBN} \exp Q_1(\Box) \Big\rangle_{\beta,N}
\geq
-C_Q.	
\end{equs}
\end{lem}

\begin{proof}

For any $K \in (0,\infty)$, define
\begin{equs}
\cH_{\beta,N,K}^{Q_1}(\phi_K)
=
\intx :\cV_\beta^{Q_1}(\phi_K): - \frac{\gamma_K}{\beta^2}:\phi_K^2: - \delta_K - \frac \eta 2 :\phi_K^2: dx	
\end{equs}
where
\begin{equs}
\cV^{Q_1}_\beta(a)
=
\cV_\beta(a) - \frac 1\hfbeta (\beta - a^2) - \frac 14
=
\frac 1\beta\Bigg(a^2 - \Big(\beta + \frac\hfbeta 2 \Big) \Bigg)^2.
\end{equs}

Then, by Propositions \ref{prop: testing wick square} and \ref{prop:zbounds}, there exists $C=C(\eta)>0$ such that
\begin{equs}
\Big\langle \prod_{\sBox \in \BBN} \exp Q_1(\Box) \Big\rangle_{\beta,N}
&=
\lim_{K \rightarrow \infty} \Big\langle \exp \Big( \frac 1\hfbeta \intx \beta - :\phi_k^2: dx \Big) \Big\rangle_{\beta,N,K}
\\
&\leq
e^{\frac 14 N^3} \lim_{K \rightarrow \infty} \frac{1}{\sZ_{\beta,N,K}} \EE_N e^{-\cH^{Q_1}_{\beta,N,K}(\phi_K)}
\\
&\leq
e^{\big(C+\frac 14\big)N^3} \limsup_{K \rightarrow \infty} \EE_N e^{-\cH_{\beta,N,K}^{Q_1}(\phi_K)}
\end{equs}
where

Therefore, we have reduced the problem to proving Proposition \ref{prop:zbounds} for the potential $\cV_\beta^{Q_1}$ instead of $\cV_\beta$. The proof follows essentially word for word after two observations: first, the same $\gamma_K$ and $\delta_K$ works for both $\cV_\beta$ and $\cV_\beta^{Q_1}$ since the quartic term is unchanged. Second, since $\sqrt{\beta + \frac \hfbeta 2} = \hfbeta + o(\hfbeta)$ as $\beta \rightarrow \infty$, the treatment of $\beta$-dependence of the estimates in Section \ref{subsec: estimates} is exactly the same.
\end{proof}

\begin{lem} \label{lem: Q2}
There exists $\beta_0 > 1$ and $C_Q > 0$ such that, for any $\beta > \beta_0$,
\begin{equs}\label{q2 bound}
-\frac{1}{N^3} \log \Big\langle \prod_{\sBox \in \BBN} \exp Q_2(\Box) \Big\rangle_{\beta,N}
\geq
-C_Q.
\end{equs}
\end{lem}

\begin{proof}
By Propositions \ref{prop: testing phi4}, \ref{prop: testing wick square}, and \ref{prop:zbounds}, there exists $C=C(\eta)>0$ such that, for $\beta$ sufficiently large,
\begin{equs}
\Big\langle \prod_{\Box \in \BBN} \exp Q_2(\Box) \Big\rangle_{\beta,N}
&=
\lim_{K \rightarrow \infty} \Big\langle \exp \Big( \frac 1\hfbeta \sum_{\Box \in \BBN} \phi_K(\Box)^2 - \frac 1\hfbeta \intx :\phi_K^2: dx \Big\rangle_{\beta,N,K}
\\
&\leq
e^{CN^3} \limsup_{K \rightarrow \infty} \EE_N e^{-\cH_{\beta,N,K}^{Q_2}(\phi_K)}
\end{equs}
where
\begin{equs}
\cH_{\beta,N,K}^{Q_2}(\phi_K)
=
\cH_{\beta,N,K}(\phi_K) + \frac 1\hfbeta \sum_{\sBox \in \BBN} \phi_K(\Box)^2 - \frac 1\hfbeta \intx :\phi_K^2: dx.
\end{equs}

As in Section \ref{sec: translation}, we perform the expansion
\begin{equs} \label{q2 low temp expansion}
-\log \EE_N e^{-\cH_{\beta,N,K}^{Q_2}(\phi_K)}
=
\sum_{\sigma \in \{ \pm 1 \}^\BBN} e^{-F_{\beta,N,K}^\sigma}\EE_N e^{-\cH_{\beta,N,K}^{Q_2,\sigma}(\phi_K)} 	
\end{equs}
where $F_{\beta,N,K}^\sigma$ is defined in \eqref{def: f} and
\begin{equs}
\cH_{\beta,N,K}^{Q_2,\sigma}(\phi_K)
=
\cH_{\beta,N,K}^{Q_2}(\phi_K + g_K) - \sum_{\sBox \in \BBN} \log \Big( \chi_{\sigma(\sBox)} \big( (\phi_K+g_K)(\Box) \big) \Big)
\end{equs}

Fix $\sigma \in \{\pm 1\}^\BBN$. For $v \in \HH_{b,K}$, define
\begin{equs}\label{def: psi q_2}
\Psi^{Q_2}_K(v)
&=
\Psi_K(v) + \frac 1\hfbeta \sum_{\Box \in \BBN} \Big( \int_\Box \<1>_K + V_K + g_K dx \Big)^2
\\
&\quad\quad\quad
- \frac 1\hfbeta \intx :(\<1>_K + V_K + g_K)^2: dx
\end{equs}
where $\Psi_K=\Psi_K^\sigma$ is defined in \eqref{def: psi}.

We estimate second term in \eqref{def: psi q_2}. First, note that
\begin{equs}
	\frac 1\hfbeta \sum_{\sBox \in \BBN} &\Big( \int_\sBox \<1>_K + V_K + g_K dx \Big)^2
	\\
	&\leq
	\sum_{\sBox \in \BBN} \frac 2\hfbeta \Big( \int_\sBox \<1>_K dx \Big)^2 + \frac 2\hfbeta \Big( \int_\sBox V_K + g_K dx \Big)^2.
\end{equs}
By a standard Gaussian covariance calculation, there exists $C=C(\eta)>0$ such that
\begin{equs}
\sum_{\sBox \in \BBN} \EE \Big( \int_\sBox \<1>_K dx \Big)^2
=
\sum_{\sBox \in \BBN} \int_\sBox \int_\sBox \EE[ \<1>_K(x) \<1>_K(x')] dxdx'
\leq 
CN^3.
\end{equs}
For the other term, by the Cauchy-Schwarz inequality followed by bounds on the potential \eqref{tool: potential bound} and $g_K$ \eqref{tool: g bounds}, the following estimate holds: for any $\zeta > 0$, 
\begin{equs}
\frac 2\hfbeta \sum_{\sBox \in \BBN} \Big( \int_\sBox V_K + g_K dx \Big)^2
&\leq
\intx \frac 2\hfbeta (V_K+g_K)^2 dx
\\
&\leq C(\zeta,C_P) N^3 + \zeta C_P \intx \cV_\beta(V_K+g_K) dx
\end{equs}
where $C_P > 0$ is the Poincar\'e constant on unit boxes \eqref{tool: poincare}.

We now estimate the third term in \eqref{def: psi q_2}. Since $\EE \<2>_K = \EE[ \<1>_K g_K] = 0$,
\begin{equs}\label{q2 bound eq1}
\frac 1\hfbeta \intx :(\<1>_K + V_K +g_K)^2: dx 
\approx 
\frac 1\hfbeta \intx 2\<1>_K V_K + (V_K + g_K)^2 dx.	
\end{equs}

For the first integral on the righthand side of \eqref{q2 bound eq1}, by change of variables \eqref{eq: intermediate integrated ansatz}, and the paraproduct decomposition \eqref{tool: bony decomp}, we have
\begin{equs}
\frac 1\hfbeta \intx 2\<1>_K V_K dx
&=
 \intx -\frac{8}{\beta^\frac 52}(\<1>_K \pl \<30>_K + \<31>_K + \<1>_K \pg \<30>_K) + \frac 2\hfbeta \<1>_K U_K dx.
\end{equs}
Hence, by \eqref{q2 bound eq1}, Proposition \ref{prop: diagrams}, duality \eqref{tool: duality}, the potential bounds \eqref{tool: potential bound}, and the bounds on $U_K$ \eqref{eq: u estimate},  for any $\varepsilon > 0$ there exists $C=C(\varepsilon,\eta)>0$ such that
\begin{equs}
\Big| \frac 1\hfbeta \intx 2\<1>_K V_K dx \Big|
\leq 
CN^\Xi_K + \varepsilon \Big( \intx \cV_\beta(V_K+g_K) dx + \frac 12 \intx \intt r_k^2 dk dx \Big).
\end{equs}

For the second integral on the righthand side of \eqref{q2 bound eq1}, again by \eqref{tool: potential bound} and \eqref{tool: g bounds}, there exists an inessential constant $C>0$ such that
\begin{equs}
	\intx \frac 1\hfbeta (V_K+g_K)^2 dx
	\leq
	CN^3 + \zeta C_P \intx \cV_\beta(V_K+g_K) dx.
\end{equs}

Arguing as in Section \ref{subsec: proof z lower} and taking into account the calculations above, the following estimate holds: let $\zeta < \frac 1{8C_P}$ and $\varepsilon = 1-8C_P\zeta > 0$ as in Section \ref{subsec: proof z lower}. Then, provided $\eta < \frac{1}{196C_P}$ and $\beta > 1$,
\begin{equs}
\EE \Psi_K^{Q_2}(v)
&\geq
\EE \Bigg[ -C(\varepsilon,\zeta,\eta) N^\Xi_K + \Big(\frac{1-\varepsilon}{2} - \frac{4C_P\zeta}{2\beta^3} - 2C_P \zeta \Big) \intx \cV_\beta(V_K+g_K) dx
\\
&\quad\quad\quad
+ \Big(\frac{1-\varepsilon}2 - 4\zeta C_P \Big) \intx \intt r_k^2 dk dx - 4\zeta C_P \intx |\nabla \tilde g_K |^2 dx \Bigg]
\\
&\geq
-CN^3 -4\zeta C_P \intx |\nabla \tilde g_K|^2 dx.	
\end{equs}
Hence, by Proposition \ref{prop: bd} applied with the Hamiltonian $\cH_{\beta,N,K}^{Q_2,\sigma}(\phi_K)$ with total energy cutoff suppressed (see Remark \ref{rem: local martingale issue}),
\begin{equs}
F_{\beta,N,K}^\sigma - \log \EE_N e^{-\cH_{\beta,N,K}^{Q_2,\sigma}}
\geq
-CN^3 + \Big( \frac 12 - 4\zeta C_P\Big) \intx |\nabla \tilde g_K|^2 dx \geq - CN^3	
\end{equs}
This estimate is uniform in $\sigma$, thus summing over the $2^{N^3}$ terms in the expansion \eqref{q2 low temp expansion} yields \eqref{q2 bound}.
\end{proof}

\begin{lem} \label{lem: Q3}
There exists $\beta_0 > 1$ and $C_Q > 0$ such that, for any $\beta > \beta_0$,
\begin{equs}\label{q3 bound}
-\frac{1}{N^3} \log \Big\langle \prod_{\{\sBox,\sBox'\} \in B} \exp |Q_3(\Box,\Box')| \Big\rangle_{\beta,N}
\geq
-C_Q
\end{equs}
where $B$ is a set of unordered pairs of nearest-neighbour blocks that partitions $\BBN$. 
\end{lem}

\begin{proof}
By Propositions \ref{prop: testing phi4} and \ref{prop:zbounds} there exists $C=C(\eta)>0$ such that, for $\beta$ sufficiently large,
\begin{equs}
\Big\langle \prod_{\{\sBox,\sBox'\} \in B} \exp |Q_3(\Box,\Box')| \Big\rangle_{\beta,N}
&=
\lim_{K \rightarrow \infty} \Big\langle \exp \Big( \sum_{\{\sBox,\sBox'\} \in B} \Big| \int_{\sBox} \phi_K dx - \int_{\sBox'} \phi_K dx \Big| \Big\rangle_{\beta,N,K}
\\
&\leq
e^{CN^3} \limsup_{K \rightarrow \infty} \EE_N e^{-\cH_{\beta,N,K}^{Q_3}(\phi_K)}
\end{equs}
where
\begin{equs}
\cH_{\beta,N,K}^{Q_3}(\phi_K) 
= 
\cH_{\beta,N,K}^{Q_3}(\phi_K)	- \sum_{\{\sBox,\sBox'\} \in B} \Big| \int_{\sBox} \phi_K dx - \int_{\sBox'} \phi_K dx \Big|.
\end{equs}

We expand
\begin{equs}
-\log \EE_N e^{-\cH_{\beta,N,K}^{Q_3}(\phi_K)}
=
\sum_{\sigma \in \{ \pm 1 \}^\BBN} e^{-F_{\beta,N,K}^\sigma}\EE_N e^{-\cH_{\beta,N,K}^{Q_3,\sigma}}	
\end{equs}
where $F_{\beta,N,K}^\sigma$ is defined in \eqref{def: f} and
\begin{equs}
\cH_{\beta,N,K}^{Q_3,\sigma}(\phi_K)
=
\cH_{\beta,N,K}^{Q_3}(\phi_K + g_K) - \sum_{\sBox \in \BBN} \log \Big( \chi_{\sigma(\sBox)}\big( (\phi_K + g_K)(\Box) \big) \Big).
\end{equs}

Fix $\sigma \in \{\pm 1\}^\BBN$. For $v \in \HH_{b,K}$, define
\begin{equs}
\Psi_K^{Q_3}(v)
=
\Psi_K(v) - \sum_{\{ \sBox, \sBox' \} \in B} \Big| \int_\sBox \<1>_K + V_K + g_K dx - \int_{\sBox'} \<1>_K + V_K +g_K dx \Big|
\end{equs}
where $\Psi_K(v) = \Psi_K^\sigma(v)$ is defined in \eqref{def: psi}.

A standard Gaussian calculation yields $\EE |\<1>_K| \leq CN^3$ for some constant $C=C(\eta)>0$. Hence, by the triangle inequality, Proposition \ref{prop: diagrams} and the Cauchy-Schwarz inequality,  
\begin{equs}
	\sum_{\{ \sBox, \sBox' \} \in B} &\Big| \int_\sBox \<1>_K + V_K + g_K dx - \int_{\sBox'} \<1>_K + V_K +g_K dx \Big|
	\\
	&\lesssim
	CN^\Xi_K + \frac{1}{\beta^2} \Big| \intx \intt \cJ_k\big(\<2>_k \pg (V_k^\flat +g_k^\flat) \big) dk dx	\\
	&\quad\quad\quad
	+ \sum_{\{ \sBox, \sBox' \} \in B} \Big| \int_\sBox (R_K + g_K) dx - \int_{\sBox'} (R_K + g_K) dx \Big|.
\end{equs}
The integral with the paraproduct can be estimated as in \eqref{eq: u estimate eq2} to establish: for any $\zeta > 0$,
\begin{equs}
	\frac{1}{\beta^2} \Big| \intx \intt \cJ_k\big(\<2>_k \pg (V_k^\flat +g_k^\flat) \big) dk dx
	\leq
	C(\zeta,C_P)N^3 + \frac{2\zeta C_P}{\beta^3} \intx \cV_\beta(V_K+g_K) dx
\end{equs}
where $C_P>0$ is the Poincar\'e constant on unit blocks \eqref{tool: poincare}.

We now estimate the remaining integral. Assume without loss of generality that $\Box' = \Box + e_1$. Then, by the triangle inequality and the fundamental theorem of calculus,
\begin{equs}
\Big| \int_\sBox (R_K + g_K) dx &- \int_{\sBox'} (R_K + g_K) dx \Big|
\\
&=
\int_\sBox \Big( R_K(x) - R_K(x+e_1) + g_K(x) - g_k(x+e_1) \Big) dx
\\
&\leq
\int_0^1 \int_\sBox |\nabla R_K(x+te_1)| + |\nabla g_K(x+te_1)| dx dt
\\
&\leq
\int_{\sBox \cup \sBox'} |\nabla R_K| + |\nabla g_K| dx.
\end{equs}

Hence, by the Cauchy-Schwarz inequality, the bound on the drift \eqref{eq: l2 drift bound} and the bound on $\nabla g_K$ \eqref{tool: grad g g tilde}, we have the following estimate: for any $\zeta > 0$,
\begin{equs}
	\sum_{\{ \sBox, \sBox' \} \in B} \Big| \int_\sBox (R_K + g_K) dx &- \int_{\sBox'} (R_K + g_K) dx \Big|
	\\
	&\leq
	C(\zeta,C_P)N^3 + 4\zeta C_P \Big( \intx |\nabla R_K|^2 dx + \intx |\nabla g_K|^2 dx \Big)
	\\
	&\leq
	C(\zeta,C_P)N^3 + 4\zeta C_P \Big( \intx \intt r_k^2 dk dx + \intx |\nabla \tilde g_K|^2 dx \Big).
\end{equs}

Thus, by arguing as in Section \ref{subsec: proof z lower}, one can show the following estimate: let $\zeta < \frac 1{16C_P}$ and $\varepsilon = 1 - 8\zeta C_P > 0$. Then, provided $\eta < \frac{1}{392 C_P}$ and $\beta > 1$,
\begin{equs}
\EE \Psi_K^{Q_3}(v)
&\geq
\EE \Bigg[ -CN^\Xi_K + \Big( \frac{1-\varepsilon}{2} - \frac{2\zeta C_P}{\beta^3} - \frac{2\zeta C_P}{\beta^3} \Big) \intx \cV_\beta(V_K+g_K) dx
\\
&\quad\quad\quad
+ \Big( \frac{1-\varepsilon}2 - 4\zeta C_P - 4\zeta C_P \Big) \intx \intt r_k^2 dk dx
\\
&\quad\quad\quad- (4\zeta C_P + 4\zeta C_P) \intx |\nabla \tilde g_K|^2 dx \Bigg]
\\
&\geq
-CN^3 - 8\zeta C_P \intx |\nabla \tilde g_K|^2 dx. 
\end{equs}

Applying Proposition \ref{prop: bd} with Hamiltonian $\cH_{\beta,N,K}^{Q_3,\sigma}(\phi_K)$, with total energy cutoff suppressed (see Remark \ref{rem: local martingale issue}), yields
\begin{equs}
F_{\beta,N,K}^\sigma - \log \EE_N e^{-\cH_{\beta,N,K}^{Q_3}(\phi_K)} 
\geq 
- CN^3 + \Big( \frac 12 - 8\zeta C_P \Big) \intx |\nabla \tilde g_K|^2 dx
\geq
-CN^3.	
\end{equs}
This estimate is uniform over all $2^{N^3}$ choices of $\sigma$, hence establishing \eqref{q3 bound}.
\end{proof}

\section{Chessboard estimates} \label{sec: chessboard estimates}

In this section we prove Proposition \ref{prop: cosh} using the chessboard estimates of Proposition \ref{prop:chessboard_estimates} and the estimates obtained in Section \ref{sec: free energy}. In addition, we establish that $\nubn$ is reflection positive.

\subsection{Reflection positivity of $\nubn$} \label{subsec: rp}
We begin by defining reflection positivity for general measures on spaces of distributions following \cite{S86} and \cite{GJ87}.

For any $a \in \{0, \dots, N-1\}$ and $\{i,j,k\} = \{ 1,2,3 \}$, let 
\begin{equs}
\cR_{\Pi_{a,i}}(x) = (2a-x_i)e_i + e_j + e_k
\end{equs}
where $x = x_i e_i + x_j e_j + x_k e_j \in \TTN$ and addition is understood modulo $N$. Define
\begin{equs} \label{def:hyperplane}
	 \Pi_{a,i}
	 =
	 \{ x \in \TTN : \cR_{\Pi_{a,i}}(x)=x\}.
\end{equs}
Note that for any $x \in \Pi_{a,i}$, $x_i = a \text{ or } a+\frac N2$. We say that $\cR_{\Pi_{a,i}}$ is the reflection map across the hyperplane $\Pi_{a,i}$. 

Fix such a hyperplane $\Pi$. It separates $\TTN = \TTN^+ \sqcup \Pi \sqcup \TTN^-$ such that $\TTN^+ = \cR_\Pi \TTN^-$. For any $f \in C^\infty(\TTN)$, we say $f$ is $\TTN^+$-measurable if $\text{supp} f \subset \TTN^+$. The reflection of $f$ in $\Pi$ is defined pointwise by $\cR_\Pi f(x) = f(\cR_\Pi x)$. For any $\phi \in S'(\TTN)$, we say that $\phi$ is $\TTN^+$-measurable if $\phi(f) = 0$ unless $f$ is $\TTN^+$ measurable, where $\phi(f)$ denotes the duality pairing between $S'(\TTN)$ and $C^\infty(\TTN)$. For any such $\phi$, we define $\cR_\Pi \phi$ pointwise by $\cR_\Pi \phi(f) = \phi(\cR_\Pi f)$. 

Let $\nu$ be a probability measure on $S'(\TTN)$. We say that $F \in L^2(\nu)$ is $\TTN^+$-measurable if it is measurable with respect to the $\sigma$-algebra generated by the set of $\phi \in S'(\TTN)$ that are $\TTN^+$-measurable. For any such $F$, we define $\cR_\Pi F$ pointwise by $\cR_\Pi F(\phi) = F(\cR_\Pi \phi)$.

The measure $\nu$ on $S'(\TTN)$ is called \textit{reflection positive} if, for any hyperplane $\Pi$ of the form \eqref{def:hyperplane},
\begin{equs}
    \int_{S'(\TTN)} F(\phi) \cdot \cR_\Pi F(\phi) d\nu(\phi)
    \geq 
    0
\end{equs}
for all $F \in L^2(\nu)$ that are $\TTN^+$-measurable.

\begin{prop}\label{prop:rp}
The measure $\nubn$ is reflection positive.	
\end{prop}

\subsubsection{Proof of Proposition \ref{prop:rp}}
In general, Fourier approximations to $\nubn$ (such as $\nu_{\beta,N,K}$) are not reflection positive. Instead, we prove Proposition \ref{prop:rp} by considering lattice approximations to $\nubn$ for which reflection positivity is straightforward to show.

Let $\TT_N^{\varepsilon} = (\varepsilon \ZZ / N \ZZ)^3$ be the discrete torus of sidelength $N$ and lattice spacing $\varepsilon > 0$. In order to use discrete Fourier analysis, we assume that $\varepsilon^{-1} \in \NN$. Note that any hyperplane $\Pi$ of the form \eqref{def:hyperplane} is a subset of $\TT_N^\varepsilon$. 

For any $\varphi \in (\RR)^{\TT_N^\varepsilon}$, define the lattice Laplacian
\begin{equs}
\Delta^\varepsilon \varphi(x)
=
\frac 1{\varepsilon^2} \sum_{\substack{y \in \TT_N^\varepsilon \\ |x-y| = \varepsilon}} (\varphi(y) - \varphi(x)).
\end{equs}
Let $\tilde\mu_{N,\varepsilon}$ be the Gaussian measure on $\RR^{\TT_N^\varepsilon}$ with density
\begin{equs}
d\tilde\mu_{N,\varepsilon}(\varphi)
\propto
\exp\Big(-\frac {\varepsilon^3}{2} \sum_{x \in \TT_N^\varepsilon} \varphi(x) \cdot (-\Delta^\varepsilon + \eta)\varphi(x) \Big) \prod_{x \in \TT_N^\varepsilon} d\varphi(x)
\end{equs}
where $d\vphi(x)$ is Lebesgue measure. 

A natural lattice approximation to $\nubn$ is given by the probability measure $\tilde\nu_{\beta,N,\varepsilon}$ with density proportional to
\begin{equs}
d\tilde\nu_{\beta,N,\varepsilon}(\varphi)
\propto
e^{-\tilde\cH_{\beta,N,\varepsilon}(\varphi)} d\tilde{\mu}_{N,\varepsilon}(\varphi)
\end{equs}
where
\begin{equs}
\tilde\cH_{\beta,N,\varepsilon}(\varphi)
=
\varepsilon^3 \sum_{x \in \TT_n^\varepsilon} \cV_\beta(\varphi(x)) - \Big(\frac \eta 2 + \frac 12 \delta m^2(\varepsilon, \eta) \Big) \varphi(x)^2 	
\end{equs}
where $\frac 12 \delta m^2(\varepsilon,\eta)$ is a renormalisation constant that diverges as $\varepsilon \rightarrow 0$ (see Proposition \ref{prop: lattice convergence}). Note two things: first, the renormalisation constant is chosen dependent on $\eta$ for technical convenience. Second, no energy renormalisation is included since we are only interested in convergence of measures.

\begin{rem}
By embedding $\RR^{\TTN^\varepsilon}$ into $S'(\TTN)$, we can define reflection positivity for lattice measures. We choose this embedding so that the pushforward of $\tilde\nu_{\beta,N,\varepsilon}$ is automatically reflection positive, but other choices are possible. 
\end{rem}

For any $\varphi \in \RR^{\TT_N^\varepsilon}$, we write $\rext^\varepsilon \varphi$ for its unique extension to a trigonometric polynomial on $\TTN$ of degree less than $\varepsilon^{-1}$ that coincides with $\varphi$ on lattice points (i.e. in $\TT_N^\varepsilon$). Precisely,
\begin{equs}
\rext^\varepsilon(\varphi)(x)
=
\frac{\varepsilon^3}{N^3} \sum_{n} \sum_{ y \in \TT_N^\varepsilon} e_n(y-x) \varphi(y)	
\end{equs}
where the sum ranges over all $n=(a_1,a_2,a_3) \in (N^{-1}\ZZ)^3$ such that $|a_i| \leq \varepsilon^{-1}$, and we recall $e_n(x) = e^{2\pi i n \cdot x}$.

\begin{lem}\label{lem: reflection positivity of lattice measures}
Let $\varepsilon > 0$ such that $\varepsilon^{-1} \in \NN$. Denote by $\rext^\varepsilon_* \tilde\nu_{\beta,N,\varepsilon}$ the pushforward of $\tilde\nu_{\beta,N,\varepsilon}$ by the map $\rext^\varepsilon$. Then, the measure $\rext^\varepsilon_* \tilde\nu_{\beta,N,\varepsilon}$ is reflection positive. 	
\end{lem}

\begin{proof}
Fix a hyperplane $\Pi$ of the form \eqref{def:hyperplane} and recall that $\Pi$ separates $\TTN=\TTN^+ \sqcup \Pi \sqcup \TTN^-$. Write $\TT_{N,\varepsilon}^+ = \TTN^+ \cap \TT_N^\varepsilon$. 

Since the measure $\tilde\nu_{\beta,N,\varepsilon}$ is reflection positive on the lattice by \cite[Theorem 2.1]{S86}, the following estimate holds: let $F^\varepsilon \in L^2(\tilde\nu_{\beta,N,\varepsilon})$ be $\TT_{N,\varepsilon}^+$-measurable - i.e. $F^\varepsilon(\varphi)$ depends only on $\varphi(x)$ for $x \in \TT_{N,\varepsilon}^+$. Then,
\begin{equs}\label{eq: rp of lattice}
\int F^\varepsilon(\varphi) \cdot \cR_\Pi F^\varepsilon(\varphi) d\tilde\nu_{\beta,N,\varepsilon}(\varphi)
\geq
0.	
\end{equs}

Let $F \in L^2(\rext^\varepsilon_*\tilde\nu_{\beta,N,\varepsilon})$ be $\TTN^+$-measurable. Then, $F \circ \rext^\varepsilon \in L^2(\tilde\nu_{\beta,N,\varepsilon})$ is $\TT_{N,\varepsilon}^+$-measurable. Using that $\rext^\varepsilon$ and $\cR_\Pi$ (the reflection across $\Pi$) commute,
\begin{equs}
\int F(\phi) \cdot \cR_\Pi F(\phi) d \rext^\varepsilon_* \tilde\nu_{\beta,N,\varepsilon}(\phi)
&=
\int (F \circ \rext^\varepsilon)(\varphi) \cdot  (F \circ \cR_\Pi \circ \rext^\varepsilon) (\varphi) d\tilde\nu_{\beta,N,\varepsilon}(\varphi)
\\
&=
\int (F\circ \rext^\varepsilon) (\varphi) \cdot (F \circ \rext^\varepsilon) (\cR_\Pi \varphi) d\tilde\nu_{\beta,N,\varepsilon}(\varphi)
\\
&\geq
0
\end{equs}
where the last inequality is by \eqref{eq: rp of lattice}. Hence, $\rext^\varepsilon_*\tilde\nu_{\beta,N,\varepsilon}$ is reflection positive. 
\end{proof}

\begin{prop} \label{prop: equivalence of cutoffs}
There exist constants $\frac 12 \delta m^2(\bullet,\eta)$ such that $\rext^\varepsilon_*\vec\nu_{\beta,N,\varepsilon} \rightarrow \nubn$ weakly as $\varepsilon \rightarrow \infty$.
\end{prop}

\begin{proof}
The existence of a weak limit of $\rext^\varepsilon_*\tilde\nu_{\beta,N,\varepsilon}$ as $\varepsilon \rightarrow 0$ was first established in \cite{P75}. The fact the lattice approximations and the Fourier approximations (i.e. $\nu_{\beta,N,K}$) yield the same limit as the cutoff is removed is not straightforward in 3D because of the mutual singularity of $\nubn$ and $\mu_N$ \cite{BG20}. Previous approaches have relied on Borel summation techniques to show that the correlation functions agree with (resummed) perturbation theory \cite{MS77}.

In Section \ref{subsec: spde} we give an alternative proof using stochastic quantisation techniques. The key idea is to view $\nubn$ as the unique invariant measure for a singular stochastic PDE with a local solution theory that is robust under different approximations. This allows us to show directly that $\rext^\varepsilon_*\tilde\nu_{\beta,N,\varepsilon}$ converges weakly to $\nubn$ and avoids the use of Borel summation and perturbation theory. The strategy is explained in further detail at the beginning of that section. 
\end{proof}

\begin{proof}[Proof of Proposition \ref{prop:rp} assuming Proposition \ref{prop: equivalence of cutoffs}]

Proposition \ref{prop:rp} is a direct consequence of Lemma \ref{lem: reflection positivity of lattice measures} and Proposition \ref{prop: equivalence of cutoffs} since reflection positivity is preserved under weak limits.
\end{proof}

\subsection{Chessboard estimates for $\nubn$}\label{subsec: chessboard}
Let $B \subset \BBN$ be either a unit block or a pair of nearest-neighbour blocks. Recall the natural identification of $B$ with the subset of $\TTN$ given by the union of blocks in $B$. $\TTN$ can be written as a disjoint union of translates of $B$. Let $\BB_N^B$ be the set of these translates; its elements are also identified with subsets of $\TTN$. Note that if $B = \Box \in \BBN$, then $\BB_N^B = \BBN$. 

We say that $f \in C^\infty(\TTN)$ is $B$-measurable if $\text{supp} f \subset B$ and $\text{supp} f \cap \partial B = \emptyset$. We say that $\phi \in S'(\TTN)$ is $B$-measurable if $\phi(f) = 0$ for every $f \in C^\infty(\TTN)$ unless $f$ is $B$-measurable. We say that $F \in L^2(\nubn)$ is $B$-measurable if it is measurable with respect to the $\sigma$-algebra generated by $\phi \in S'(\TTN)$ that are $B$-measurable.

\begin{prop}\label{prop:chessboard_estimates}
Let $N \in 4\NN$. Let $\{ F_{\tilde B} : \tilde B \in \BB_N^B \}$ be a given set of $L^2(\nubn)$-functions such that each $F_{\tilde B}$ is $\tilde B$-measurable. 

Fix $\tilde B \in \BB_N^B$ and define an associated set of $L^2(\nubn)$-functions $\{ F_{\tilde B, B'} : B' \in \BB_N^B \}$ by the conditions: $F_{\tilde B, \tilde B} = F_{\tilde B}$; and, for any $B', B'' \in \BB_N^B$ such that $B'$ and $B''$ share a common face, 
\begin{equs}
F_{\tilde B, B'} = \cR_\Pi F_{\tilde B, B''}	
\end{equs}
where $\Pi$ is the unique hyperplane of the form \eqref{def:hyperplane} containing the shared face between $B'$ and $B''$. 

Then,
\begin{equs}
\Big| \Big\langle \prod_{\tilde B \in \BB_N^B} F_{\tilde B} \Big\rangle_{\beta,N} \Big| 
\leq
\prod_{\tilde B \in \BB_N^B} \Big| \Big\langle \prod_{B' \in \BB_N^B} F_{\tilde B, B'} \Big\rangle_{\beta,N} \Big|^ \frac{|B|}{N^3}.
\end{equs}
\end{prop}

\begin{proof}
	This is a consequence of the reflection positivity of $\nubn$. The condition $N \in 4\NN$ guarantees $F_{\tilde B, B'}$ is well-defined. See \cite[Theorem 2.2]{S86}.
\end{proof}

\subsection{Proof of Proposition \ref{prop: cosh}} \label{subsec: proof cosh prop}

In order to be able to apply Proposition \ref{prop:chessboard_estimates} to the random variables $Q_i$ of Proposition \ref{prop: cosh}, we need the following lemma.
\begin{lem}\label{lem:expQdefn}
Let $N \in \NN$ and $\beta > 0$. Then, for any $\Box \in \BBN$, $\exp Q_1(\Box), \exp Q_2(\Box) \in L^2(\nubn)$ is $\Box$-measurable. 

In addition, for any nearest neighbours $\Box, \Box' \in \BBN$, $\exp Q_3(\Box, \Box') \in L^2(\nubn)$ is $\Box \cup \Box'$-measurable.
\end{lem}

\begin{proof}
The fact that $\exp Q_1(\Box), \exp Q_2(\Box), \exp Q_3(\Box,\Box') \in L^2(\nubn)$ follows from estimates obtained in Proposition \ref{prop: testing phi4}. The $\Box$ and $\Box\cup\Box'$ measurability of these observables comes from taking approximations to indicators which are supported on blocks (e.g. using some appropriate regularisation of the distance function) and estimates obtained in Proposition \ref{prop: testing phi4}. 	
\end{proof}

\begin{proof}[Proof of Proposition \ref{prop: cosh}]

Let $B_1, B_2 \subset \BBN$ and $B_3$ be a set of unordered pairs of nearest neighbour blocks in $\BBN$. Then,
\begin{equs}\label{eq:coshreduction1}
\begin{split} 
    \cosh &Q_1(B_1) \cosh Q_2(B_2) \cosh Q_3(B_3)
    \\
    &= 
    2^{-|B_1|-|B_2|-|B_3|} \prod_{\sBox_1 \in B_1}\prod_{\sBox_2 \in B_2}\prod_{\{\sBox_3,\sBox_3'\} \in B_3} \Big(e^{Q_1(\sBox_1)} + e^{-Q_1(\sBox_1)} \Big)
    \\
    &\quad
    \times \Big(e^{Q_2(\sBox_2)} + e^{-Q_2(\sBox_2)} \Big)\Big(e^{Q_3(\sBox_3,\sBox_3')} + e^{Q_3(\sBox_3',\sBox_3)} \Big)
    \\
 &\leq
2^{-|B_1|-|B_2|} \sum_{B_1^+,B_1^-, B_2^+, B_2^-}\prod_{i=1}^2 \Bigg( \prod_{\sBox_i^+ \in B_i^+} e^{Q_i(\sBox_i^+)}\prod_{\sBox_i^- \in B_i^-}e^{-Q_i(\sBox_i^-)} \Bigg) 
\\
&\quad
\times
\prod_{\{\sBox_3,\sBox_3'\} \in B_3} e^{|Q_3(\sBox_3,\sBox_3')|}
\end{split}
\end{equs}
where $\cosh Q_i(B_i)$ is defined in \eqref{eq:cosh} and the sum is over all partitions $B_1^+ \sqcup B_1^- = B_1$ and $B_2^+\sqcup B_2^- = B_2$.

It suffices to prove that there exists $\tilde C_Q>0$ such that, for any $B_1^\pm, B_2^\pm$ and $B_3$ as above,
\begin{equs}\label{eq:coshreduction2}
\begin{split}
\Big\langle \prod_{i=1}^2 \Bigg(\prod_{\sBox_i^+ \in B_i^+} e^{Q_1(\sBox_i^+)}\prod_{\sBox_i^- \in B_i^-} &e^{- Q_2(\sBox_i^-)} \Bigg)\prod_{\{\sBox_3,\sBox_3'\} \in B_3} e^{|Q_3(\sBox_3,\sBox_3')|} \Big\rangle_{\beta,N}
\\
&\leq
e^{\tilde C_Q(|B_1|+|B_2|+|B_3|)}.
\end{split}
\end{equs}

Then, taking expectations in \eqref{eq:coshreduction1} and using \eqref{eq:coshreduction2}
\begin{equs}
\Big\langle \cosh &Q_1(B_1)  \cosh Q_2(B_2) \cosh Q_3(B_3) \Big\rangle_{\beta,N}
\\
&\leq
 2^{|B_1|+|B_2|}\sum_{B_1^+,B_1^-}\sum_{B_2^+,B_2^-}
\\
&\quad
\Big\langle \prod_{i=1}^2 \Bigg( \prod_{\sBox_i^+ \in B_i^+} e^{Q_1(\sBox_i^+)}\prod_{\sBox_i^- \in B_i^-} e^{- Q_2(\sBox_i^-)} \Bigg)\prod_{\{\sBox_3,\sBox_3'\} \in B_3} e^{|Q_3(\sBox_3,\sBox_3')|} \Big\rangle_{\beta,N}
\\
&\leq
e^{\tilde C_Q(|B_1| + |B_2| + |B_3|)}
\end{equs}
which yields Proposition \ref{prop: cosh} with $C_Q = \tilde C_Q$.

To prove \eqref{eq:coshreduction2}, first fix $B_1^\pm$ and $B_2^\pm$. Then, by H\"older's inequality, 
\begin{equs} \label{eq:q_holder}
\begin{split}
\Big\langle \prod_{i=1}^2 &\Bigg( \prod_{\sBox_i^+ \in B_i^+} e^{Q_1(\sBox_i^+)}\prod_{\sBox_i^- \in B_i^-} e^{- Q_2(\sBox_i^-)} \Bigg) \prod_{\{\sBox_3,\sBox_3'\} \in B_3} e^{|Q_3(\sBox_3,\sBox_3')|}
\Big\rangle_{\beta,N}
\\
&\leq
\prod_{i=1,2} \Bigg( \Big\langle \prod_{\sBox_i^+ \in B_i^+} e^{5 Q_i(\sBox_i^+)}\Big\rangle_{\beta,N}^\frac 15  \Big\langle \prod_{\sBox_i^- \in B_i^-}e^{5Q_i(\sBox_i^-)}\Big\rangle_{\beta,N}^\frac 15 \Bigg)
\\
&\quad\quad\quad
\times \Big\langle \prod_{\{\sBox_3,\sBox_3'\} \in B_3} e^{5|Q_3(\sBox_3,\sBox_3')|} \Big\rangle_{\beta,N}^
\frac 15.
\end{split}
\end{equs}

Let $i=1,2$. Without loss of generality, we use Proposition \ref{prop:chessboard_estimates} to estimate
\begin{equs}
\Big\langle \prod_{\Box \in B_i^+} e^{5Q_i(\Box)} \Big\rangle_{\beta,N}.
\end{equs}

Define $F_\sBox = e^{5Q_i(\sBox)}$ if $\Box \in B_i^+$ and $1$ otherwise. For each $\Box \in \BBN$, we generate the family of functions $\{ F_{\sBox,\sBox'} : \sBox' \in \BBN \}$ as in Proposition \ref{prop:chessboard_estimates}. Note that for $\Box, \Box' \in \BBN$ such that $\Box$ and $\Box'$ are nearest-neighbours,
\begin{equs}
\cR e^{5Q_i(\sBox)}
=
e^{5Q_i(\sBox')}.
\end{equs}
where $\cR$ is the reflection across the unique hyperplane containing the shared face of $\Box$ and $\Box'$. Thus, we have $F_{\sBox,\sBox'} = e^{5Q_i(\sBox')}$ for every $\Box \in B_i^+$ and $\Box' \in \BB_N^B$. If $\Box \not\in B_i^+$, we have $F_{\sBox,\sBox'}=1$ for every $\Box' \in \BBN$.  

Lemma \ref{lem:expQdefn} ensures that $F_\sBox \in L^2(\nubn)$ is $\Box$-measurable for every $\Box \in \BBN$. Hence, by Proposition \ref{prop:chessboard_estimates}, we obtain
\begin{equs}
\Big\langle \prod_{\sBox \in B_i^+} e^{5Q_i(\sBox)} \Big\rangle_{\beta,N} 
\leq 
\prod_{\sBox \in B_i^+} \Big\langle \prod_{\sBox' \in \BBN} e^{5Q_i(\sBox')}\Big\rangle_{\beta,N}^\frac{1}{N^3}.
\end{equs}

Therefore, by Proposition \ref{prop: q bound main}, there exists $C_Q'>0$ such that, for all $\beta$ sufficiently large,
\begin{equs} \label{eq:q_iterm}
\Big\langle \prod_{\sBox \in B_i^+}e^{5Q_i(\sBox)} \Big\rangle_{\beta,N} 
\leq 
e^{C_Q'|B_i^+|}.
\end{equs}

For the remaining term involving $Q_3$, partition $B_3 = \bigcup_{k=1}^6 B_3^{(k)}$ such that each $B_3^{(k)}$ is a set of disjoint pairs of nearest neighbour blocks, all with same orientation. Then, by H\"older's inequality,
\begin{equs} \label{eq:q_3nonoverlap}
	\Big\langle \prod_{\{\sBox,\sBox'\} \in B_3}e^{5|Q_3(\sBox,\sBox')|} \Big\rangle_{\beta,N} 
	\leq 
	\prod_{k=1}^6 \Big\langle \prod_{\{\sBox,\sBox'\}\in B_3^{(k)}}e^{30|Q_3(\sBox,\sBox')|} \Big\rangle_{\beta,N}^\frac 16. 
\end{equs}
Assuming that we have established that there exists $C_Q'>0$ such that
\begin{equs}
	\Big\langle \prod_{\{\sBox,\sBox'\}\in B_3^{(k)}}e^{30|Q_3(\sBox,\sBox')|} \Big\rangle_{\beta,N}
	\leq
	e^{C_Q'|B_3^{(k)}|}
\end{equs}
for every $k \in \{1,\dots,6\}$, then \eqref{eq:q_3nonoverlap} yields
\begin{equs}
	\Big\langle \prod_{\{\sBox,\sBox'\}\in B_3}e^{5|Q_3(\sBox,\sBox')|} \Big\rangle_{\beta,N}
	\leq
	e^{\frac{C_Q'}{6}|B_3|}.
\end{equs}

Hence, without loss of generality, we may assume $B_3$ is a set of disjoint pairs of nearest neighbour blocks, all of the same orientation. 

Define $F_B = e^{5 |Q_3(\sBox,\sBox')|}$ for any $B=\{\Box,\Box'\} \in B_3$ and $1$ otherwise. Note that for any two pairs of nearest-neighbour blocks, $\{ \Box,\Box'\},\{\tilde\Box,\tilde\Box'\} \subset \BBN$,
\begin{equs}
\cR e^{5|Q_3(\sBox,\sBox')|}
= 
e^{5|Q_3(\tilde\sBox,\tilde\sBox')|}
\end{equs}
where $\cR$ is the reflection across the unique hyperplane containing the shared face of $\Box \cup \Box'$ and $\tilde\Box \cup \tilde \Box'$. Thus, for any $B=\{\Box,\Box'\} \in B_3$ and $B'=\{\tilde\Box,\tilde\Box'\} \in \BB_N^B$, we have $F_{B,B'} = e^{5|Q_3(\tilde\sBox,\tilde\sBox')|}$. If $B \not\in B_3$, then we have $F_{B,B'}=1$ for all $B' \in \BB_N^B$.

Lemma \ref{lem:expQdefn} ensures that $\exp(|Q_3(\Box,\Box')|)$ is $\Box\cup\Box'$-measurable. Thus, applying Propositions \ref{prop:chessboard_estimates} and \ref{prop: q bound main}, there exists $C_Q'>0$ such that, for all $\beta$ sufficiently large,
\begin{equs} \label{eq:q_3_end}
\begin{split}
\Big\langle \prod_{B=\{\sBox,\sBox'\} \in B_3} &e^{5|Q_3(\{\sBox,\sBox'\})|} \Big\rangle_{\beta,N}
\\
&\leq
\prod_{B = \{\sBox,\sBox'\} \in B_3} \Big\langle \prod_{B'=\{\tilde\sBox,\tilde\sBox'\} \in \BB_N^B} e^{5| Q_3(\tilde\sBox,\tilde\sBox')|} \Big\rangle_{\beta,N}^\frac{2}{N^3}
\\
&\leq 
e^{2C_Q'|B_3|}.
\end{split}
\end{equs}

Inserting \eqref{eq:q_iterm} and \eqref{eq:q_3_end} into \eqref{eq:q_holder}, and taking into account \eqref{eq:q_3nonoverlap}, yields \eqref{eq:coshreduction2} with $\tilde C_Q = \frac{C_Q'}{15}$, thereby finishing the proof. 
\end{proof}

\subsection{Equivalence of the lattice and Fourier cutoffs} \label{subsec: spde}

This section is devoted to a proof of Proposition \ref{prop: equivalence of cutoffs} using stochastic quantisation techniques. In Section \ref{subsec: giving meaning}, we give a rigorous interpretation to \eqref{def: dynamical phi4} via the change of variables \eqref{def: dynamical phi4 ansatz}. Subsequently, in Section \ref{subsec: stochastic quantisation}, we establish that $\nubn$ is the unique invariant measure of \eqref{def: dynamical phi4}, see Proposition \ref{prop: stochastic quantisation}. In Section \ref{subsec: proof of rp}, we first establish that local solutions of spectral Galerkin and lattice approximations to \eqref{def: dynamical phi4} converge to the same limit (see Propositions \ref{prop: local sg to phi4} and \ref{prop: lattice convergence}); these approximations admit unique invariant measures given by $\nu_{\beta,N,K}$ and $\tilde\nu_{\beta,N,\varepsilon}$, respectively. Then, using the global existence of solutions and uniqueness of the invariant measure of \eqref{def: dynamical phi4}, we show that both of these measures converge to $\nubn$ as the cutoffs are removed.

\subsubsection{Giving a meaning to \eqref{def: dynamical phi4}} \label{subsec: giving meaning}

Let $\xi$ be space-time white noise on $\TTN$ defined on a probability space $(\Omegab, \PPb)$. This means that $\xi$ is a Gaussian random distribution on $\Omega$ satisfying
\begin{equs}
\EEb [\xi(\Phi) \xi(\Psi)]
= 
\int_0^\infty \intx \Phi \Psi dx dt	
\end{equs}
where $\Phi,\Psi\in C^\infty(\RR_+\times\TTN)$ and $\EEb$ denotes expectation with respect to $\PPb$. We use the colour blue here to distinguish between the space random processes defined in Section \ref{sec: bd} and the space-time random processes that we consider here.

We interpret \eqref{def: dynamical phi4} as the limit of renormalised approximations. For every $K \in (0,\infty)$, the Glauber dynamics of $\nu_{\beta,N,K}$ is given by the stochastic PDE
\begin{equs}\label{def: spectral galerkin}
\begin{split}
(\partial_t -\Delta + \eta)	\Phi_K
&=
-\frac 4\beta \rho_K (\rho_K \Phi_K)^3 
\\
&\quad
+ \Big(4+\eta+ \frac{12}\beta \<tadpole>_K + \frac {2\gamma_K}{\beta^2} \Big) \rho^2_K \Phi_K + \sqrt 2 \xi.
\end{split}
\end{equs}
Above, $\rho_K$ is as in Section \ref{sec: model} and we recall $\rho_K^2 \neq \rho_K$; $\<tadpole>_K$ is defined in \eqref{def: tadpole}; and $\gamma_K = -4^2 \cdot 3 \<sunset>_K$, where $\<sunset>_K$ is defined in \eqref{def: sunset}.

\begin{rem}
Recall that the Glauber dynamics for the measure $\nu$ with formal density $d\nu(\phi) \propto e^{-\cH(\phi)}\prod_{x\in \TTN} d\phi(x)$ is given by the (overdamped) Langevin equation 
\begin{equs}
\partial_t \Phi(t)
=
\partial_\phi \cH(\Phi(t)) + \sqrt 2 \xi	
\end{equs}
where $\partial_\phi \cH$ denotes the functional derivative of $\cH$.
\end{rem}

For fixed $K$, the (almost sure) global existence and uniqueness of mild solutions to \eqref{def: spectral galerkin} is standard (see e.g. \cite[Section III]{DPZ88}). Moreover, $\nu_{\beta,N,K}$ is its unique invariant measure (see \cite[Theorem 2]{Z89}). The approximations \eqref{def: spectral galerkin}, which we call spectral Galerkin approximations, are natural in our context since $\nubn$ is constructed as the weak limit of $\nu_{\beta,N,K}$ as $K \rightarrow \infty$. 

The difficulty in obtaining a local well-posedness theory that is stable in the limit $K \rightarrow \infty$ lies in the roughness of the white noise $\xi$. The key idea is to exploit that the small-scale behaviour of solutions to \eqref{def: spectral galerkin} is governed by the Ornstein-Uhlenbeck process
\begin{equs}
\<1b>
=
(\partial_t - \Delta + \eta)^{-1} \sqrt 2 \xi.	
\end{equs}
This allows us to obtain an expansion of $\Phi_K$ in terms of explicit (renormalised) multilinear functions of $\<1b>$, which give a more detailed description of the small-scale behaviour of $\Phi_K$, plus a more regular remainder term. Given the regularities of these explicit stochastic terms, the local solution theory then follows from deterministic arguments. 

\begin{rem}
We are only concerned with the limit $K \rightarrow \infty$ in \eqref{def: spectral galerkin}. We do not try to make sense of the joint $K,N \rightarrow \infty$ limit. 	
\end{rem}

We use the paracontrolled distribution approach of \cite{MW17}, which is modification of the framework of \cite{CC18} (both influenced by the seminal work of \cite{GIP15}). In this approach, the expansion of $\Phi_K$ is given by an ansatz, see \eqref{def: spectral galerkin ansatz}, that has similarities to the change of variables encountered in Section \ref{subsec: change of variables}. See Remark \ref{rem: ansatz comparison}. There are also related approaches via regularity structures \cite{H14,H16, MW18} and renormalisation group \cite{K16}, but we do not discuss them further. 

For every $K \in (0,\infty)$, define
\begin{equs}
\<1b>_K 
&=
\rho_K \<1b>
\\
\<2b>_K
&=
\<1b>_K^2 - \<tadpole>_K
\\
\<3b>_K
&=
\<1b>_K^3 - 3 \<tadpole>_K \<1b>_K
\\
\<20b>_K
&=
(\partial_t -\Delta + \eta)^{-1} \rho_K \<2b>_K
\\
\<30b>_K
&=
(\partial_t -\Delta + \eta)^{-1} \rho_K \<3b>_K
\\
\<31b>_K
&=
\<1b>_K \pe \rho_K \<30b>_K
\\
\<22b>_K
&=
\<2b>_K \pe \rho_K \<20b>_K - \frac 23 \<sunset>_K
\\
\<32b>_K
&=
\<2b>_K \pe \rho_K \<30b>_K - 2 \<sunset>_K \<1b>_K.
\end{equs}
We recall that the colour blue is used to distinguish between the above space-time diagrams and the space diagrams of Section \ref{subsec: diagrams}.

For any $T>0$, the vector $\Xib_K = \Big( \<1b>_K, \<2b>_K, \<30b>_K, \<31b>_K, \<22b>_K, \<32b>_K \Big)$ is space-time stationary and almost surely an element of the Banach space
\begin{equs}
\cX_T
&=
C([0,T]; \cC^{-\frac 12 - \kappa})\times C([0,T]; \cC^{-1-\kappa})
\\
&\quad
\times \Big( C([0,T]; \cC^{\frac 12 - \kappa}) \cap C^\frac 18([0,T]; \cC^{\frac 14 - \kappa}) \Big)
\\
&\quad
\times C([0,T]; \cC^{-\kappa})\times C([0,T]; \cC^{-\kappa}) \times C([0,T]; \cC^{-\frac 12 -\kappa})
\end{equs}
where the norm on $\cX_T$ is given by the maximum of the norms on the components. Above, for any $s \in \RR$, $C([0,T];\cC^s)$ consists of continuous functions $\Phi :[0,T] \rightarrow \cC^s$ and is a Banach space under the norm $\sup_{t\in[0,T]}\| \cdot \|_{\cC^s}$. In addition, for any $\alpha \in (0,1)$, $C^\alpha([0,T];\cC^s)$ consists of $\alpha$-H\"older continuous functions $\Phi : [0,T] \rightarrow \cC^s$ and is a Banach space under the norm $\| \cdot \|_{C([0,T];\cC^s)} + | \cdot |_{\alpha, T}$ where
\begin{equs}
| \Phi |_{\alpha,T}
=
\sup_{0<s<t<T} \frac{\|\Phi(t) - \Phi(s)\|_{\cC^s}}{|t-s|^\alpha}.	
\end{equs}

\begin{prop}\label{prop: blue diagrams}
There exists a stochastic process $\Xib = (\<1b>, \<2b>, \<30b>, \<31b>, \<22b>, \<32b>)$ such that, for every $T > 0$, $\Xib \in \cX_T$ almost surely and 
\begin{equs}
\lim_{K \rightarrow \infty}\EEb \| \Xib_K - \Xib \|_{\cX_T}
=
0.
\end{equs}
\end{prop}
\begin{proof}

The proof follows from \cite[Section 4]{CC18} (see also \cite{MWX17} and \cite[Section 10]{H14}). The only subtlety is to check that the renormalisation constants $\<tadpole>_K$ and $\<sunset>_K$, which were determined by the field theory $\nubn$, are sufficient to renormalise the \textit{space-time} diagrams appearing in the analysis of the SPDE. Precisely, it suffices to show $\EE [\<2b>_K^2(t,x)] = \<tadpole>_K$ and $\EE \Big[ \<2b>_K \rho_K \<20b>_K (t,x) \Big] = \frac 23 \<sunset>_K$ for every $(t,x) \in \RR_+\times\TTN$.

There exists a set of complex Brownian motions $\{ W^n(\bullet) \}_{n \in (N^{-1}\ZZ)^3}$ defined on $(\Omegab,\PPb)$, independent modulo the condition $W^n(\bullet) = \overline{W^{-n}(\bullet)}$, such that
\begin{equs}
\xi(\phi)
&=
\frac{1}{N^3} \sum_{n \in (N^{-1}\ZZ)^3} \int_\RR \cF(\phi)(t,n) N^\frac 32 dW^n(t)	
\end{equs}
for every $\phi \in L^2(\RR \times \TTN)$. 

For $t \geq 0$ and $n \in (N^{-1}\ZZ)^3$, let $H(t,n) = e^{-t \langle n \rangle^2}$ be the (spatial) Fourier transform of the heat kernel associated to $(\partial_t - \Delta + \eta)$. For any $K>0$, define $H_K(t,n) = \rho_K (n) H(t,n)$. We extend both kernels to $t \in \RR$ by setting $H(t,\cdot) = H_K(t,\cdot) = 0$ for any $t < 0$. Then
\begin{equs}
\cF \<1b>_K(t,n)
=
\sqrt 2 N^\frac 32 \int_{\RR} H_K(t-s,n) dW^n(s).	
\end{equs}

By Parseval's theorem and It\^o's isometry,
\begin{equs}
	\EEb & \<1b>_K^2(t,x)
	\\
	&=
	\frac{2}{N^3} \sum_{n_1,n_2 \in (N^{-1}\ZZ)^3} \EEb \Bigg[ \Bigg(\int_\RR H_K(t-s,n_1) dW^{n_1}(s) \Bigg) \Bigg(\int_\RR H_K(t-s, n_2) dW^{n_2}(s) \Bigg) \Bigg]
	\\
	&=
	\frac{2}{N^3} \sum_{n \in (N^{-1}\ZZ)^3} \rho_K^2(n) \int_{-\infty}^t e^{-2(t-s)\langle n \rangle^2} ds
	=
	\<tadpole>_K
\end{equs}
for all $(t,x) \in \RR_+\times\TTN$. With this observation the convergence of $\<1b>_K, \<2b>_K$, $\<30b>_K$ and $\<31b>_K$ follows from mild adaptations of \cite[Section 4]{CC18}. 

For the remaining two diagrams, one can show from arguments in \cite[Section 4]{CC18} that
\begin{equs}
\rho_K \<20b>_K \pe \<2b>_K - \EEb \Big[\rho_K \<20b>_K  \<2b>_K\Big] 
\text{ and } 
\rho_K \<30b>_K \pe \<2b>_K - 3 \EEb \Big[\rho_K \<20b>_K \<2b>_K \Big] \<1b>_K
\end{equs}
converge to well-defined space-time distributions. 

Writing
\begin{equs}
\<2b>(t,x)
=
\frac{1}{N^3}\sum_{n_1, n_2 \in (N^{-1}\ZZ)^3} e_{n_1+n_2}(x)\int_{\RR^2} H_K(t-s,n_1)H_K(t-r,n_2) dW^{n_1}(s) dW^{n_2}(r) 	
\end{equs}
we have, by Parseval's theorem and It\^o's isometry, 
\begin{equs}
\EEb &\Big[ \<20b>_K \rho_K \<2b> (t,x) \Big]
\\
&=
\frac{8}{N^6} \EEb \Bigg[ \sum_{\substack{n_1,n_2,n_3,n_4 \in (N^{-1}\ZZ)^3 \\ n_1 + n_3 = n_2 +n_4 = 0}} e_{n_1+n_2+n_3+n_4}(x) \rho_K(n_3+n_4)
\\
&\quad\quad\quad 
\times \int_{\RR^5} H_K(t-s, n_1 + n_2) H_K(s-u_1, n_1) H_K(s-u_2, n_2) H_K(t-u_3, n_3)
\\
&\quad\quad\quad\quad 
\times  H_K(t-u_4,n_4) dW^{n_1}(u_1) dW^{n_2}(u_2) dW^{n_3}(u_3) dW^{n_4}(u_4) ds \Bigg]
\\
&=
\frac{8}{N^6} \sum_{\substack{n_1,n_2,n_3,n_4 \in (N^{-1}\ZZ)^3 \\ n_1 = -n_3, n_2 = - n_4 }} \rho_K(n_3 + n_4) \int_{\RR^3} H_K(t-s,n_1+n_2) H_K(s-u_1,n_1) 
\\
&\quad\quad\quad\quad
\times H_K(s-u_2,n_2) H_K(t-u_1,n_1) H_K(t-u_2,n_2) du_1 du_2 ds
\\
&=
\frac{8}{N^6} \sum_{n_1, n_2 \in (N^{-1}\ZZ)^3} \rho_K^2(n_1+n_2)\rho_K^2(n_1)\rho_K^2(n_2) \int_{\RR} H(t-s, n_1+n_2)H(t-s,n_1)
\\
&\quad\quad\quad\quad
\times H(t-s,n_2) \int_{\RR^2} H(2(s-u_1),n_1)H(2(s-u_2),n_2) du_1 du_2 ds
\\
&=
\frac{2}{N^6} \sum_{n_1,n_2 \in (N^{-1}\ZZ)^3} \frac{\rho_K^2(n_1) \rho_K^2(n_2)\rho_K^2(n_1+n_2) }{\langle n_1 \rangle^2 \langle n_2 \rangle^2 (\langle n_1 + n_2 \rangle^2 + \langle n_1 \rangle^2 + \langle n_2 \rangle^2)}.
\end{equs}

By symmetry,
\begin{equs}
\frac{2}{N^6} &\sum_{n_1,n_2 \in (N^{-1}\ZZ)^3} \frac{\rho_K^2(n_1)\rho_K^2(n_2)\rho_K^2(n_1+n_2)}{\langle n_1 \rangle^2 \langle n_2 \rangle^2 (\langle n_1 + n_2 \rangle^2 + \langle n_1 \rangle^2 + \langle n_2 \rangle^2)}	
\\
&=
\frac {2}{3N^6} \sum_{n_1 + n_2 + n_3 = 0} \frac{\rho_K^2(n_1)\rho_K^2(n_2)\rho_K^2(n_3)}{\langle n_1 \rangle^2 + \langle n_2 \rangle^2 + \langle n_3 \rangle^2} \\
&\quad\quad\quad
\times \Big( \frac 1 {\langle n_1 \rangle^2 \langle n_2 \rangle^2} + \frac 1 {\langle n_2 \rangle^2 \langle n_1 + n_2 \rangle^2} + \frac 1 {\langle n_1 + n_2 \rangle^2 \langle n_1 \rangle^2} \Big)
\\
&=
\frac 23 \<sunset>_K
\end{equs}
thereby completing the proof.
\end{proof}

We return now to the solution theory for \eqref{def: dynamical phi4}/\eqref{def: spectral galerkin}. Fix $K \in (0,\infty)$. Using the change of variables
\begin{equs} \label{def: spectral galerkin ansatz}
\Phi_K
=
\<1b> - \frac{4}\beta \<30b>_K + \Upsilon_K + \Theta_K
\end{equs}
we say that $\Phi_K$ is a mild solution of \eqref{def: spectral galerkin} with initial data $\phi_0 \in \cC^{-\frac 12-\kappa}$ if $(\Upsilon_K, \Theta_K)$ is a mild solution to the system of equations
\begin{equs} \label{def: spectral galerkin system}
\begin{split}
(\partial_t - \Delta + \eta) \Upsilon_K
&= 
F_K(\Upsilon_K, \Theta_K; \Xib_K)
\\
(\partial_t - \Delta + \eta) \Theta_K
&=
G_K(\Upsilon_K,\Theta_K;  \Xib_K)
\end{split}
\end{equs}
where
\begin{equs}
F_K(\Upsilon_K,\Theta_K;\Xib_K)
&=
-\frac{4\cdot 3}{\beta}\rho_K \Big\{ \<2b>_K \pg \rho_K(\Phi_K - \<1b>) \Big\}
\\
G_K(\Upsilon_K,\Theta_K;\Xib_K)
&=
-\frac{4 \cdot 3}{\beta} \rho_K \Big\{ \<2b>_K \pe \Big( - \frac 4\beta \rho_K \<30b>_K + \rho_K(\Upsilon_K + \Theta_K) \Big) \Big\}
\\
&\quad
- \frac{4\cdot 3}{\beta} \rho_K \Big\{ \<2b>_K \pl \rho_K(\Phi_K - \<1b>) + \<1b>_K \big(\rho_K (\Phi_K - \<1b>) \big)^2 \Big\}
\\
&\quad
- \frac 4\beta \rho_K \big( \rho_K(\Phi_K - \<1b>) \big)^3 + \Big(4 + \eta + \frac{2\gamma_K}{\beta^2}\Big) \rho_K \Phi_K
\end{equs}
with initial data $(\Upsilon_K(0,\cdot),\Theta_K(0,\cdot)) = \Big(0,\phi_0 + \sqrt 2 \<xi>(0) - \frac{4\cdot(\sqrt 2)^3}\beta \<30b>_K(0) \Big)$.

We split $G_K(\Upsilon_K,\Theta_K;\Xib_K) = G^1_K(\Upsilon_K,\Theta_K;\Xib_K) + G^2_K(\Upsilon_K,\Theta_K;\Xib_K)$,
\begin{equs}
G^1_K(\Upsilon_K,\Theta_K;\Xib_K)
&=
\frac{4^2\cdot 3}{\beta^2} \rho_K \Big\{ \<32b>_K + 3\<22b>_K \rho_K (\Phi_K - \<1b>) \Big\} 
\\
&\quad
+ G^{1,a}_K(\Upsilon_K,\Theta_K;\Xib_K) + G^{1,b}_K(\Upsilon_K,\Theta_K;\Xib_K)
\\
G^2_K(\Upsilon_K,\Theta_K;\Xib_K)
&=
- \frac{4 \cdot 3}{\beta} \rho_K \Big\{ \<2b>_K \pe \rho_K \Theta_K + \<2b>_K \pl \rho_K (\Phi_K - \<1b>) 
\\
&\quad
+ \<1b>_K \big( \rho_K(\Phi_K - \<1b>) \big)^2 \Big\} - \frac 4\beta \rho_K \big( \rho_K(\Phi_K - \<1b>) \big)^3 + (4+\eta) \rho_K \Phi_K	
\end{equs}
where $G_K^{1,a}(\Upsilon_K,\Theta_K;\Xib_K)$ and $G_K^{2,a}(\Upsilon_K,\Theta_K;\Xib_K)$ are commutator terms defined through the manipulations
\begin{equs} \label{eq: spectral galerkin commutator}
\begin{split}
- &\frac{4 \cdot 3}\beta \rho_K \Big\{ \<2b>_K \pe \rho_K \Upsilon_K \Big\}
\\
&=
\frac{4^2 \cdot 3^2}{\beta^2} \rho_K \Big\{ \<2b>_K \pe \rho_K (\partial_t - \Delta + \eta)^{-1} \big( \rho_K \{ \<2b>_K \pg \rho_K (\Phi_K - \<1b>) \} \big) \Big\}
\\
&=
\frac{4^2 \cdot 3^2}{\beta^2} \rho_K \Big\{ \<2b>_K \pe \rho_K \Big( \<20b>_K \pg \rho_K(\Phi_K - \<1b>) \Big) \Big\} + G_K^{1,a}
\\
&=
\frac{4^2\cdot 3^2}{\beta^2} \rho_K \Big\{ \Big( \<2b>_K \pe \rho_K \<20b>_K \Big) \rho_K(\Phi_K - \<1b>) \Big\} + G_K^{1,a} + G_K^{1,b}.
\end{split}
\end{equs}

The precise choice of the splitting of $\Phi_K - \<1b> + \frac 4\beta \<30b>_K$ into $\Upsilon_K$ and $\Theta_K$ is explained in detail in \cite[Introduction]{MW17}. For our purposes, it suffices to note that $\Upsilon_K$ captures the small-scale behaviour of this difference.  On the other hand, $\Theta_K$ captures the large-scale behaviour: the term $G_K^2$ contains a cubic damping term in $\Theta_K$ (i.e. with a good sign). Finally, we note that there is a redundancy in the specification of initial condition: any choice such that $\Upsilon_K(0,\cdot) + \Theta_K(0,\cdot) = \phi_0 + \<1b>(0) - \frac{4}\beta \<30b>(0)$ is sufficient. Our choice is informed by Remark 1.3 in \cite{MW17}.

\begin{rem} \label{rem: ansatz comparison}
Rewriting \eqref{def: spectral galerkin ansatz} as
\begin{equs}
\Phi_K
=
\<1b> - \frac{4}\beta \<30b>_K - \frac{4\cdot 3}\beta (\partial_t - \Delta + \eta)^{-1} \rho_K \Big\{ \<2b>_K \pg \rho_K(\Phi_K -\<1b>) \Big\} + \Theta_K
\end{equs}
we note the similarity between the change of variables for the stochastic PDE given above and for the field theory in \eqref{eq: intermediate integrated ansatz}.
\end{rem}

Formally taking $K \rightarrow \infty$ in \eqref{def: spectral galerkin system} leads us to the following system:
\begin{equs} \label{def: dynamical phi4 system}
\begin{split}
(\partial_t - \Delta + \eta) \Upsilon
&=
F(\Upsilon, \Theta; \Xib)
\\
(\partial_t - \Delta + \eta) \Theta
&=
G(\Upsilon, \Theta; \Xib)	
\end{split}
\end{equs}
where
\begin{equs}
F(\Upsilon, \Theta; \Xib)
&=
- \frac{4 \cdot 3 }\beta \<2b> \pg \Big( - \frac{4}\beta \<30b> + \Upsilon + \Theta \Big)
\\
G(\Upsilon, \Theta; \Xib)
&=
G^1(\Upsilon, \Theta; \Xib) + G^2(\Upsilon, \Theta; \Xib)
\\
G^1(\Upsilon, \Theta; \Xib)
&=
\frac{4^2 \cdot 3}{\beta^2} \Bigg( \<32b> + 3\<22b> \Big( - \frac 4\beta \<30b> + \Upsilon + \Theta \Big) \Bigg)
\\
&\quad
+ G^{1,a}(\Upsilon,\Theta;\Xib) + G^{2,b}(\Upsilon,\Theta;\Xib)
\\
G^2(\Upsilon, \Theta;\Xib)
&=
- \frac{4\cdot 3}\beta \Bigg( \<2b> \pe \Theta + \<2b> \pl \Big( - \frac 4\beta \<30b> + \Upsilon + \Theta \Big) \Bigg)
\\
&\quad
- \frac{4\cdot 3}{\beta} \<1b> \Big( - \frac 4\beta \<30b> + \Upsilon + \Theta \Big)^2 
\\
&\quad
- \frac 4\beta \Big( - \frac 4\beta \<30b> + \Upsilon + \Theta \Big)^3 
\\
&\quad
+(4+\eta)\Big(1 - \frac 4\beta \<30b> + \Upsilon + \Theta \Big)
\end{equs}
and $G^{1,a}$ and $G^{1,b}$ are commutator terms defined analogously as in \eqref{eq: spectral galerkin commutator}. 

For every $T>0$, define the Banach space
\begin{equs}
\cY_T
&=
\Big[ C([0,T];\cC^{-\frac 35}) \cap C((0,T]; \cC^{\frac 12 + 2\kappa}) \cap C^\frac 18 ( (0,T]; L^\infty ) \Big]
\\
&\quad
\times \Big[ C([0,T]; \cC^{-\frac 35}) \cap C((0,T]; \cC^{1+2\kappa}) \cap C^\frac 18 ( (0,T]; L^\infty ) \Big]	
\end{equs}
equipped with the norm 
\begin{equs}
&\| (\Upsilon, \Theta) \|_{\cY_T}
\\
&\quad 
=
\max \Bigg\{ \sup_{0 \leq t \leq T} \| \Upsilon(t) \|_{\cC^{-\frac 35}}, \sup_{0 < t \leq T}t^\frac 35 \| \Upsilon(t) \|_{\cC^{\frac 12 + 2\kappa}}, \sup_{0 < s < t \leq T} s^\frac 12 \frac{\| \Upsilon(t) - \Upsilon(s) \|_{L^\infty}}{|t-s|^\frac 18},
\\
&\quad\quad
\sup_{0 \leq t \leq T} \| \Theta(t) \|_{\cC^{-\frac 35}}, \sup_{0 < t \leq T} t^\frac{17}{20} \| \Theta(t) \|_{\cC^{1+2\kappa}}, \sup_{0 < s < t \leq T} s^\frac 12 \frac{\| \Theta(t) - \Theta(s) \|_{L^\infty}}{|t-s|^\frac 18} \Bigg\}.
\end{equs}

\begin{rem}
The choice of exponents	in function spaces in $\cY_T$, as well as the choice of exponents in the blow-up at $t=0$ in $\| \cdot \|_{\cY_T}$, corresponds to the one made in \cite{MW17}. It is arbitrary to an extent: it depends on the choice of initial condition, which must have Besov-H\"older regularity strictly better than $-\frac 23$. 
\end{rem}

The local well-posedness of \eqref{def: dynamical phi4 system} follows from entirely deterministic arguments, so we state it with $\Xib$ replaced by any deterministic $\tilde\Xib$.
\begin{prop}\label{prop: lwp dynamical phi4 system}
Let $\tilde\Xib \in \cX_{T_0}$ for any $T_0>0$, and let $(\Upsilon_0,\Theta_0) \in \cC^{-\frac 35}\times \cC^{-\frac 35}$.  Then, there exists $T=T(\|\tilde\Xib\|_{\cX_{T_0}}, \|\Upsilon_0\|_{\cC^{-\frac 35}}, \|\Theta_0 \|_{\cC^{-\frac 35}}) \in (0, T_0]$ such that there is a unique mild solution $(\Upsilon, \Theta) \in \cY_T$ to \eqref{def: dynamical phi4 system} with initial data $(\Upsilon_0,\Theta_0)$. 

In addition, let $\tilde\Xib, \Xib' \in \cX_{T_0}$ such that $\| \tilde\Xib \|_{\cX_{T_0}}, \| \Xib' \|_{\cX_{T_0}} \leq R$ for some $R>0$, and let $(\Upsilon_0^1,\Theta_0^1),(\Upsilon_0^2,\Theta_0^2) \in 
\cC^{-\frac 35} \times \cC^{-\frac 35}$. Let the respective solutions to \eqref{def: dynamical phi4 system} be $(\Upsilon^1,\Theta^1) \in \cY_{T_1}$ and $(\Upsilon^2,\Theta^2) \in \cY_{T_2}$ and define $T=\min(T_1,T_2)$. Then there exists $C=C(R)>0$ such that
\begin{equs}
\| (\Upsilon^1,\Theta^1) - (\Upsilon^2,\Theta^2) \|_{\cY_T}
\leq
C \Big( \| \Upsilon_0^1 - \Upsilon_0^2 \|_{\cC^{-\frac 35}} + \| \Theta_0^1 - \Theta_0^2 \|_{\cC^{-\frac 35}} + \| \tilde\Xib - \Xib' \|_{\cX_{T_0}} \Big).	
\end{equs}
\end{prop}

\begin{proof}
Proposition \ref{prop: lwp dynamical phi4 system} is proven in Theorem 2.1 \cite{MW17} (see also Theorem 3.1 \cite{CC18}) by showing that the mild solution map 
\begin{equs}
(\Upsilon,\Theta)
&\mapsto
\Big( (\partial_t - \Delta + \eta)^{-1} \Upsilon_0, (\partial_t - \Delta + \eta)^{-1} \Theta_0 \Big) 
\\
&\quad\quad\quad
+ \Big( (\partial_t - \Delta + \eta)^{-1} F(\Upsilon,\Theta;\tilde\Xib), (\partial_t - \Delta + \eta)^{-1} G(\Upsilon,\Theta;\tilde\Xib) \Big)	
\end{equs}
is a contraction in the ball
\begin{equs}
\cY_{T,M}
=
\Big\{ (\tilde\Upsilon, \tilde\Theta) \in \cY_T : \| (\tilde\Upsilon, \tilde\Theta) \|_{\cY_T} \leq M \Big\}
\end{equs}
provided that $T$ is taken sufficiently small and $M$ is taken sufficiently large (both depending on the norm of the initial data and of $\|\tilde\Xib\|_{\cX_{T_0}}$). 
\end{proof}

We say that $\Phi \in C([0,T];\cC^{-\frac 12-\kappa})$ is a mild solution to \eqref{def: dynamical phi4} with initial data $\phi_0 \in \cC^{-\frac 12-\kappa}$ if
\begin{equs} \label{def: dynamical phi4 ansatz}
\Phi
=
\<1b> - \frac{4}\beta \<30b> + \Upsilon + \Theta 	
\end{equs}
where $(\Upsilon,\Theta) \in \cY_T$ is a solution to \eqref{def: dynamical phi4 system} with $\Xib$ as in Proposition \ref{prop: blue diagrams} and initial data $\Big(0,\phi_0 + \<1b>(0) - \frac{4}\beta \<30b>(0)\Big)$.

\begin{prop}\label{prop: local sg to phi4}
For any $\phi_0 \in \cC^{-\frac 12-\kappa}$, let $\Phi \in C([0,T];\cC^{-\frac 12-\kappa})$ be the unique solution of \eqref{def: dynamical phi4} with initial data $\phi_0$ up to time $T>0$. In addition, for any $K \in (0,\infty)$, let $\Phi_K \in C(\RR_+;\cC^{-\frac 12-\kappa})$ be the unique global solution of \eqref{def: spectral galerkin} with initial data $\rho_K\phi_0$. 

Then,
\begin{equs}
\lim_{K \rightarrow \infty} \EEb \| \Phi - \Phi_K \|_{C([0,T];\cC^{-\frac 12-\kappa})}
=
0.	
\end{equs}
\end{prop}

\begin{proof}
It suffices to show convergence of $(\Upsilon_K,\Theta_K)$ to $(\Upsilon,\Theta)$ as $K \rightarrow \infty$. This follows from Proposition \ref{prop: blue diagrams} and mild adaptations of arguments in \cite[Section 2]{MW17}.
\end{proof}
 
Proposition \ref{prop: local sg to phi4} implies that $\Phi_K \rightarrow \Phi$ in probability in $C([0,T];\cC^{-\frac 12-\kappa})$. Local-in-time convergence is not sufficient for our purposes. 

The following proposition establishing global well-posedness of \eqref{def: dynamical phi4}.
\begin{prop}\label{prop: gwp dynamical phi4}
For every $\phi_0 \in \cC^{-\frac 12-\kappa}$ let $\Phi \in C([0,T^*);\cC^{-\frac 12-\kappa})$ be the unique solution to \eqref{def: dynamical phi4} with initial condition $\phi_0$ and where $T^*>0$ is the maximal time of existence. Then $T^* = \infty$ almost surely. 	
\end{prop}

\begin{proof}
Proposition \ref{prop: gwp dynamical phi4} is a consequence of a strong a priori bound on solutions to \eqref{def: dynamical phi4 system} established in \cite[Theorem 1.1]{MW17}.
\end{proof}

An immediate corollary of Proposition \ref{prop: gwp dynamical phi4} is a global-in-time convergence result sufficient for our purposes.
\begin{cor}
For every $\phi_0 \in \cC^{-\frac 12-\kappa}$, let $\Phi \in C(\RR_+;\cC^{-\frac 12-\kappa})$ be the unique global solution to \eqref{def: dynamical phi4} with initial condition $\phi_0$. For every $K \in (0,\infty)$, let $\Phi_K \in C(\RR_+;\cC^{-\frac 12-\kappa})$ be the unique global solution to \eqref{def: spectral galerkin} with initial condition $\rho_K \phi_0$.

For every $T>0$,
\begin{equs}
\lim_{K \rightarrow \infty} \EEb \| \Phi_K - \Phi \|_{C([0,T]; \cC^{-\frac 12-\kappa})}
=
0.	
\end{equs}
\end{cor}

\begin{rem} \label{rem: spde infinity}
	The infinite constant in \eqref{def: dynamical phi4} represents the renormalisation constants of the approximating equation \eqref{def: spectral galerkin} going to infinity as $K \rightarrow \infty$. Note that there is a one-parameter family of distinct nontrivial "solutions" to \eqref{def: dynamical phi4} corresponding to taking finite shifts of the renormalisation constants. However, the use of $\Xib$ in the change of variables \eqref{def: dynamical phi4 ansatz} fixes the precise solution.
\end{rem}

\subsubsection{$\nubn$ is the unique invariant measure of \eqref{def: dynamical phi4 ansatz}} \label{subsec: stochastic quantisation}

Denote by $B_b(\cC^{-\frac 12 -\kappa})$ the set of bounded measurable functions on $\cC^{-\frac 12 -\kappa}$ and by $C_b(\cC^{-\frac 12 - \kappa}) \subset B_b(\cC^{-\frac 12-\kappa})$ the set of bounded continuous functions on $\cC^{-\frac 12 - \kappa}$.

Let $\Phi(\cdot;\cdot)$ be the solution map to \eqref{def: dynamical phi4}: for $\phi_0 \in \cC^{-\frac 12 - \kappa}$ and $t \in \RR_+$, $\Phi(t;\phi_0)$ is the solution at time $t$ to \eqref{def: dynamical phi4} with initial condition $\phi_0$. For every $t>0$, define $\cP^{\beta,N}_t:B_b(\cC^{-\frac 12-\kappa}) \rightarrow B_b(\cC^{-\frac 12-\kappa})$ by
\begin{equs}
(\cP^{\beta,N}_t F)(\phi_0)
=
\EEb F(\Phi(t;\phi_0))	
\end{equs}
for $F \in B_b(\cC^{-\frac 12 - \kappa})$, $\phi_0 \in \cC^{-\frac 12-\kappa}$, and $t \in \RR_+$.

\begin{prop} \label{prop: strong feller}
The solution $\Phi$ to \eqref{def: dynamical phi4} is a Markov process and its transition semigroup $(\cP^{\beta,N}_t)_{t\geq 0}$ satisfies the strong Feller property, i.e. $\cP_t:B_b(\cC^{-\frac 12-\kappa}) \rightarrow C_b(\cC^{-\frac 12-\kappa})$. 
\end{prop}

\begin{proof}
See \cite[Theorem 3.2]{HMa18}.
\end{proof}

\begin{prop} \label{prop: stochastic quantisation}
The measure $\nubn$ is the unique invariant measure of \eqref{def: dynamical phi4}.	
\end{prop}

\begin{proof}

By Proposition \ref{prop: phi4 existence} the measures $\nu_{\beta,N,K}$ converge weakly to $\nubn$ as $K \rightarrow \infty$. Hence, by Skorokhod's representation theorem \cite[Theorem 25.6]{B08} we can assume that there exists a sequence of random variables $\{\phi_K\}_{K \in \NN} \subset \cC^{-\frac 12 - \kappa}$ defined on the probability space $(\Omegab,\PPb)$, independent of the white noise $\xi$, such that $\phi_K \sim \nu_{\beta,N,K}$ and $\phi_K$ converges almost surely to a random variable $\phi \sim \nubn$.

 For every $K \in (0,\infty)$, recall that the unique invariant measure of \eqref{def: spectral galerkin} is $\nu_{\beta,N,K}$. Let $\Phi_K$ denote the solution to \eqref{def: spectral galerkin} with random initial data $\phi_K$. Hence, $\Phi_K(t) \sim \nu_{\beta,N,K}$ for all $t \in \RR_+$.
 
Denote by $\Phi$ the solution to \eqref{def: dynamical phi4} with initial condition $\phi$. By Proposition \ref{prop: gwp dynamical phi4}, $\Phi_K(t)$ converges in distribution to $\Phi(t)$ for every $t \in \RR$, which implies $\Phi(t) \sim \nubn$. Thus, $\nubn$ is an invariant measure of \eqref{def: dynamical phi4}. As a consequence of the strong Feller property in Proposition \ref{prop: strong feller}, we obtain that $\nubn$ is the unique invariant measure of \eqref{def: dynamical phi4}.
\end{proof}

\subsubsection{Proof of Proposition \ref{prop: equivalence of cutoffs}} \label{subsec: proof of rp}

The Glauber dynamics of $\tilde\nu_{\beta,N,\varepsilon}$ is given by the system of SDEs
\begin{equs}\label{def: lattice dynamical phi4}
\frac{d}{dt} \tilde\Phi 
&= 
\Delta^\varepsilon \tilde \Phi - \frac 4\beta \tilde\Phi^3 + (4+\delta m^2(\varepsilon,\eta))\tilde\Phi + \sqrt 2 \xi_\varepsilon
\\
\tilde\Phi(0,\cdot)
&=
\varphi(\cdot)
\end{equs}
where $\tilde\Phi : \RR_+\times\TT_N^\varepsilon \rightarrow \RR$, $\varphi \in \RR^{\TT_N^\varepsilon}$, and $\xi_\varepsilon$ is the lattice discretisation of $\xi$ given by
\begin{equs}
\xi_\varepsilon(t,x)
=
\frac{4^3}{\varepsilon^3} \intx \xi(t,x') \1_{|x-x'| \leq \frac \varepsilon 4} dx'.
\end{equs}
Note that the integral above means duality pairing between $\xi(t,\cdot)$ and $\1_{|x-\cdot|\leq \frac \varepsilon 4}$.

For each $\varepsilon > 0$, the global existence and uniqueness of \eqref{def: lattice dynamical phi4}, as well as the fact that $\tilde\nu_{\beta,N,\varepsilon}$ is its unique invariant measure, is well-known. 

The following proposition establishes a global-in-time convergence result for solutions of \eqref{def: lattice dynamical phi4} to solutions of \eqref{def: dynamical phi4}.
\begin{prop} \label{prop: lattice convergence}
For every $\varepsilon > 0$, denote by $\tilde\Phi^\varepsilon$ the unique global solution to \eqref{def: lattice dynamical phi4} with initial data $\varphi_{\varepsilon} \in \RR^{\TT_N^\varepsilon}$. In addition, denote by $\Phi$ the unique global solution to \eqref{def: dynamical phi4} with initial data $\phi \in \cC^{-\frac 12-\kappa}$. 

Then, there exists a choice of constants $\delta m^2(\varepsilon,\eta) \rightarrow \infty$ as $\varepsilon \rightarrow 0$ such that, for every $T > 0$,
\begin{equs}
	\lim_{\varepsilon \rightarrow 0}\EEb \| \Phi - \rext^\varepsilon \tilde\Phi^\varepsilon \|_{C([0,T];\cC^{-\frac 12-\kappa})} 
	=
	0
\end{equs}
provided that
\begin{equs} \label{eq: lattice initial condition}
\lim_{\varepsilon \rightarrow 0}\| \phi - \rext^\varepsilon \varphi_{\varepsilon} \|_{\cC^{-\frac 12-\kappa}}
=
0	
\end{equs}
almost surely.
\end{prop}

\begin{proof}
See \cite[Theorem 1.1]{ZZa18} or \cite[Theorem 1.1]{HMb18}.   
\end{proof}

The next proposition establishes that the lattice measures are tight. 
\begin{prop} \label{prop: tightness of lattice measures}
Let $\delta m^2(\bullet,\eta)$ be as in Proposition \ref{prop: lattice convergence}. Then, $\rext^\varepsilon_*\tilde\nu_{\beta,N,\varepsilon}$ converges weakly to a measure $\nu$ as $\varepsilon \rightarrow 0$. 	
\end{prop}

\begin{proof}
See \cite{P75, BFS83, GH18}.
\end{proof}

\begin{proof}[Proof of Proposition \ref{prop: equivalence of cutoffs}]
For every $\varepsilon>0$, let $\varphi_\varepsilon \sim \tilde\nu_{\beta,N,\varepsilon}$ be a random variable on $(\Omegab,\PPb)$ and independent of the white noise $\xi$. By Proposition \ref{prop: tightness of lattice measures} and in light of Skorokhod's representation theorem \cite[Theorem 25.6]{B08}, we may assume that $\rext^\varepsilon \varphi_\varepsilon$ converges almost surely to $\phi \sim \nu$ as $\varepsilon \rightarrow 0$. Reflection positivity is preserved by weak limits hence, by Lemma \ref{lem: reflection positivity of lattice measures}, $\nu$ is reflection positive.

Denote by $\tilde\Phi^\varepsilon$ the solution to \eqref{def: lattice dynamical phi4} with initial data $\varphi_\varepsilon$. Since $\tilde\nu_{\beta,N,\varepsilon}$ is the invariant measure of \eqref{def: lattice dynamical phi4}, $\tilde\Phi^\varepsilon(t) \sim \tilde\nu_{\beta,N,\varepsilon}$ for every $t \in \RR_+$. 

Denote by $\Phi$ the (global-in-time) solution to \eqref{def: dynamical phi4} with initial data $\phi$. For every $t>0$, $\rext^\varepsilon \tilde\Phi^\varepsilon(t) \rightarrow \Phi(t)$ in distribution as $\varepsilon \rightarrow 0$ as a consequence of Proposition \ref{prop: lattice convergence}. Hence, $\Phi(t) \sim \nu$ for every $t > 0$. Thus, $\nu$ is an invariant measure of \eqref{def: dynamical phi4}. By Proposition \ref{prop: strong feller} the invariant measure of \eqref{def: dynamical phi4} is unique. Therefore, $\nu = \nubn$.  
\end{proof}

\section{Decay of spectral gap}\label{sec: decay of gap}

\begin{proof}[Proof of Corollary \ref{cor: sg}]

The Markov semigroup $(\cP_t^{\beta,N})_{t \geq 0}$ associated to \eqref{def: dynamical phi4} is reversible with respect to $\nubn$ (see \cite[Corollary 1.3]{HMb18} or \cite[Lemma 4.2]{ZZb18}). Thus, one can express $\lambda_{\beta,N}$ as the sharpest constant in the Poincar\'e inequality
	\begin{equs} \label{eq:poincareinequal}
	\lambda_{\beta,N} 
	=
	\inf_{F \in D(\cE_{\beta,N})} \frac{\cE_{\beta,N}(F,F)}{\langle F^2 \rangle_{\beta,N} - \langle F \rangle_{\beta,N}^2}	
	>
	0
	\end{equs}
	where $\cE_{\beta,N}$ is the associated Dirichlet form with domain $D(\cE_{\beta,N}) \subset L^2(\nubn)$. See \cite[Corollary 1.5]{ZZb18}.
	
The proof of Corollary \ref{cor: sg} amounts to choosing the right test function in \eqref{eq:poincareinequal} and then using the explicit expression for $\cE_{\beta,N}$ for sufficiently nice functions due to \cite[Theorem 1.2]{ZZb18}.

Let $\mathrm{Cyl}$ be the set of $F \in L^2(\nubn)$ of the form
\begin{equs}
F(\cdot)
=
f\Big( l_1(\cdot),\dots,l_m(\cdot) \Big)
\end{equs}
where $m \in \NN$, $f \in C^1_b(\RR^m)$, $l_1,\dots,l_m$ are real trigonometric polynomials, and $l_i(\cdot)$ denotes the ($L^2$) duality pairing between $l_i$ and elements in $\cC^{-\frac 12-\kappa}$. For any $F \in \mathrm{Cyl}$, let $\partial_{l_i} F$ denote the G\^ateaux derivative of $F$ in direction $l_i$. Let $\nabla F: \cC^{-\frac 12 -\kappa} \rightarrow \RR$ be the unique function such that $\partial_{l_i} F(\phi) = \intx \nabla F(\phi) l_i dx$ for every $\phi \in \cC^{-\frac 12 - \kappa}$. In other words, $\nabla F$ is the representation of the G\^ateaux derivative with respect to the $L^2$ inner product. Then, for any $F,G \in \mathrm{Cyl}$,
\begin{equs}
	\cE_{\beta,N}(F,G)
	=
	\Big\langle \intx \nabla F \nabla G dx  \Big\rangle_{\beta,N}.
\end{equs}

Now we choose a test function in ${\rm Cyl}$ to insert into \eqref{eq:poincareinequal}. Take any $\zeta \in (0,1)$ and $m \in [0, (1-\zeta)\hfbeta)$. Let $\chi_m:\RR \rightarrow \RR$ be a smooth, non-decreasing odd function such that $\chi_m(a) = -1$ for $a \leq -m$ and $\chi_m(a) = 1$ for $a \geq m$. Define
	\begin{equs}
	F(\phi) = \chi_m(\fm_N(\phi)).	
	\end{equs}
	Then, $F \in \mathrm{Cyl}$ and $\langle F \rangle_{\beta,N}= 0$. Moreover, its Fr\'echet derivative $DF$ is supported on the set $\{ \fm_N \in [-m,m] \}$.

Thus, inserting $F$ into \eqref{eq:poincareinequal}, we obtain
\begin{equs} \label{testfunction}
\lambda_{\beta,N} 
\leq 
\frac{\cE_{\beta,N}(F,F)}{\langle F^2 \rangle_{\beta,N}} 
\leq 
\frac{\Big\| \intx |\nabla F|^2 dx \Big\|_{L^\infty(\nu_{\beta,N})}}{\langle F^2 \rangle_{\beta,N}} \nubn(\fm_N \in [-m,m]).	
\end{equs}

For any $g \in L^2(\TTN)$ and $\varepsilon > 0$, by the linearity of $\fm_N$ and the Cauchy-Schwarz inequality,
\begin{equs}
\frac{F(\phi+\varepsilon g)-F(\phi)}{\varepsilon}
&\leq
|\chi'_m|_\infty \Big|\frac{\fm_N(\phi + \varepsilon g) - \fm_N(\phi)}{\varepsilon} \Big|
\\
&\leq
|\chi'_m|_\infty \frac{\intx g dx}{N^3} 
\leq
|\chi'_m|_\infty \frac{ \Big(\intx g^2 dx\Big)^\frac 12}{N^\frac 32}
\end{equs}
where $\chi_m'$ is the derivative of $\chi_m$ and $|\cdot|_\infty$ denotes the supremum norm. Note that this estimate is uniform over $\phi \in \cC^{-\frac 12 -}$. Then, by duality and the definition of $\nabla F$,
\begin{equs} \label{testfn2}
\begin{split}
\Bigg\| \intx |\nabla F|^2 dx \Bigg\|_{L^\infty(\nubn)}
&=
\Bigg\| \Big( \sup_{g \in L^2: \intx g^2 dx = 1} \intx \nabla F g dx \Big)^2 \Bigg\|_{L^\infty(\nubn)}
\\
&\leq
\frac{|\chi'_m|_\infty^2}{N^3}.
\end{split}
\end{equs}

For the other term in \eqref{testfunction}, using that $F^2$ is identically $1$ on $\{ |\fm_N| \geq m \}$,
\begin{equs} \label{testfn3}
\begin{split}
	\langle F^2 \rangle_{\beta,N}
	&=
	\nubn(|\fm_N| \geq m) + \langle F^2 \1_{\fm_N \in (-m,m)} \rangle_{\beta,N}
	\\
	&\geq
	1-\nubn(\fm_N \in (-m,m)).
\end{split}
\end{equs}

We insert \eqref{testfn2} and \eqref{testfn3} into \eqref{testfunction} to give
\begin{equs}
	\lambda_{\beta,N} 
	\leq 
	\frac{|\chi_m|_\infty^2}{N^3} \frac{\nu_{\beta,N}(\fm_N \in [-m,m])}{1 - \nu_{\beta,N}( \fm_N \in (-m,m))}.
\end{equs}
By Theorem \ref{thm: ld}, there exists $C=C(\zeta,\eta)>0$ and $\beta_0 = \beta_0(\zeta,\eta) > 0$ such that, for all $\beta > \beta_0$,
	\begin{equs}
	\lambda_{\beta,N} 
	\leq 
	\frac{|\chi_m'|_\infty^2}{N^3} \frac{e^{-C\hfbeta N^2}}{1-e^{-C\hfbeta N^2}}
	\end{equs}
from which \eqref{eq: cor} follows.

\end{proof}

\appendix

\section{Analytic notation and toolbox} \label{appendix: toolbox}

\subsection{Basic function spaces on the torus} \label{appendix: sub: basic}

Let $\TTN = (\RR / N\ZZ)^3$ be the 3D torus of sidelength $N \in \NN$. Denote by $C^\infty(\TTN)$ the space of smooth functions on $\TTN$ and by $S'(\TTN)$ the space of distributions. For $\phi \in S'(\TTN)$ and $f \in C^\infty(\TTN)$, we write $\intx \phi f dx$ to denote their duality pairing. For any $p \in [1,\infty]$, let $L^p(\TTN)=L^p(\TTN,\dbar x)$ denote the Lebesgue space with respect to the normalised Lebesgue measure $\dbar x = \frac{dx}{N^3}$.

Let $\cF$ denote the Fourier transform, i.e. for any $f \in C^\infty(\TTN)$ and $n \in (N^{-1}\ZZ)^3$, 
\begin{equs}
\cF f(n)
=
\intx f e_{-n} d x, 
\quad 
f 
= 
\frac 1 {N^3} \sum_{n \in (N^{-1}\ZZ)^3} \cF f(n) e_n
\end{equs}
where $e_n(x) = e^{2\pi i n \cdot x }$. 

For any $\rho : \RR^3 \rightarrow \RR$, let $T_\rho$ be the Fourier multiplier with symbol $\rho(\cdot)$ defined on smooth functions via
\begin{equs}
T_\rho f
=
\frac {1}{N^3} \sum_{n \in (N^{-1}\ZZ)^3} \rho(n) \cF f(n) e_n.	
\end{equs}
When clear from context, we simply write $\rho f$ instead of $T_\rho f$. 

For $s \in \RR$, the inhomogeneous Sobolev space $H^s$ is the completion of $f \in C^\infty$ with respect to the norm
\begin{equs}
\| f \|_{H^s}
=
\| \langle \cdot \rangle^s f \|_{L^2}
\end{equs}
where $\langle \cdot \rangle = \sqrt{ \eta + 4\pi^2 | \cdot |^2}$ for a fixed $\eta > 0$ (see Section \ref{sec: model}). The norms depend on $\eta$ but they are equivalent for different choices.

\subsection{Besov spaces} \label{appendix: sub: besov}

In this section, we introduce Besov spaces on $\TTN$ and give some useful estimates. All of the results can be found in \cite[Section 2.7]{BCD11} stated for Besov spaces on $\RR^3$, but can be adapted to $\TTN$.

Let $B(x,r)$ denote the ball centred at $x \in \RR^3$ of radius $r > 0$ and let $A$ denote the annulus $B(0, \frac 43) \setminus B(0, \frac 38)$. Let $\tilde \Delta, \Delta \in C^\infty_c(\RR^3;[0,1])$ be radially symmetric and satisfy 
\begin{itemize}
\item $\mathrm{supp} \tilde \chi \subset B(0, \frac 43)$ and $\mathrm{supp} \chi \subset A$; 
\item $\sum_{k \geq -1} \chi_k = 1$, where $\chi_{-1} = \tilde \chi$ and $\chi_k(\cdot) = \chi(2^{-k}\cdot)$ for $k \in \NN \cup \{0\}$. 
\end{itemize}
Identify $\Delta_k$ with its Fourier multiplier. 

$\{ \Delta_k \}_{k \in \NN \cup \{-1\}}$ are called Littlewood-Paley projectors. For $f \in C^\infty(\TTN)$, we have
\begin{equs}
f
=
\sum_{k \geq -1} \Delta_k f. 	
\end{equs}
For $k \geq 0$, $\Delta_k f$ contains the frequencies of $f$ order $2^k$. $\Delta_{-1}$ contains all the low frequencies (i.e. of size less than order 1).	

For $s \in \RR$, $p,q \in [1,\infty]$, we define the Besov spaces $B^s_{p,q}(\TTN)$ to be the completion of $C^\infty(\TTN)$ with respect to the norm
\begin{equs}
    \| f \|_{B^s_{p,q}} 
    = 
    \Big\| \Big( 2^{ks} \|\Delta_k f \|_{L^p} \Big)_{k \geq -1} \Big\|_{l^q}
\end{equs}
where $l^q$ is the usual space of $q$-summable sequences, interpreted as a supremum when $q=\infty$. Note that these spaces are separable. Besov-H\"older spaces are denoted $B^s_{\infty,\infty}(\TTN) = \cC^{s}(\TTN)$ and are a strict subset of the usual H\"older spaces (which are not separable) for $s \in \RR_+ \setminus \NN$. Moreover, the $B^s_{2,2}(\TTN) = H^s(\TTN)$ and their norms are equivalent.

\begin{prop}[Duality]\label{prop: tool duality}
	Let $s \in \RR$ and $p_1,p_2,q_1,q_2 \in [1,\infty]$ such that $\frac 1{p_1} + \frac 1{p_2} = \frac 1{q_1} + \frac 1{q_2} = 1$. Then,
	\begin{equs} \label{tool: duality}
	\Big| \intx f g \dbar x \Big|
	\leq
	\| f \|_{B^{-s}_{p_1,q_1}} \|g \|_{B^s_{p_2,q_2}}	
	\end{equs}
	for $f,g \in C^\infty(\TTN)$.
\end{prop}

\begin{proof}
See \cite[Lemma 2.1]{GOTW18}.
\end{proof}

\begin{prop}[Fractional Leibniz estimate]\label{prop: tool Leibniz}
	Let $s \in \RR$, $p,p_1,p_2,p_3,p_4,q \in [1,\infty]$ satisfy $\frac 1p = \frac 1{p_1} + \frac 1{p_2} = \frac 1{p_3} + \frac 1{p_4}$. Then, there exists $C=C(s,p_1,p_2,p_3,p_4,q,\eta) > 0$ such that
	\begin{equs}\label{tool: Leibniz}
	\| fg \|_{B^s_{p,q}}
	\leq
	C\| f \|_{B^s_{p_1,q}} \| g \|_{L^{p_2}} + \| f \|_{L^{p_3}} \|g \|_{B^s_{p_4,q}}	
	\end{equs}
	for $f,g \in C^\infty(\TTN)$.
\end{prop}

\begin{proof}
See \cite[Lemma 2.1]{GOTW18}.
\end{proof}

\begin{prop}[Interpolation]\label{prop: tool interp}
	Let $s,s_1,s_2 \in \RR$ such that $s_1 < s < s_2$, $p,p_1,p_2,q,q_1,q_2 \in [1,\infty]$ and $\theta \in (0,1)$ satisfy
	\begin{equs} \label{tool: interp}
	s
	&=
	\theta s_1 + (1-\theta)s_2
	\\
	\frac 1p 
	&= 
	\frac \theta {p_1} + \frac{1-\theta}{p_2}
	\\
	\frac 1q
	&=
	\frac \theta {q_1} + \frac{1-\theta}{q_2}.
	\end{equs}
Then, there exists $C=C(s,s_1,s_2,p,p_1,p_2,q,q_1,q_2,\theta,\eta) > 0$ such that
\begin{equs}\label{tool: interpolation}
	\| f \|_{B^s_{p,q}}
	\leq 
	C\| f \|_{B^{s_1}_{p_1,q_1}}^\theta \| f \|_{B^{s_2}_{p_2,q_2}}^{1-\theta}
\end{equs}
for $f \in C^\infty(\TTN)$.
\end{prop}

\begin{proof}
See \cite[Proposition 5.7]{BM18}.
\end{proof}

\begin{prop}[Bernstein's inequality] \label{prop: tool bernstein}
For $R > 0$, denote $B_f(R) = \{ n \in (N^{-1}\ZZ)^3 : |n| \leq R \}$. Let $s_1, s_2 \in \RR$ such that $s_1 < s_2$, $p,q \in [1,\infty]$.
Then, there exists $C=C(s_2,s_2,p,q,\eta)>0$ such that
\begin{equs}
\begin{split} \label{tool: bernstein ball}
	\| f \|_{B^{s_2}_{p,q}}
	&\leq
	C R^{s_2 - s_1} \| f \|_{B^{s_1}_{p,q}}
\end{split}
\\
\begin{split} \label{tool: bernstein annulus}
	\| g \|_{B^{s_1}_{p,q}}
	&\leq
	CR^{s_1 - s_2}\| g \|_{B^{s_2}_{p,q}}
\end{split}
\end{equs}
for $f,g \in C^\infty(\TTN)$ such that $\mathrm{supp} (\cF f) \subset B_f(R)$ and $\mathrm{supp} (\cF g) \subset (N^{-1}\ZZ)^3 \setminus B_f(R)$. 
\end{prop}

\begin{proof}
See \cite[Lemma 2.1]{BCD11}	for a proof on $\RR^3$.
\end{proof}

\subsection{Paracontrolled calculus} \label{appendix: sub: paracontrolled}

Let $f,g \in C^\infty(\TTN)$. Define the paraproduct
\begin{equs}
f \pg g
=
\sum_{l < k-1} \Delta_k f \Delta_l g	
\end{equs}
and the resonant product
\begin{equs}
f \pe g
=
\sum_{|k-l| \leq 1} \Delta_k f \Delta_l g.	
\end{equs}
Then,
\begin{equs}\label{tool: bony decomp}
fg
=
f \pl g + f \pe g + f \pg g.	
\end{equs}

\begin{prop}[Paraproduct estimates] \label{prop:paraproduct}
Let $s \in \RR$ and $p,p_1,p_2,q \in [1,\infty]$ be such that $\frac 1p = \frac 1 {p_1} + \frac 1{p_2}$. Then, there exists $C=C(s,p,p_1,p_2,q,\eta)>0$ such that
\begin{equs} \label{tool: paraproduct}
	\|f \pg g\|_{B^{s}_{p,q}}
	\leq
	C\| f \|_{B^s_{p_1,q}} \| g \|_{L^{p_2}}
\end{equs}
for $f,g \in C^\infty(\TTN)$.
\end{prop}

\begin{proof}
See \cite[Theorem 2.82]{BCD11} for a proof on $\RR^3$.
\end{proof}

\begin{prop}[Resonant product estimate]\label{prop: resonant}
Let $s_1,s_2 \in \RR$ such that $s=s_1 + s_2 > 0$. Let $p,p_1,p_2,q \in [1,\infty]$ satisfy $\frac 1p = \frac 1{p_1} + \frac 1{p_2}$. Then, there exists $C=C(s_1,s_2,p,p_1,p_2,q,\eta) > 0$ such that
\begin{equs} \label{tool: resonant}
	\| f \pe g \|_{B^s_{p,q}}
	\leq
	C\| f \|_{B^{s_1}_{p_1,\infty}} \| g \|_{B^{s_2}_{p_2,q}}	
\end{equs}
for $f,g \in C^\infty(\TTN)$.	
\end{prop}

\begin{proof}
See \cite[Theorem 2.85]{BCD11} for a proof on $\RR^3$.	
\end{proof}

We now state some useful commutator estimates.
\begin{prop} \label{prop:commutator1}
Let $s_1, s_3 \in \RR$, $s_2 \in (0,1)$ such that $s_1 + s_3 < 0$ and $s_1 + s_2 + s_3 = 0$. Moreover, let $p,p_1,p_2,q_1,q_2 \in [1,\infty]$ satisfy $\frac 1p + \frac 1{p_1} + \frac 1{p_2} = 1$ and $\frac 1{q_1} + \frac 1{q_2} = 1$. Then, there exists $C=C(s_1,s_2,s_3,p,p_1,p_2,q_1,q_2,\eta)>0$ such that
\begin{equs} \label{tool: commutator 1}
\Big| \intx (f \pg g)h - (f \pe h)g \dbar x \Big|
\leq
C\| f \|_{B^{s_1}_{p,\infty}} \|g\|_{B^{s_2}_{p_1,q_1}} \| h \|_{B^{s_3}_{p_2,q_2}}
\end{equs}
for $f,g,h \in C^\infty(\TTN)$.
\end{prop}

\begin{proof}
This is a modification of \cite[Lemma A.6]{GUZ20}. See \cite[Proposition 7]{BG19}.	
\end{proof}

\begin{prop}\label{prop:commutator2}
Let $s_1,s_3 \in \RR$, $s_2 \in (0,1)$ such that $s_1+s_3 < 0$ but $s_1+s_2+s_3 > 0$. Morover, let $p,p_1,p_2,p_3 \in [1,\infty]$ satisfy $\frac 1p = \frac 1{p_1} + \frac 1{p_2} + \frac 1{p_3}$. Then, there exists $C=C(s_1,s_2,s_3,p,p_1,p_2,\eta)>0$ such that
\begin{equs} \label{tool: commutator 2}
\| (f\pg g)\pe h - (f\pe h)g \|_{B^{s_1+s_2+s_3}_{p,\infty}}
&\leq
C\| f \|_{B^{s_1}_{p_1,\infty}} \|g \|_{B^{s_2}_{p_2,\infty}} \| h \|_{B^{s_3}_{p_3,\infty}} 
\end{equs}
for $f,g,h \in C^\infty(\TTN)$.
\end{prop}

\begin{proof}
This is a modification of \cite[Lemma 2.4]{GIP15}. See \cite[Proposition 6]{BG19}.	
\end{proof}

\subsection{Analytic properties of $\cJ_k$} \label{subsec: analytic cj}

The family of operators $\{ \cJ_k \}_{k \geq 0}$ defined in Section \ref{subsec: construction of stochastic objects} satisfies the following estimate: for every multi-index $\alpha \in \NN^3$, there exists $C=C(\alpha,\eta)$ such that
\begin{equs} \label{eq: multiplier bound}
	\Big| \partial^\alpha \cJ_k (x)|
	\leq
	\frac{C}{\langle k \rangle^\frac 12 (1+|x|)^{1+|\alpha|}}.
\end{equs}

\begin{prop}\label{prop: multiplier estimate}
	Let $s \in \RR$, $p,q\in [1,\infty]$. Then, there exists $C=C(s,p,q)>0$ such that
	\begin{equs} \label{tool: multiplier estimate}
	\| \cJ_k f \|_{B^{s+1}_{p,q}}
	\leq
	\frac{C}{\langle k \rangle^\frac 12}\| f \|_{B^s_{p,q}}	
	\end{equs}
	for every $f \in C^\infty(\TTN)$
\end{prop}

\begin{proof}
This follows from \eqref{eq: multiplier bound} and \cite[Proposition 2.78]{BCD11}.
\end{proof}

We now state another useful commutator estimate.
\begin{prop} \label{prop: multiplier commutator}
	Let $s_1 \in \RR$, $s_2 \in  (0,1)$, $p,p_1,p_2,q,q_1,q_2 \in [1,\infty]$ such that $\frac 1p = \frac 1{p_1} + \frac 1{p_2}$ and $\frac 1q = \frac 1{q_1} + \frac 1{q_2}$. Then, for any $\kappa > 0$, there exists $C=C(s_1,s_2,p,p_1,p_2, q, \kappa, \eta)>0$ such that
	\begin{equs} \label{tool: multiplier commutator}
		\| \cJ_k( f \pg g) - \cJ_k f \pg g \|_{B^{s_1 + s_2 -\kappa}_{p,q}}
		\leq
		C\| f \|_{B^{s_1}_{p_1,\infty}}\| g \|_{B^{s_2}_{p_2,\infty}}
	\end{equs}
	for $f,g \in C^\infty(\TTN)$.
\end{prop}

\begin{proof}
This follows from \eqref{eq: multiplier bound} and \cite[Lemma 2.99]{BCD11}.
\end{proof}

\subsection{Poincar\'e inequality on blocks}

\begin{prop}\label{prop:poincare_blocks}
	There exists $C_P > 0$ such that, for any $N \in \NN$ and $\Box \subset \TTN$ a unit block, the following estimate holds for all $f \in C^\infty(\TTN)$:
	\begin{equs} \label{tool: poincare}
	\int_\sBox \big(f - f(\Box)\big)^2 dx
	\leq
	C_P \int_\sBox |\nabla f|^2 dx
	\end{equs}
	where $f(\Box) = \int_\sBox f dx$.
\end{prop}

\begin{proof}
See \cite[(7.45)]{GT15}.
\end{proof}

\subsection{Bounds on discrete convolutions}

\begin{lem}\label{lem: discrete convolution bound}
Let $d \geq 1$ and $\alpha, \beta \in \RR$ satisfy 
\begin{equs}
\alpha + \beta > d \text{ and } \alpha, \beta < d.	
\end{equs}
Then, there exists $C=C(d,\alpha,\beta)>0$ such that, uniformly over $n \in (N^{-1}\ZZ)^d$,
\begin{equs}
\frac{1}{N^3} \sum_{\substack{n_1,n_2 \in (N^{-1}\ZZ)^d \\ n_1 + n_2 = n}} \frac{1}{\langle n_1 \rangle^\alpha \langle n_2 \rangle^\beta}
\lesssim
\frac{1}{\langle n \rangle^{\alpha + \beta - d}} 	
\end{equs}
\end{lem}

\begin{proof}
Follows from \cite[Lemma 4.1]{MWX17} and by keeping track of $N$ dependence.
\end{proof}

\endappendix

\bibliographystyle{Martin}
\bibliography{biblio.bib}

\end{document}